\documentclass[a4paper,11pt,openright, twosides]{amsart}

\usepackage[english]{babel}
\usepackage{amssymb,amsmath,amsfonts}
\usepackage[foot]{amsaddr}
\SetSymbolFont{letters}{bold}{OML}{cmm}{b}{it}
\SetSymbolFont{operators}{bold}{OT1}{cmr}{bx}{n}
\usepackage{graphicx}
\usepackage{color}
\usepackage{tikz}
\usepackage{natbib}
\usepackage{mathrsfs} 		% police M phase space
\usepackage{mathabx}		% wideaccent
\usepackage[normalem]{ulem}	% souligner
\usepackage{cancel}
\usepackage{geometry}
\newcounter{Ctheo}
\newcounter{Ccoro}
\newcounter{Clem}
\newcounter{Cprop}
\newtheorem{theorem}[Ctheo]{Theorem}
\newtheorem{Corollary}[Ccoro]{Corollary}
\newtheorem{lemma}[Clem]{Lemma}
\newtheorem{remark}{Remark}
\newtheorem{proposition}[Cprop]{Proposition}

\numberwithin{equation}{section}
\numberwithin{Clem}{section}
\numberwithin{Ctheo}{section}
\numberwithin{Cprop}{section}
\numberwithin{Ccoro}{section}
\numberwithin{remark}{section}

\newcommand{\nnb}{\nonumber}
\newcommand{\be}{\begin{equation}}
\newcommand{\ee}{\end{equation}}
\newcommand{\bes}{\begin{equation*}}
\newcommand{\ees}{\end{equation*}}
\newcommand{\bpm}{\begin{pmatrix}}
\newcommand{\epm}{\end{pmatrix}}
\newcommand{\qtext}[1]{\quad \mbox{ #1 } \quad}

\newcommand{\comment}[1]{}

\newcommand{\dron}[2]  {\frac{\partial#1   }{\partial#2}}

\newcommand{\Poi }[2]{\left\{#1,#2\right\}}

\newcommand{\norm}[1]{\left\Vert#1 \right\Vert}

\newcommand{\modu}[1]{\left\vert#1 \right\vert}

\newcommand\cA{{\mathcal A}}

\newcommand\cAt{{\tilde{\cA}}}
\newcommand\cAb{{\overline{ \cA}}}
\newcommand\At{{\tilde{A}}}

\newcommand\cB{{\mathcal B}}

\newcommand\cBb{{\overline{\cB}}}
\newcommand\cBt{{\tilde{\cB}}}

\newcommand{\CC}{{\mathbb C}}
\newcommand\sC{{\mathscr C}}
\newcommand\cC{{\mathcal C}}
\newcommand\cCt{{\tilde{\cC}}}
\newcommand\cCb{{\overline{\cC}}}

\newcommand\rC{{\mathrm C}}

\newcommand\cD{{\mathcal D} }
\newcommand\fD{{\mathfrak D} }

\newcommand\sD{{\mathscr D}}

\newcommand\fDh{{\hat{\fD}}}
\newcommand\Dh{{\hat{D}}}
\newcommand\sE{{\mathscr E} }

\newcommand\Ent{{\mathrm E }}

\newcommand\cF  {{\mathcal F }}
\newcommand\fF{{\mathfrak F}}
\newcommand\rF {{\mathrm F }}

\newcommand\sF  {{\mathscr F }}
\newcommand\fFh{{\hat{\fF}}}

\newcommand\cG  {{\mathcal G }}

\newcommand\gt   {{\tilde{g} }}

\newcommand\fG{{\mathfrak G}}

\newcommand\fGh{{\hat{\fG}}}

\newcommand\Hb   {\overline{H} }
\newcommand\Hd   {H^\dagger }
\newcommand\Hbast   {\overline{H}_* }

\newcommand\Ht   {{\tilde{H} }}
\newcommand\cH  {{\mathcal H }}
\newcommand\sH  {{\mathscr H }}
\newcommand\sHt  {\tilde{\mathscr H }}
\newcommand\sHh  {\hat{\mathscr H }}
\newcommand\sHb  {\overline{\mathscr H }}
\newcommand\sHbast  {\overline{\mathscr H_*} }
\newcommand\sHd  {\mathscr{H}^\dagger }

\newcommand\rH  {{\mathrm H }}

\newcommand\sHc  {\check{\mathscr H }}
\newcommand\hh  {\hat{h}}
\newcommand\rh {{\mathrm h }}

\newcommand\rhh {{\hat{\rh} }}

\newcommand\cI  {{\mathcal I }}
\newcommand\bI  {{\bf{I} }}

\newcommand\cIh  {{\hat {\cI} }}
\newcommand\sI {\mathscr{I}}

\newcommand\Ih { {\hat{I}} }

\newcommand\bJ  {{\bf{J} }}
\newcommand\Js   { \und{J{}    } }

\newcommand\bJs   { \und{\bf{J}{}    } }
\newcommand\Jh   { \hat{J    } }

\newcommand\cK  {{\mathcal K }}
\newcommand\fK  {{\mathfrak K }}
\newcommand\fKh  {{\hat{\fK} }}

\newcommand\cKh  {{\hat{\cK}}}

\newcommand\bk{\mathbf k}

\newcommand\bl{\boldsymbol{l}}
\newcommand\tL {\tilde L }

\newcommand\tM{\tilde M }
\newcommand\hm{\widehat m}

\newcommand{\NN}{{\mathbb N}}

\newcommand\rN  {{\mathrm N }}

\newcommand\cO{{\mathcal O} }
\newcommand\bp   {{\bf{p} }}
\newcommand\pt   {{\widetilde{p} }}
\newcommand\ph   {{\hat{p} }}
\newcommand\bpt   {{\widetilde{\bp} }}
\newcommand\rP  {{\mathrm P }}

\newcommand\sQt  {\tilde{\mathscr Q }}
\newcommand\sQb  {\overline{\mathscr Q }}
\newcommand\sQ  {{\mathscr Q}}

\newcommand\rpq  {{\mathrm q }}

\newcommand\rpqh{{\hat{\rpq}}}
\newcommand\Qt  {\tilde{ Q }}

\newcommand\br{{\bf r} }
\newcommand{\RR}{{\mathbb{R}}}
\newcommand\brt{\tilde{\bf r} }
\newcommand\sR  {{\mathscr R }}

\newcommand\sRt  {\tilde{\mathscr R }}
\newcommand\sRb  {\overline{\mathscr R }}
\newcommand\rR  {{\mathrm R }}

\newcommand\sRc{\check{\sR}}
\newcommand\rb{\bar r}

\newcommand\cS  {{\mathcal S }}
\newcommand\cSh  {{\hat \cS }}
\renewcommand\sb{\bar s}

\newcommand{\TT}{{\mathbb T}}

\newcommand\cU  {{\mathcal U }}

\newcommand\cUh  {{\hat{\cU} }}

\newcommand\cV{{\mathcal V} }

\newcommand\bw  {{\bf {w} }}
\newcommand\bwt {{\widetilde{\bw} }}
\newcommand\wt {{\widetilde{w} }}
\newcommand\ws   { \und{w{}    } }
\newcommand\wts   { \und{\wt{}    } }
\newcommand\bwts   { \und{\bwt{}    } }
\newcommand\bws   { \und{\bw{}    } }

\newcommand\bx  {{\bf{x} }}
\newcommand\bxt  {{\widetilde{\bf{x}} }}

\newcommand\xt   {{\widetilde{x} }}
\newcommand\xb   {{\overline{x} }}
\newcommand\xs   { \und{x{}} }
\newcommand\xts   { \und{\widetilde{x}{}} }
\newcommand\bxs   { \und{\bf{x}{}} }
\newcommand\bxts   { \und{\widetilde{\bf{x}{}}} }

\newcommand\by  {{\bf{y} }}

\newcommand\bZ {{\bf {Z} }}
\newcommand{\ZZ}{{\mathbb Z}}
\newcommand\Zs   { \und{Z{}    } }
\newcommand\bZs   { \und{\bf{Z}{}    } }
\newcommand\bz {{\bf {z} }}
\newcommand\bzt { {\widetilde{\bf z}} }
\newcommand\bzs {\und{\bf {z}{} }}
\newcommand\bzts { \und{\widetilde{\bf z}{}} }
\newcommand\zt { {\widetilde{z}} }
\newcommand\zs {\und{{z}{} }}
\newcommand\zts { \und{\widetilde{ z}{}} }

\newcommand\gam   {\gamma  }
\newcommand\Gam   {\Gamma  }
\newcommand\bGam{\boldsymbol{\Gam}}
\newcommand\bGams {\und{\bGam } }
\newcommand\Gams {\und{\Gam } }

\newcommand\Deltah{\hat{\Delta}}
\newcommand\deltas   { \und{\delta{}} }

\newcommand\eps   {{\varepsilon}}
\newcommand\bzeta{\boldsymbol{\zeta}}
\newcommand\zetas { \und{\zeta{}} }
\newcommand\bzetas { \und{\bzeta}{}}
\newcommand\btheta{\boldsymbol{\theta}}

\newcommand\kappat{{\tilde{\kappa}}}
\newcommand\lam   {\lambda }
\newcommand\Lam   {\Lambda }
\newcommand\brho{\boldsymbol{\rho}}

\newcommand\rhoh{\hat{\rho}}

\newcommand\rhot{\widetilde{\rho}}

\newcommand\bsig{\boldsymbol{\sigma}}
\newcommand\sig   {\sigma  }
\newcommand\sigh   {\hat{\sig }}

\newcommand\Ups{\Upsilon}
\newcommand\Upsb{\overline{\Upsilon}}
\newcommand\Upst{\tilde{\Upsilon}}

\newcommand\varphih{\hat{\varphi{}} }

\newcommand\bphis{\und{\bphi{}} }
\newcommand\phis{\und{\phi{}} }
\newcommand\bvphi{\boldsymbol{\varphi} }
\newcommand\bphi{\boldsymbol{\phi} }

\newcommand\bpsi{\boldsymbol{\psi}}
\newcommand\psis { \und{\psi{}} }
\newcommand\bpsis { \und{\bpsi}{}}
\newcommand\Psit{\tilde{\Psi}}
\newcommand\Psib{\overline{\Psi}}
\newcommand\Psic{\check{\Psi}}

\newcommand\bxi{\boldsymbol{\xi}}

\newcommand\om    {\omega }
\newcommand\Om    {\Omega }

\newcommand\bgomega{\boldsymbol{\Omega}}
\newcommand\bom{\boldsymbol{\omega}}
\newcommand\bOm{\boldsymbol{\Omega}}
\newcommand\bomt{\tilde{\bom}}
\newcommand\bomh{\hat{\bom}}

\makeatletter
\newcommand*\bigcdot{\mathpalette\bigcdot@{.6}}
\newcommand*\bigcdot@[2]{\mathbin{\vcenter{\hbox{\scalebox{#2}{$\m@th#1\bullet$}}}}}
\makeatother
\newcommand{\leqp}{{\,\leq \! \bigcdot \,}}
\newcommand{\pleq}{{\,\bigcdot \! \leq\,}}
\newcommand{\eqp}{{\, =\!\bigcdot\,}}
\newcommand{\peq}{{\, \bigcdot\! =\,}}
\newcommand{\md}[1]{{\text{\d{$#1$}}}}
\newcommand{\und}[1]{\md{#1}}
\DeclareMathOperator{\re}{Re}
\DeclareMathOperator{\im}{Im}
\DeclareMathOperator{\conj}{conj}

\DeclareMathOperator{\Id}{Id}

\newcommand\bzero  {{\bf{0} }}
\DeclareTextSymbol{\degre}{OT1}{23}
\newcommand{\rd}{\mathrm{d}}
\newcommand\phicrit{\varphi_{1,\delta}^{\min}}
\newcommand\phicritz{\varphi_{1,\deltas}^{\min}}

\begin{document}

%%% Abstract & Adresse %%%%%
%%%%%%%%%%%%%%%%%%%%%%%%%%%%%%
% COORBKAMPLUS
% TITLE, ABSTRACT & ADRESSES
% MAJ: 2018-01-31 ALEX
%%%%%%%%%%%%%%%%%%%%%%%%%%%%%%

%%% Titre %%%%%%%%%%%%%%%%%%%%%%%%
\title[Existence of quasiperiodic horseshoe-shaped orbits]{On the co-orbital motion in the three-body problem: existence of quasi-periodic horseshoe-shaped orbits}
%%%%%%%%%%%%%%%%%%%%%%%%%%%%%%

%%% Auteurs %%%%%%%%%%%%%%%%%%%%%%
\author{Laurent Niederman$^{1,3}$}
	\address[]{$^1$\textit{D\'epartement de Math\'ematiques d'Orsay}, Universit\'e Paris-Saclay, F-91405, Orsay, France. %\textit{supported by the ANR project BEKAM (ANR-15-CE40-0001)}
	}
\author{Alexandre Pousse$^2$}
\address[]{$^2$\textit{Dipartimento di Matematica ``Tullio Levi-Civita"}, Universit\`a degli Studi di  Padova, via Trieste, 63, 35131 Padova, Italia.%\textit{supported by the H2020-ERC project 677793 StableChaoticPlanetM}
}
\author{Philippe Robutel$^3$}
%	\address[]{$^3$\textit{Astronomie et Syst\`emes Dynamiques}, IMCCE, Observatoire de Paris, PSL University}
\address[]{$^3$\textit{ASD}-IMCCE, CNRS-UMR8028, Observatoire de Paris, PSL University, Sorbonne Universit\'e, 77 Avenue Denfert-Rochereau, 75014 Paris, France.}

%%%%%%%%%%%%%%%%%%%%%%%%%%%%%%%

%%% Date %%%%%%%%%%%%%%%%%%%%%%%%%
\date{\today}

%%%%%%%%%%%%%%%%%%%%%%%%%%%%%%%

%%%
%\maketitle
%%%

%%% Abstract %%%%%%%%%%%%%%%%%%%%%%
\begin{abstract}

Janus and Epimetheus are two moons of Saturn with very peculiar motions.
As they orbit around Saturn on quasi-coplanar and quasi-circular trajectories whose radii are only 50 km apart (less than their respective diameters), every four (terrestrial) years the bodies approach each other and their mutual gravitational influence lead to a swapping of the orbits: the outer moon becomes the inner one and vice-versa.  This behavior generates horseshoe-shaped trajectories depicted in an appropriate  rotating frame.
In spite of analytical theories and numerical investigations developed to describe their long-term dynamics, so far very few rigorous long-time stability results on the ``horseshoe motion" have been obtained even in the restricted three-body problem.
Adapting the idea of \cite{1963Ar} to a resonant case (the co-orbital motion is associated with trajectories in 1:1 mean motion resonance), we provide a rigorous proof of existence of 2-dimensional elliptic invariant tori on which the trajectories are similar to those followed by Janus and Epimetheus. For this purpose, we apply KAM theory to the planar three-body problem.
\end{abstract}
%%%%%%%%%%%%%%%%%%%%%%%%%%%%%%%

%%%
\maketitle
%%%

\begin{flushright}
\textit{In memoriam Pascal Norbelly}
\end{flushright}
%%% MAJ 2018-04-17

%%% Table des MatiÃšres %%%%%
%\tableofcontents
%\newpage
%%% MAJ 2018-04-17

%%% Intro %%%%%%%%%%%%%%
%%%%%%%%%%%%%%%%%%%%%%%%%%%%%%
% COORBKAMPLUS
% INTRO
% MAJ: 2018-01-30 ALEX
%%%%%%%%%%%%%%%%%%%%%%%%%%%%%%

\section{Introduction}
\label{sec:intro}
%%%%%%%% Coorbital motion, Lagrange & Trojans %%%
In the framework of the planetary three-body problem (two bodies orbiting a more massive one), the co-orbital motion is associated with trajectories in 1:1 mean-motion resonance. In other words, the planets share the same orbital period.
This problem possesses a very rich dynamics which is related to the five famous ``Lagrange" configurations%%
%%%
\footnote{For two of these configurations the three bodies are located at the vertices of an equilateral triangle.
	These equilibria correspond to the fixed points $L_4$ and $L_5$ in the restricted three-body problem (RTBP).
	The other three are the Euler collinear configurations  ($L_1$, $L_2$, and $L_3$ in the RTBP).}.
%%%
This resonance  has been extensively studied since the discovery of Jupiter's ``Trojan" asteroids whose trajectories librate around one of the $L_4$ and $L_5$ equilibria with respect to the Sun and the planet.
%%%%%%%%%%%%%%%%%%%%%%%%%%%%%%%%%%
%
%
%%%% Saturn co-orbitaux: description TP & HS %%%%%
Since then, other co-orbital objects have been discovered in the Solar System and particularly in the system of Saturn's satellites which presently holds five pairs of co-orbital moons:
Calypso and Telesto, which are co-orbital with Tethys, Helene and Polydeuces, co-orbital with Dione, and the pair Janus-Epimetheus.
%%%%%%%%%%%%%%%%%%%%%%%%%%%%%%%%%%

%%%%%%%%%%%%%%%%%%%%
\begin{figure}[ht]
\begin{center}
\def\svgwidth{1\textwidth}
\begingroup%
  \makeatletter%
  \providecommand\color[2][]{%
    \errmessage{(Inkscape) Color is used for the text in Inkscape, but the package 'color.sty' is not loaded}%
    \renewcommand\color[2][]{}%
  }%
  \providecommand\transparent[1]{%
    \errmessage{(Inkscape) Transparency is used (non-zero) for the text in Inkscape, but the package 'transparent.sty' is not loaded}%
    \renewcommand\transparent[1]{}%
  }%
  \providecommand\rotatebox[2]{#2}%
  \ifx\svgwidth\undefined%
    \setlength{\unitlength}{1692.17797852bp}%
    \ifx\svgscale\undefined%
      \relax%
    \else%
      \setlength{\unitlength}{\unitlength * \real{\svgscale}}%
    \fi%
  \else%
    \setlength{\unitlength}{\svgwidth}%
  \fi%
  \global\let\svgwidth\undefined%
  \global\let\svgscale\undefined%
  \makeatother%
  \begin{picture}(1,0.48541895)%
    \put(0,0){\includegraphics[width=\unitlength]{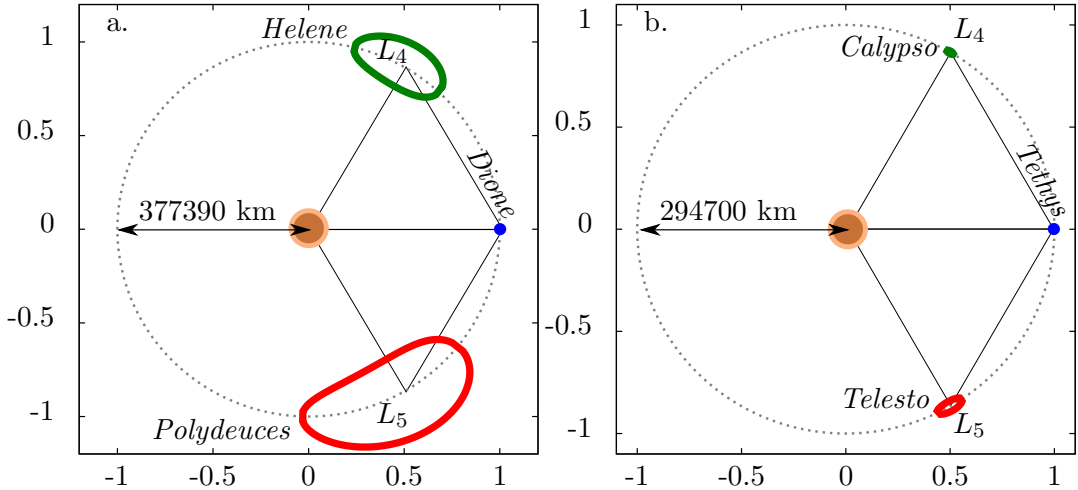}}%
    \put(0.35308311,0.40519493){\color[rgb]{0,0,0}\makebox(0,0)[rb]{\smash{$L_4$}}}%
    \put(0.35357066,0.07260491){\color[rgb]{0,0,0}\makebox(0,0)[rb]{\smash{$L_5$}}}%
    \put(0.58108725,0.25824698){\color[rgb]{0,0,0}\makebox(0,0)[lb]{\smash{294700 km}}}%
    \put(0.84738321,0.42463378){\color[rgb]{0,0,0}\makebox(0,0)[lb]{\smash{$L_4$}}}%
    \put(0.84738321,0.06341295){\color[rgb]{0,0,0}\makebox(0,0)[lb]{\smash{$L_5$}}}%
    \put(0.0791045,0.43100263){\color[rgb]{0,0,0}\makebox(0,0)[lb]{\smash{a.}}}%
    \put(0.56664187,0.43100241){\color[rgb]{0,0,0}\makebox(0,0)[lb]{\smash{b.}}}%
    \put(0.44069331,0.25696391){\color[rgb]{0,0,0}\rotatebox{-60}{\makebox(0,0)[rb]{\smash{\textit{Dione}}}}}%
    \put(0.94064082,0.25700819){\color[rgb]{0,0,0}\rotatebox{-60.000003}{\makebox(0,0)[rb]{\smash{\textit{Tethys}}}}}%
    \put(0.10856709,0.25820263){\color[rgb]{0,0,0}\makebox(0,0)[lb]{\smash{377390 km}}}%
    \put(0.29874073,0.42414128){\color[rgb]{0,0,0}\makebox(0,0)[rb]{\smash{\textit{Helene}}}}%
    \put(0.24763295,0.05758927){\color[rgb]{0,0,0}\makebox(0,0)[rb]{\smash{\textit{Polydeuces}}}}%
    \put(0.83470633,0.40794479){\color[rgb]{0,0,0}\makebox(0,0)[rb]{\smash{\textit{Calypso}}}}%
    \put(0.82703645,0.08581696){\color[rgb]{0,0,0}\makebox(0,0)[rb]{\smash{\textit{Telesto}}}}%
    \put(0.03228896,0.07539099){\color[rgb]{0,0,0}\makebox(0,0)[rb]{\smash{-1}}}%
    \put(0.03228896,0.16106462){\color[rgb]{0,0,0}\makebox(0,0)[rb]{\smash{-0.5}}}%
    \put(0.03228896,0.24678257){\color[rgb]{0,0,0}\makebox(0,0)[rb]{\smash{0}}}%
    \put(0.03228896,0.33245615){\color[rgb]{0,0,0}\makebox(0,0)[rb]{\smash{0.5}}}%
    \put(0.03228896,0.41812978){\color[rgb]{0,0,0}\makebox(0,0)[rb]{\smash{1}}}%
    \put(0.08913069,0.01558791){\color[rgb]{0,0,0}\makebox(0,0)[b]{\smash{-1}}}%
    \put(0.17586803,0.01558791){\color[rgb]{0,0,0}\makebox(0,0)[b]{\smash{-0.5}}}%
    \put(0.26264967,0.01558791){\color[rgb]{0,0,0}\makebox(0,0)[b]{\smash{0}}}%
    \put(0.34938701,0.01558791){\color[rgb]{0,0,0}\makebox(0,0)[b]{\smash{0.5}}}%
    \put(0.43612437,0.01558791){\color[rgb]{0,0,0}\makebox(0,0)[b]{\smash{1}}}%
    \put(0.51982634,0.05983407){\color[rgb]{0,0,0}\makebox(0,0)[rb]{\smash{-1}}}%
    \put(0.51982634,0.1533084){\color[rgb]{0,0,0}\makebox(0,0)[rb]{\smash{-0.5}}}%
    \put(0.51982634,0.24678257){\color[rgb]{0,0,0}\makebox(0,0)[rb]{\smash{0}}}%
    \put(0.51982634,0.34021248){\color[rgb]{0,0,0}\makebox(0,0)[rb]{\smash{0.5}}}%
    \put(0.51982634,0.4336867){\color[rgb]{0,0,0}\makebox(0,0)[rb]{\smash{1}}}%
    \put(0.56088959,0.01558791){\color[rgb]{0,0,0}\makebox(0,0)[b]{\smash{-1}}}%
    \put(0.65551616,0.01558791){\color[rgb]{0,0,0}\makebox(0,0)[b]{\smash{-0.5}}}%
    \put(0.75018704,0.01558791){\color[rgb]{0,0,0}\makebox(0,0)[b]{\smash{0}}}%
    \put(0.84481361,0.01558791){\color[rgb]{0,0,0}\makebox(0,0)[b]{\smash{0.5}}}%
    \put(0.93944018,0.01558791){\color[rgb]{0,0,0}\makebox(0,0)[b]{\smash{1}}}%
  \end{picture}%
\endgroup%
\end{center}
\caption{a,b) Schematic representation of the orbital motion of the co-orbital moons of Saturn.
The average values of the semi-major axes are rescaled to $1$ while the moons' radial excursions are exaggerated by a factor of $200$.
The trajectories of Helene, Polydeuces, Calypso, and Telesto are seen in  frames that rotate  respectively with Dione and Tethys.
Polydeuces' and Helene's trajectories (respectively Calypso's and Telesto's trajectories) describe a tadpole shape that surrounds the Lagrange points $L_4$ and $L_5$ with respect to the Sun and Dione (respectively Tethys).}
\label{fig:RTTadpole}
\end{figure}
%%%%%%%%%%%%%%%%%%%%%

%%%% Trajectoires TP %%%%%%%%%%%%%%%%%%%%
As displayed in Figure \ref{fig:RTTadpole}, the trajectories  of Calypso and Telesto  (resp. Helene and Polydeuces)  in the rotating reference frame with Tethys (resp. Dione) describe a tadpole shape which corresponds to a small deformation of the Lagrange equilateral configurations $L_4$ or $L_5$ with respect to Saturn and Tethys. This ``tadpole" motion, which is also characteristic of Jupiter's Trojans, has been extensively investigated in recent decades especially long-term stability of these asteroids \citep[see][]{MR980547,2006RoGa}.
%%%%%%%%%%%%%%%%%%%%%%%%%%%%%%%%%%

%%%%%%%%%%%%%%%%%%%%%%%%%%%%%%%%
\begin{figure}[ht]
\begin{center}
\def\svgwidth{0.78\textwidth}
%% Creator: Inkscape inkscape 0.92.2, www.inkscape.org
%% PDF/EPS/PS + LaTeX output extension by Johan Engelen, 2010
%% Accompanies image file 'JE0.eps' (pdf, eps, ps)
%%
%% To include the image in your LaTeX document, write
%%   \input{<filename>.pdf_tex}
%%  instead of
%%   \includegraphics{<filename>.pdf}
%% To scale the image, write
%%   \def\svgwidth{<desired width>}
%%   \input{<filename>.pdf_tex}
%%  instead of
%%   \includegraphics[width=<desired width>]{<filename>.pdf}
%%
%% Images with a different path to the parent latex file can
%% be accessed with the `import' package (which may need to be
%% installed) using
%%   \usepackage{import}
%% in the preamble, and then including the image with
%%   \import{<path to file>}{<filename>.pdf_tex}
%% Alternatively, one can specify
%%   \graphicspath{{<path to file>/}}
%%
%% For more information, please see info/svg-inkscape on CTAN:
%%   http://tug.ctan.org/tex-archive/info/svg-inkscape
%%
\begingroup%
  \makeatletter%
  \providecommand\color[2][]{%
    \errmessage{(Inkscape) Color is used for the text in Inkscape, but the package 'color.sty' is not loaded}%
    \renewcommand\color[2][]{}%
  }%
  \providecommand\transparent[1]{%
    \errmessage{(Inkscape) Transparency is used (non-zero) for the text in Inkscape, but the package 'transparent.sty' is not loaded}%
    \renewcommand\transparent[1]{}%
  }%
  \providecommand\rotatebox[2]{#2}%
  \ifx\svgwidth\undefined%
    \setlength{\unitlength}{894.88348389bp}%
    \ifx\svgscale\undefined%
      \relax%
    \else%
      \setlength{\unitlength}{\unitlength * \real{\svgscale}}%
    \fi%
  \else%
    \setlength{\unitlength}{\svgwidth}%
  \fi%
  \global\let\svgwidth\undefined%
  \global\let\svgscale\undefined%
  \makeatother%
  \begin{picture}(1,0.91476968)%
    \put(0,0){\includegraphics[width=\unitlength]{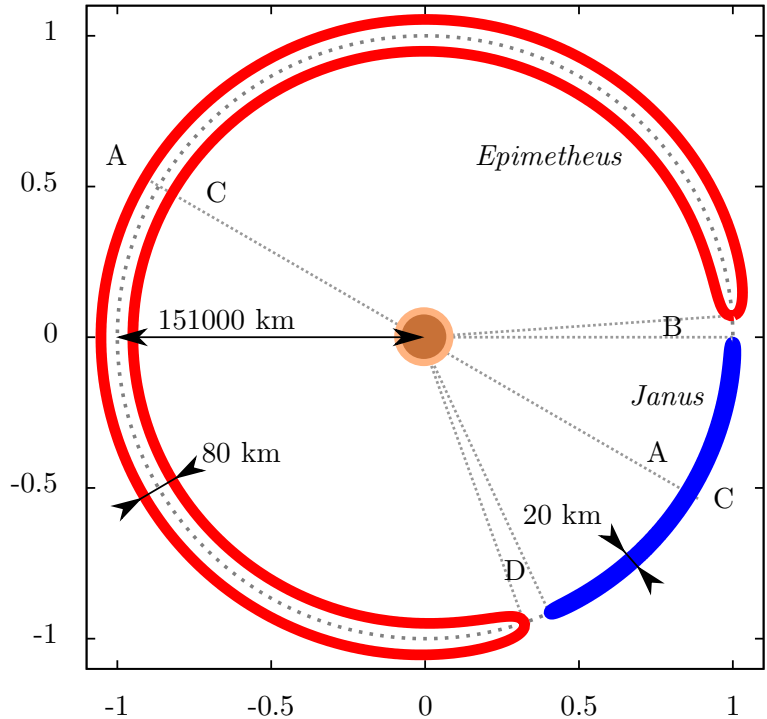}}%
    \put(0.18338801,0.48287659){\color[rgb]{0,0,0}\makebox(0,0)[lb]{\smash{151000 km}}}%
    \put(0.23428421,0.32624292){\color[rgb]{0,0,0}\makebox(0,0)[lb]{\smash{80 km}}}%
    \put(0.72515616,0.67412863){\color[rgb]{0,0,0}\makebox(0,0)[rb]{\smash{\textit{Epimetheus}}}}%
    \put(0.82044202,0.39365277){\color[rgb]{0,0,0}\makebox(0,0)[rb]{\smash{\textit{Janus}}}}%
    \put(0.06693433,0.11419097){\color[rgb]{0,0,0}\makebox(0,0)[rb]{\smash{-1}}}%
    \put(0.06693433,0.29094574){\color[rgb]{0,0,0}\makebox(0,0)[rb]{\smash{-0.5}}}%
    \put(0.06693433,0.46770062){\color[rgb]{0,0,0}\makebox(0,0)[rb]{\smash{0}}}%
    \put(0.06693433,0.6443717){\color[rgb]{0,0,0}\makebox(0,0)[rb]{\smash{0.5}}}%
    \put(0.06693433,0.82112658){\color[rgb]{0,0,0}\makebox(0,0)[rb]{\smash{1}}}%
    \put(0.13620157,0.03052338){\color[rgb]{0,0,0}\makebox(0,0)[b]{\smash{-1}}}%
    \put(0.31513548,0.03052338){\color[rgb]{0,0,0}\makebox(0,0)[b]{\smash{-0.5}}}%
    \put(0.4941532,0.03052338){\color[rgb]{0,0,0}\makebox(0,0)[b]{\smash{0}}}%
    \put(0.67308711,0.03052338){\color[rgb]{0,0,0}\makebox(0,0)[b]{\smash{0.5}}}%
    \put(0.85202102,0.03052338){\color[rgb]{0,0,0}\makebox(0,0)[b]{\smash{1}}}%
    \put(0.60773027,0.254131){\color[rgb]{0,0,0}\makebox(0,0)[lb]{\smash{20 km}}}%
    \put(0.65181645,1.64084301){\color[rgb]{0,0,0}\makebox(0,0)[lb]{\smash{}}}%
    \put(0.14759108,0.6746695){\color[rgb]{0,0,0}\makebox(0,0)[rb]{\smash{A}}}%
    \put(0.77682793,0.32605058){\color[rgb]{0,0,0}\makebox(0,0)[rb]{\smash{A}}}%
    \put(0.2390625,0.63091298){\color[rgb]{0,0,0}\makebox(0,0)[lb]{\smash{C}}}%
    \put(0.8287301,0.2722496){\color[rgb]{0,0,0}\makebox(0,0)[lb]{\smash{C}}}%
    \put(0.79332845,0.47459567){\color[rgb]{0,0,0}\makebox(0,0)[rb]{\smash{B}}}%
    \put(0.59882711,0.18834046){\color[rgb]{0,0,0}\makebox(0,0)[b]{\smash{D}}}%
  \end{picture}%
\endgroup%
\end{center}
\caption{Schematic representation of the Saturn-Janus-Epimetheus trajectories which are depicted in an appropriate rotating frame that rotates with the moons' average mean-motion.
They describe a horseshoe shape whose radial amplitude is about $80$ km for Epimetheus (red curve) and $20$ km for Janus (blue curve). Starting from the configuration A where Janus, Saturn, and Epimetheus are aligned and the latter is the outer moon, Janus catches up with Epimetheus and a close encounter occurs: due to their mutual gravitational interaction the inner moon shifts towards the outer one and vice-versa  (configuration B).
More precisely, without overtaking Epimetheus, Janus decelerates and ``falls" towards the outer orbit.
Likewise, Epimetheus accelerates as it becomes the inner moon and moves away from Janus until another aligned configuration is reached (configuration C).
Next, Epimetheus catches up with Janus, a close encounter occurs, and another orbital exchange takes place (configuration D).
It takes about $4$ years between each orbital exchange and about $8$ years for Janus and Epimetheus to cover all their horseshoe-shaped trajectories (which corresponds to  $4000$ revolutions around Saturn).}
\label{fig:RTHS}
\end{figure}
%%%%%%%%%%%%%%%%%%%%%%%%%%%%%%%%

%%%% Trajectoires HS %%%%%%%%%%%%%%%%%%%%
Regarding Janus and Epimetheus, as Figure \ref{fig:RTHS} shows, they exhibit a horseshoe-shaped trajectory.
As they orbit around Saturn (in about 17 hours) on quasi-coplanar and quasi-circular trajectories whose radii are only $50$ km apart (less than their respective diameters), their mean orbital frequency is slightly different (the inner body being a little faster than the outer one).
Thus, the bodies approach each other every four terrestrial years and their mutual gravitational influence leads to a swapping of the orbits:
the outer moon becomes the inner one and vice-versa. This behavior generates the horseshoe trajectories depicted in an appropriate%%
%%%
\footnote{The horseshoe trajectories are depicted in the frame that rotates with the moons' average mean motion.}
%%%	
rotating frame.
This surprising dynamics of the Janus-Epimetheus co-orbital pair was confirmed by Voyager 1's flyby in 1981 \citep[see][]{1985Ak}.

%%% Outside the hill sphere %%%%%%%%%%%%%%%%%%%
Actually, the two moons exchange their orbits after a relatively close approach whose minimal distance is larger than $10 000$ km, which is too far apart to get in their respective Hill's sphere%%
%%%%%
	\footnote{Which is the gravitational sphere of influence where the primary acts as a perturbator.}
%%%%%
whose radius is around $150$ km. Hence, the gravitational influence of the planet dominates the orbital dynamics of Janus and Epimetheus while their mutual interaction remains only a perturbation.

 Summarizing, we look for  coplanar, low eccentricity co-orbital trajectories which mimic the behavior of these satellites. Contrarily to the tadpole orbits (Helene, Polydeuces, Calypso, and Telesto as in Figure \ref{fig:RTTadpole} and also Jupiter's Trojans) where the difference of mean longitudes oscillates around  $\pm 60^\circ$ (see Section \ref{sec: r11}), for the considered horseshoe trajectories  this quantity oscillates around $180^\circ$ with a large amplitude, larger than $312^\circ$ (see Sections  \ref{sec: r11} and \ref{sec:mech_syst} for details).

%%%%%%%%%%%%%%%%%%%%%%%%%%%%%%%%%%%%

%%%%%%%%%%%%%%%%%%%%%%%%%%%%%%%%%%

\smallskip

%%% Travaux sur les HS (avant J & E) %%%%%%%%%%%
From a theoretical point of view, using a suitable approximation of the restricted three-body problem (RTBP)%%
%%%
\footnote{In this approximation, it is assumed that the massless one  does not affect the motion of the other two, which is consequently Keplerian.}
%%%	
and without available observations, \cite{1911Br}  was the first to consider ``horseshoe" orbits which encompass $L_4$, $L_3$, $L_5$ equilibria and predicted that they were possible solutions of the system.
Subsequently, some horseshoe orbits and families of orbits of this kind have been found numerically in the planar RTBP with respect to the Sun and Jupiter by \cite{1961Ra} followed by numerous other authors.
%%%%%%%%%%%%%%%%%%%%%%%%%%%%%%%%%%

%%% Travaux sur les HS (après J &  E) %%%%%%%%%%%
Several analytical theories have been developed to describe the long-term dynamics of the Janus-Epimetheus co-orbital pair and, more generally, of horseshoe motions in the three-body problem.
%%%%%%%%%%%%%%%%%%%%%%%%%%%%%%%%%%

%%% Matching approach M%%%%%%%%%%%%%%%%%%
One approach, elaborated by \cite{1985SpWa}, lies in the description of the two moon dynamics by matching two adapted approximations: the outer one where the moons do not interact when they are apart and the inner one where the mutual gravitational influence dominates during the close encounter. As we have seen, from an astronomical point of view this is not relevant with the observed motions of the Janus and Epimetheus. From a mathematical point of view, and this is central in the present paper, the reasoning followed by \cite{1985SpWa} does not lie in classical setting of the  implicit function theorem or KAM theory where we need to consider a perturbation of a unique integrable system.

%%%%%%%%%%%%%%%%%%%%%%%%%%%%%%%%%%

%%% Perturbation theory approach %%%%%%%%%%%%%
 Another approach, which is followed in the present paper, is precisely based on the introduction of a unique integrable approximation associated with the co-orbital resonance.
This kind of global model was introduced in the late seventies. In the context of the restricted problem (RTBP) by tacking the mass ratio between the secondary and the primary small enough, \cite{1977Ga} develops an  approximation adapted to quasi-circular orbits in co-orbital resonance  in order to study the behavior of the Trojan asteroids.
%%% Garfinkel %%%%%%%%%%%%%%%%%%%%%%%%%
%%% Yoder %%%%%%%%%%%%%%%%%%%%%%%%%%%
Following the same idea, \cite{1983YoCoSy} give a less accurate approximation of the co-orbital resonance but applicable to the situation of two comparable moons such as Janus and Epimetheus%%
%%%	
	\footnote{Indeed, Janus is only $3$ times more massive than Epimetheus.
	This is a particular case since for all the co-orbital pairs of celestial objects observed up to now, one is very small with respect to the other hence the RTBP is a good model except for Janus-Epimetheus.}.
%%%
%%%%%%%%%%%%%%%%%%%%%%%%%%%%%%%%%%%
%
%%% Dermott %%%%%%%%%%%%%%%%%%%%%%%%%%
Going back to the framework of the restricted problem (RTBP), the most detailed numerical exploration of horseshoe dynamics has been carried out in \cite{1981DeMu,1981DeMua}. Focusing on quasi-circular trajectories, they provide some general properties such as heuristic estimates of the horseshoe orbit lifetime and the relative width of the tadpole and horseshoe domain.
%%%%%%%%%%%%%%%%%%%%%%%%%%%%%%%%%%%

%%%  %%%%%%%%%%%%%%%%%%%%%%%%%%
 In spite of these analytical theories as well as the indications provided by some numerical investigations \citep[see][]{2001LlOl,2013BeFaPe}, so far very few rigorous long-time stability results have been obtained on the ``horseshoe motion", even in the restricted three-body problem.
%%%%

  \cite{2003CoHa} studied the persistence of these trajectories in the three-body problem with the help of a Hamiltonian formulation of the planetary %\footnote{In the planetary three-body problem, two bodies gravitate around a much more massive primary.}
problem by introducing several small quantities (the mass of the moons, the radii difference, and the minimum angular separation between the moons) whose relative sizes are determined in order to explore the horseshoe dynamics. Their approximation of the dynamics in the co-orbital resonance retains  terms up to a given order in the expansion of the perturbation which correspond to the mutual interaction between the moons. An important feature in \cite{2003CoHa} is that their approximation is valid for an area in phase space composed of orbits where the mutual distance at closest approach of the satellites is comparable to $\varepsilon^\alpha$ for $0<\alpha <1/5$ with the ratio $\varepsilon$ of the moons' masses to the central body's mass. This allows one to get results on the existence of horseshoe-shaped orbits over times of order $\varepsilon^{\alpha -2}$.

%%%%%%%%%Les espagnols

 After completion of the present work, \cite{2019CoPaYa} published an article about the
existence of quasi-periodic horseshoe trajectories.
  The authors use \cite{2003CoHa} approximation to give a computer-assisted proof for the
existence of 4-dimensional Lagrangian tori associated with co-orbital orbits where there
are close approaches of the satellites. Even if some details of the proof are omitted in
this paper, this promising work provides normalizing transformations of the considered Hamiltonian
designed to prove the existence of invariant tori.

\smallskip

%%%%%%%Our paper

 Going back to the strategy of \cite{1977Ga}, \cite{1983YoCoSy}, \cite{1981DeMu,1981DeMua}, a Hamiltonian formalism adapted to the study of the motion of two planets in co-orbital resonance was developed in \cite{2013RoPo}. This yields
 a global integrable 1:1-resonant normal form which was specified in \cite{2016RoNiPo} where estimates on the required averaging process are given (and also some stability results). In particular, this model is valid for any orbits in 1:1 resonance provided the mutual distance at closest approach is reduced by $\varepsilon^{1/3}$, which corresponds to Hill's sphere, and its domain of definition includes the $L_3$, $L_4$ and $L_5$ equilibria.  A drawback to this method is that the action-angle variables in this integrable approximation are not explicit. Nevertheless, we can compute asymptotic estimates on the frequencies of trajectories close to orbits homoclinic to the hyperbolic equilibrium $L_3$ (see Figure \ref{fig:RTHS}).

%%% But %%%%%%%%%%%%%%%%%%%%%%%%%%%%%%
In the present paper, our goal is to prove the existence of invariant tori on which the trajectories are similar to those followed by Janus and Epimetheus and for this purpose we apply KAM theory with the latter integrable approximation in co-orbital resonance.

%%%%%%%%%%%%%%%%%%%%%%%%%%%%%%%%%%
In his seminal article, \citeauthor{1963Ar} (\citeyear{1963Ar}) proved rigorously the existence of quasi-periodic motions in the planar planetary three-body  problem. This has been extended to the spatial $N$-body problem by \cite{2004Fe} and \cite{ChPi2011}.
 Going back to the spatial three-body  problem,  \cite{BiCheVa2003} proved the existence of lower dimensional invariant tori while \cite{MR967629} and \cite{Lei2015} proved the existence of quasi-periodic almost-collisional orbits.

Assuming that the planets never experience close encounters, \citeauthor{1963Ar} first considered two uncoupled Kepler problems as the integrable part of the Hamiltonian.
In order to get Kolmogorov non-degeneracy of the frequency map, he added a suitable approximation of the secular part of the perturbation%%
%%%
	\footnote{Which is the averaged perturbation along the Keplerian flows.}.

%%% KAM Difficultés 2 %%%%%%%%%%%%%%%%%%%%%%
In the co-orbital resonance, KAM theory has already been applied: in the restricted three-body problem, \cite{1962Le} proved the existence of quasi-periodic tadpole trajectories.
His reasonings are based on a fourth degree expansion  of the Hamiltonian around Lagrange equilateral configurations which yields a Kolmogorov non-degenerate integrable Hamiltonian.
Unfortunately, this method is only relevant in the neighborhood of the equilateral equilibria and does not fit trajectories that encompass the  $L_4$, $L_3$, and $L_5$ equilibria such as the horseshoe orbits.
%%%%%%%%%%%%%%%%%%%%%%%%%%%%%%%%%%%%

\medskip

%%% Outside the hill sphere %%%%%%%%%%%%%%%%%%%
  As discussed before, in our context the mutual interaction of the moons remains a perturbation of the main force which comes from the central attractor.  As a consequence, the planetary three-body problem studied by \citeauthor{1963Ar} is also relevant for modeling Janus and Epimetheus' trajectories around Saturn.

%%% KAM Difficultés 1 %%%%%%%%%%%%%%%%%%%%%%
  We would like to use KAM theory in order to prove the existence of quasi-periodic trajectories whose main features are those of the observed satellite's trajectories but our context is tricky : unlike \citeauthor{1963Ar}'s situation that relies on non-resonant Kepler orbits, we are strictly in 1:1 resonance which prevents of use the secular perturbation in order to get a non-degeneracy.
%%%%%%%%%%%%%%%%%%%%%%%%%%%%%%%%%%%%

%%% Paragraphe disparu
%{\bf Je crois que l'on avait pas mis ce paragraphe à cette endroit mais je ne me souviens plus o\`u....}
%%% R1:1 %%%%%%%%%%%%%%%%%%%%%%%%%%%%%%
%More specifically, assuming the moons as co-orbitals, their mean motions are commensurable and lead to a $1:1$ mean motion resonant behaviour called co-orbital resonance.
%This peculiar dynamics possesses a ``resonant angle" given by the difference of mean longitude of the moons that does not rotate but oscillates around some particular values instead, and characterizes the kind of co-orbital motion.
%Indeed, as we can see in the figure \textbf{ImageArticlePhilippeMaryame}, the resonant angle of Saturn's co-orbital moons in tadpole motion oscillates around $60^\circ$ (or $-60^\circ$) while, for Janus and Epimetheus $\zeta_1$ oscillates around $180^\circ$ with a large amplitude lying between $5^\circ$ and $-5^\circ$.
%%%%%%%%%%%%%%%%%%%%%%%%%%%%%%%%%%%%

%%% Idée de la preuve %%%%%%%%%%%%%%%%%%%%%%%%%%
In order to prove the existence of quasi-periodic horseshoe orbits, we replace the previous secular perturbation by the integrable 1:1-resonant normal form introduced by \cite{2013RoPo}. Since the action-angle variables in the integrable approximation are not explicit, it is very tricky to check Kolmogorov's non-degeneracy condition as in \citeauthor{1963Ar}'s article. However, it is possible to look at weaker non-degereracy conditions, like those stated by  \cite{1996Po} to prove the persistence of lower dimensional normally elliptic invariant tori in the context of non-linear partial differential equations%%
%%%
\footnote{This result was initially stated by \cite{1965Me} and independently proved by \cite{1988El} and \cite{1988Ku}.}.
%%%
This latter result was already applied in celestial mechanics by \cite{BiCheVa2003} to prove the existence of 2d elliptic invariant tori for the three-body planetary problem in a non-resonant case (while co-orbital trajectories are resonant). In our context, we will follow the same scheme of proof and, as a consequence, we give a rigorous proof of 2-dimensional tori associated with horseshoe like motions.
%%%%%%%%%%%%%%%%%%%%%%%%%%%%%%%%%%%%%%%%
%Our main theorem is theorem \ref{maintheorem} at the end of section \ref{sec:intro_tech}.
Our main theorem (Theorem \ref{maintheorem}) is stated at the end of Section \ref{sec:intro_tech}.

 Actually, it is certainly possible to compute higher order normal forms with a computer-assisted proof in order to check Kolmogorov's non-degeneracy condition in our setting and ensure the existence of Lagrangian invariant tori.

\medskip

%%% Plan %%%%%%%%%%%%%%%%%%%%%%%%%%%%%%%%%%

%The difficulty is overcame by implementing a more recent KAM theorem developed
%%%%%
%	\footnote{Theorem stated by \cite{1965Me}  and independently proved by \cite{1988El} and \cite{1988Ku}.}
%%%%%
% by \cite{1996Po} which require particularly weak conditions.
%%%%%%%%%%%%%%%%%%%%%%%%%%%%%%%%%%%%%%%%

%%% Plan %%%%%%%%%%%%%%%%%%%%%%%%%%%%%%%%%%

In Section \ref{sec:intro_tech}, we specify the characteristics of the quasi-periodic orbits we want to obtain.
In Section \ref{sec:notations}, some useful notations are introduced.
In Section \ref{sec:red_sec}, we describe the different steps of our reduction scheme in order to build an integrable approximation associated with the horseshoe motion.
Section \ref{sec:KAM} is dedicated to the application of KAM theory.
Finally, Section \ref{sec:Comments} is devoted to extensions, comments, and prospects.

Appendix \ref{sec:proof} concerns the proof of the technical propositions and lemmas used in our reasonings.
%
%The definitions and notations and definitions  in the remainder of the paper   will be introduced in section \ref{sec:notations}. The last sections are dedicated \texttt{to be continued...}.}

%%%%%%%%%%%%%%%%%%%%%%%%%%%%%%%%%%%%%%%%
%%% MAJ 2018-04-17

%%% Sec2  Intro Technique %%%%%%%
% Hamiltonien,
% Variables PoincarÃ©
% Variables co-orbitales
% RÃ©sumÃ© RoPo13 (quasi-circular manifold)
% But de l'article
% ÃnoncÃ© du rÃ©sultat
%%%%%%%%%%%%%%%%%%%%%%%%%%%%%%
% COORBKAMPLUS
% SEC 2 MAIN RESULT INTRO
% MAJ: 2018-01-31 ALEX
%%%%%%%%%%%%%%%%%%%%%%%%%%%%%%

\section{2d co-orbital tori and horseshoe trajectories in the planetary problem}
\label{sec:intro_tech}

\subsection{Canonical heliocentric coordinates}
\label{sec:canon}
%%% Problème planétaire %%%%%%%%%%%%%%%%
We consider two planets of respective masses $\eps m_1$ and $\eps m_2$ orbiting in a plane around a central body (the Sun or a star) of mass $m_0$, $\eps$ being an arbitrarily small positive parameter.
We assume that the three bodies  are only influenced by their mutual gravitational interaction.
Without loss of generality, we assume  the gravitational constant to be equal to $1$ and set
\bes
0< m_1\leq  m_2 \, .
%\label{massehierarchie}.
\ees
%Thus, this model is as adapted to systems of two planets around a star as two moons around a planet.
%%%%%%%%%%%%%%%%%%%%%%%%%%%%%%%%
%
%%% Heliocentric coodinates %%%%%%%%%%%%%%
Using  heliocentric coordinates \citep{1995LaRo} and  rescaling both action variables and time \citep[see][for more details]{2016RoNiPo}, the Hamiltonian of the three-body problem reads
\be
\label{eq:ham_plan}
	\begin{split}
	\cH(\brt_j,\br_j) &= \cH_K(\brt_j,\br_j) +   \cH_P(\brt_j,\br_j) \,\,\text{with} \\
		\cH_K(\brt_j,\br_j) &=  \sum _{j\in\, \{1,2\}}
	\left(
		\frac{\norm{\brt_j}^2}{2\hm_j} - \frac{\mu_j \hm_j}{\norm{\br_j}}
	\right) \,\text{and}\; \cH_P(\brt_j,\br_j) = \eps
	\left(
		\frac{\brt_1\bigcdot\brt_2}{m_0} -  \frac{m_1m_2}{\norm{\br_1 -\br_2}}
	\right)
	\end{split}
\ee
where ``$\,\bigcdot\,$" and ``$\norm{\,\cdot\,}$" are respectively the Euclidean scalar product and norm.
%%%%%%%%%%%%%%%%%%%%%%%%%%%%%%%%

%%% Description des variables et paramètres %%%%%
In these expressions, the canonical variable $\br_j$ corresponds to the heliocentric position of  planet $j$ while $\brt_j$, the conjugated variable of $\br_j$, is associated with the rescaled barycentric linear momentum of the same body.
The mass parameters $\hm_j$ and $\mu_j$ are defined by
\bes
	\widehat m_j = \frac{m_0 m_j}{m_0+\eps m_j}
\qtext{and}
	\mu_j = m_0+\varepsilon m_j.
%\label{eq:masse_red}
\ees
%%%%%%%%%%%%%%%%%%%%%%%%%%%%%%%%

%%% Description du Hamiltonien %%%%%%%%%%%%
The Hamiltonian $\cH$, which is an analytical function in the domain
\bes
	\cD =
	\left\{
		(\brt_1,\br_1,\brt_2,\br_2)\in\RR^8\ \text{such that}\ \br_1\neq\br_2
	\right\},
\ees
possesses two components:   $\cH_K$, which describes the unperturbed Keplerian motion of the two planets (the motion of a body of mass $\hm_j$ around a fixed center of mass $m_0+\eps m_j$), and $\cH_P$, which models the perturbations due to the gravitational interaction between the two planets and  the fact that the heliocentric frame is not a Galilean one.
%%%%%%%%%%%%%%%%%%%%%%%%%%%%%%%%

%%% Moment cinétique %%%%%%%%%%%%%%%%%%
Finally, the planetary Hamiltonian $\cH$ is invariant under the action of the symmetry group SO(2) associated with the rotations around the vertical axis.
This property is equivalent to the fact that the total angular momentum, that is
$\cCt(\brt_j, \br_j) = \sum_{j\in\{1,2\}} \brt_j \times \br_j$
 (where ``$\times$" is the vectorial product), is preserved.
%%%%%%%%%%%%%%%%%%%%%%%%%%%%%%%%

%%% Domaine d'analyticité %%%%%%%%%%%%%%%
%The Hamiltonian $\cH$ is analytical on the domain
%%
%$$
%	\cD =
%	\left\{
%		(\brt_1,\br_1,\brt_2,\br_2)\in\RR^8\ \text{such that}\ \br_1\neq\br_2
%	\right\}
%$$
%%
%which corresponds to the whole phase space ($\RR^{8}$ since we consider only the planar problem) excluding the collision manifold.
%%%%%%%%%%%%%%%%%%%%%%%%%%%%%%%

\subsection{Poincar\'e complex variables.}
\label{sec:poincare}

%%% Variables poincaré complexes %%%%%%%%%%
In order to define a canonical coordinate system related to the elliptic elements $(a_j,e_j,\lam_j,\varpi_j)$ (respectively the semi-major axis, the eccentricity, the mean longitude, and the longitude of the pericenter of the planet $j$), we use   Poincar\'e's complex variables
$(\Lam_j,\lam_j,x_j,\xt_j)_{j\in\{ 1,2\} }\in (\RR\times\TT \times\CC \times\CC)^2$:
\bes
\begin{gathered}
	\Lam_j = \hm_j\sqrt{\mu_ja_j},
\quad
	x_j= \sqrt{\Lam_j}\sqrt{1-\sqrt{1-e_j^2}}\exp(i\varpi_j),
\quad\mbox{and}\quad
\xt_j = -i\xb_j.
\end{gathered}
\ees
This coordinate system has the advantage of being regular when the eccentricities tend to zero.
Consequently, the product $\Upst = (\Upst_1,\Upst_2)$ of analytic symplectic transformations
\bes
	\Upst_j(\Lam_j,\lam_j,x_j,\xt_j )=(\brt_j ,\br_j ) \quad  (j\in\{1,2\})
\ees
yields the new Hamiltonian
\bes
\begin{split}
	\Ht(\Lam_j,\lam_j,x_j,\xt_j)&=\Ht_K(\Lam_1,\Lam_2) +\Ht_P(\Lam_j,\lam_j,x_j,\xt_j) \\
\mbox{where} \quad	\cH_K(\brt_j,\br_j)	&=\Ht_K (\Lam_1,\Lam_2)
								= -\sum_{j\in\, \{1,2\}} \frac{1}{2}\frac{\mu_j^2\hm_j^3}{\Lam_j^2}.
\end{split}
%\label{eq:Ham_poinc}
\ees
%
%%%%%%%%%%%%%%%%%%%%%%%%%%%%%%%
%%% Analyticité de l'Hamiltonien %%%%%%%%%%%
$\Ht$ is analytic on the domain
$\Upst^{-1}(\cD) \subset  (\RR\times\TT \times\CC\times\CC)^2$.

%%%%%%%%%%%%%%%%%%%%%%%%%%%%%%%

	\subsection{The 1:1 resonance}
	\label{sec: r11}

%%% Description R1:1: Approx Kep %%%%%%%%%%
In the limit of the Keplerian approximation,  two planets are in  co-orbital resonance, or 1:1 mean-motion resonance, when their two orbital frequencies are equal.
According to the third Kepler law, the exact resonance occurs when $(\Lam_1,\Lam_2) = (\Lam_{1,0},\Lam_{2,0})$ with
%%%%%%%%%%
\be
\begin{aligned}
&\dron{\Ht_K}{\Lam_1}(\Lam_{1,0},\Lam_{2,0})  	=  \dron{\Ht_K}{\Lam_2}(\Lam_{1,0},\Lam_{2,0})
																			= \upsilon_0 >0, \\
&\qtext{that is}\, 	\frac{\mu_1^2\hm_1^3}{(\Lam_{1,0})^3}  	=\frac{\mu_2^2\hm_2^3}{(\Lam_{2,0})^3}
																	=\upsilon_0, \\
&\qtext{where}  \Lam_{j,0} =\hm_j \sqrt{\mu_j a_{j,0}}
\qtext{and} a_{j,0} = \mu_j^{1/3}\upsilon_0^{-2/3}
\end{aligned}
\label{eq:exact_res}
\ee
are respectively the exact-resonant action and  semi-major axis of the planet $j$.
%%%%%%%%%%%%%%%%%%%%%%%%%%%%%%

%%% Description R1:1: Perturbation %%%%%%%%%
%
In order to construct a coordinate system adapted to the co-orbital resonance, let us introduce the symplectic transformation
\bes
	\Ups(\bZ, \bzeta, \bx, \bxt) = (\Lam_1,\lam_1,x_1,\xt_1,\Lam_2,\lam_2,x_2,\xt_2)	
\ees
such that
\bes
\begin{split}
	\bZ=
	\begin{pmatrix}
		Z_1\\
		Z_2
	\end{pmatrix}
	=
	\begin{pmatrix}
		1 	& 	0\\
		1 	& 	1
	\end{pmatrix}
	\begin{pmatrix}
		\Lam_1-\Lam_{1,0}\\
		\Lam_2-\Lam_{2,0}
	\end{pmatrix},&
\quad
	\bzeta =
	\begin{pmatrix}
		\zeta_1\\
		\zeta_2
	\end{pmatrix}
	=
	\begin{pmatrix}
		1	&	-1\\
		0	& 	1
	\end{pmatrix}
	\begin{pmatrix}
		\lam_1 \\
		\lam_2
	\end{pmatrix},\\
\bx=
	\begin{pmatrix}
	x_1\\
	x_2
	\end{pmatrix},&
\quad
\bxt = -i\overline{\bx}.
\end{split}
\ees
%
%%%%%%%%%%%%%%%%%%%%%%%%%%%%%
%
%%% Hamiltonien %%%%%%%%%%%%%%%%%%
In these variables, the planetary Hamiltonian becomes
\bes
\begin{split}
 		H(\bZ, \bzeta, \bx, \bxt)  &= H_K(\bZ) + H_P(\bZ, \bzeta, \bx, \bxt),\\
	\mbox{where}\quad H_K(\bZ) &= 	- \frac{\hm_1^3\mu_1^2}{2(\Lam_{1,0} +  Z_1)^2}
					- \frac{\hm_2^3\mu_2^2}{2(\Lam_{2,0}  + Z_2 - Z_1 )^2}
			\mbox{ and } H_P = \Ht_P\circ  \Ups ,
\end{split}
%\label{eq:planetaryHam}
\ees
 %
%%%%%%%%%%%%%%%%%%%%%%%%%%%%
%%% Time-scale %%%%%%%%%%%%%%%%%
which yields a zero frequency $\partial_{Z_1} H_K (\bzero)$ at the origin.  Hence, the temporal evolution of the angles $\zeta_j$ and variables $x_j$ satisfy the relation
\bes
	\dot \zeta_2 = \upsilon_0   +  \cO(\norm{\bZ}) + \cO(\eps),
\quad
	\dot \zeta_1 =   \cO(\norm{\bZ})+  \cO(\eps),
\qtext{and}
	\dot x_j  = \cO(\eps)
\ees
where $\eps$ is the small parameter associated with the planetary masses (see Section \ref{sec:canon}).

%%%%%%%%%%%%%%%%%%%%%%%%%%%%
As a consequence, these variables  evolve at different rates: $\zeta_2$ is a ``fast" angle with a frequency of order $1$,  $\zeta_1$ undergoes ``semi-fast" variations at a frequency of order $\sqrt{\eps}$ (in the resonant domain, $\bZ$ is at most of order $\sqrt\eps$ as it is shown in Section \ref{sec:mech_syst}), while the variables $(x_j)_{j\in\{1,2\}}$ related to the eccentricities are associated with the ``slow" degrees of freedom evolving on a timescale of order $\eps$ (secular variations of the orbits).
%%%%%%%%%%%%%%%%%%%%%%%%%%%%

%%% Averaging introduction %%%%%%%%%%%
A classical way to reduce the problem in order to study the semi-fast and secular dynamics of the co-orbital resonance is to average the Hamiltonian over the fast angle $\zeta_2$ to get  a resonant normal form.
We shall prove, in Section \ref{sec:first_normal_form}, that there exists a symplectic transformation $\overline{\Ups}$ close to the identity and defined on a domain that will be specified later, such that
\bes
	\overline{\Ups} : (\bZs, \bzetas, \bxs, \bxts)  	\longmapsto    	(\bZ,\bzeta,\bx,\bxt )
\ees
and
\bes
\begin{gathered}
	 H\circ\overline{\Ups}(\bZs, \bzetas, \bxs, \bxts)
										=	H_K(\bZs) + \Hb_P(\bZs, \zetas_1, \bxs, \bxts) + H_*(\bZs, \bzetas, \bxs, \bxts)\\
\mbox{where}\quad	\Hb_P(\bZs, \zetas_1,  \bxs, \bxts ) =\frac{1}{2\pi}\int_0^{2\pi} H_P(\bZs_j, \zetas_1, \zetas_2,\bxs, \bxts) \rd\zetas_2 \, .
\end{gathered}
\ees
%
%%%
$\Hb_P$ is the averaged perturbation which depends only on the semi-fast and slow  variables while the remainder $H_*$ is supposed to be small with respect to  $\Hb_P$.
More precisely, the sizes of $\Hb_P$ and $H_*$ increase simultaneously with the distance to the singularity associated with the planetary collision.
We showed in \cite{2016RoNiPo} that  the remainder is negligible compared  to $\Hb_P$  as long as the distance to the singularity is less than a quantity of order  $\eps^{1/3}$, which will be assumed from the Section \ref{sec:red_sec}.
In addition, properties regarding the transformation $\overline{\Ups}$ and the remainder $H_*$ will be stated in Section \ref{sec:first_normal_form}.
More precisely, the averaging process will be iterated until the fast component is exponentially small with respect to $\eps$.
%%%%%%%%%%%%%%%%%%%%%%%%%%%%

\subsubsection{D'Alembert rule and averaged Hamiltonian's dynamics}

%%% DAlembert %%%%%%%%%%%%%%%%%%%%%
The Hamiltonian $H$, which is analytic on a suitable domain,  can be expanded in Taylor series in a neighborhood of $\bx=\bxt=\bzero$ as
\begin{gather}
	H(\bZ,\bzeta,\bx,\bxt) = \sum_{(k,\bp,\bpt) \in \sD} f_{k,\bp,\bpt}(\bZ,\zeta_1)x_1^{p_1}x_2^{p_2}\xt_1^{\pt_1}\xt_2^{\pt_2}\exp (i k\zeta_2)
	\label{eq:Fourier_Taylor}\\
	\mbox{where}\quad \sD=\left\{ (k,\bp,\bpt)\in\ZZ\times\NN^2\times\NN^2\,/\, k + p_1 + p_2 - \pt_1-\pt_2 = 0\right\}
	\label{eq:DAl}
\end{gather}
is known as the D'Alembert rule which is equivalent to the conservation of the angular momentum $\cCt(\brt_j, \br_j)$.
From this relation follows a key property of the averaged Hamiltonian that reads
\bes
\begin{gathered}
	\Hb(\bZs, \zetas_1,\bxs, \bxts) = H_K(\bZs) + \Hb_P(\bZs, \zetas_1,\bxs, \bxts) + \Hb_*(\bZs, \zetas_1,\bxs, \bxts)\\ \mbox{where} \quad \Hb_*(\bZs, \zetas_1,  \bxs, \bxts ) =\frac{1}{2\pi}\int_0^{2\pi} H_*(\bZs, \zetas_1, \zetas_2,\bxs, \bxts) \rd\zetas_2 \, .
\end{gathered}
\ees
Indeed, the D'Alembert rule still holds after averaging \citep[see][]{2013RoPo} and the Taylor expansion of the averaged Hamiltonian $\Hb$, which does not depend on the angle $\zetas_2$,  is even in the slow variables $(\bxs, \bxts)$.
Moreover, this propriety is equivalent to the fact that the quantity
\be
	\cC(\Zs_2, \bxs, \bxts)= \Zs_2 + i\xs_1\xts_1 + i\xs_2\xts_2
\label{eq:C_int}
\ee
is an integral of the averaged motion.
As a consequence, the set
\bes
	\rC_0=\{ \bxs = \bxts= \bzero\}
\ees
 is an invariant manifold for the flow of $\Hb$.
On this ``quasi-circular" manifold, the dynamics is controlled by the one-degree of freedom Hamiltonian
$\Hb(\bZs, \zetas_1,\bzero, \bzero) $.

	\subsection{Semi-fast dynamics and horseshoe domain}
	\label{sec:semi_HS}

%%%%%%%%%%%%%%%%%%%%%%%%%%%%%%%
\begin{figure}[ht]
	\begin{center}
		\def\svgwidth{1\textwidth}
	%% Creator: Inkscape inkscape 0.92.2, www.inkscape.org
%% PDF/EPS/PS + LaTeX output extension by Johan Engelen, 2010
%% Accompanies image file 'PhaseP.eps' (pdf, eps, ps)
%%
%% To include the image in your LaTeX document, write
%%   \input{<filename>.pdf_tex}
%%  instead of
%%   \includegraphics{<filename>.pdf}
%% To scale the image, write
%%   \def\svgwidth{<desired width>}
%%   \input{<filename>.pdf_tex}
%%  instead of
%%   \includegraphics[width=<desired width>]{<filename>.pdf}
%%
%% Images with a different path to the parent latex file can
%% be accessed with the `import' package (which may need to be
%% installed) using
%%   \usepackage{import}
%% in the preamble, and then including the image with
%%   \import{<path to file>}{<filename>.pdf_tex}
%% Alternatively, one can specify
%%   \graphicspath{{<path to file>/}}
%%
%% For more information, please see info/svg-inkscape on CTAN:
%%   http://tug.ctan.org/tex-archive/info/svg-inkscape
%%
\begingroup%
  \makeatletter%
  \providecommand\color[2][]{%
    \errmessage{(Inkscape) Color is used for the text in Inkscape, but the package 'color.sty' is not loaded}%
    \renewcommand\color[2][]{}%
  }%
  \providecommand\transparent[1]{%
    \errmessage{(Inkscape) Transparency is used (non-zero) for the text in Inkscape, but the package 'transparent.sty' is not loaded}%
    \renewcommand\transparent[1]{}%
  }%
  \providecommand\rotatebox[2]{#2}%
  \ifx\svgwidth\undefined%
    \setlength{\unitlength}{576.3114624bp}%
    \ifx\svgscale\undefined%
      \relax%
    \else%
      \setlength{\unitlength}{\unitlength * \real{\svgscale}}%
    \fi%
  \else%
    \setlength{\unitlength}{\svgwidth}%
  \fi%
  \global\let\svgwidth\undefined%
  \global\let\svgscale\undefined%
  \makeatother%
  \begin{picture}(1,0.38751236)%
    \put(0,0){\includegraphics[width=\unitlength]{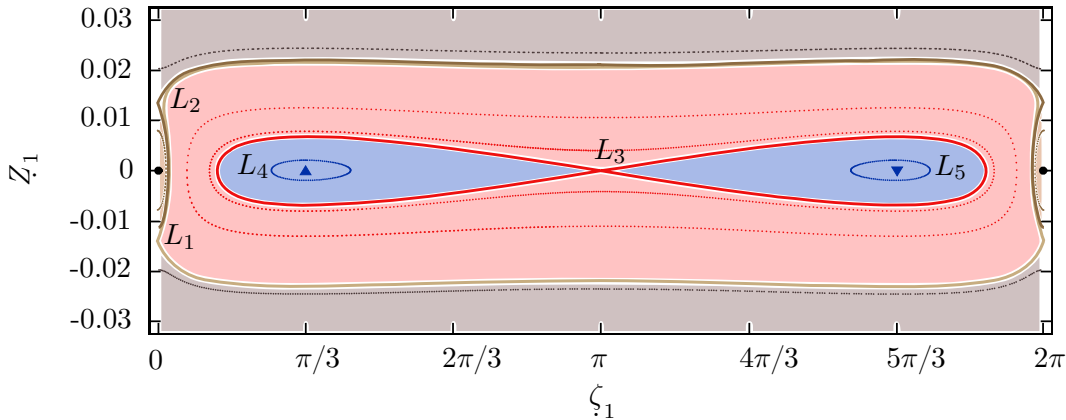}}%
    \put(-0.16648511,1.16288665){\color[rgb]{0,0,0}\makebox(0,0)[lb]{\smash{ }}}%
    \put(0.12444695,0.07582114){\color[rgb]{0,0,0}\makebox(0,0)[rb]{\smash{-0.03}}}%
    \put(0.12412923,0.16691772){\color[rgb]{0,0,0}\makebox(0,0)[rb]{\smash{-0.01}}}%
    \put(0.12469095,0.26061705){\color[rgb]{0,0,0}\makebox(0,0)[rb]{\smash{0.01}}}%
    \put(0.12500867,0.35171363){\color[rgb]{0,0,0}\makebox(0,0)[rb]{\smash{0.03}}}%
    \put(0.12557803,0.21376738){\color[rgb]{0,0,0}\makebox(0,0)[rb]{\smash{0}}}%
    \put(0.67499026,0.04028976){\color[rgb]{0,0,0}\makebox(0,0)[lb]{\smash{$4\pi/3$}}}%
    \put(0.13918493,0.03835961){\color[rgb]{0,0,0}\makebox(0,0)[lb]{\smash{$0$}}}%
    \put(0.53846109,0.00622037){\color[rgb]{0,0,0}\makebox(0,0)[lb]{\smash{$\zetas_1$}}}%
    \put(0.03147057,0.20533256){\color[rgb]{0,0,0}\rotatebox{90}{\makebox(0,0)[lb]{\smash{$\Zs_1$}}}}%
    \put(0.80864193,0.04028976){\color[rgb]{0,0,0}\makebox(0,0)[lb]{\smash{$5\pi/3$}}}%
    \put(0.94209804,0.04028981){\color[rgb]{0,0,0}\makebox(0,0)[lb]{\smash{$2\pi$}}}%
    \put(0.54003718,0.04028976){\color[rgb]{0,0,0}\makebox(0,0)[lb]{\smash{$\pi$}}}%
    \put(0.40547453,0.04028976){\color[rgb]{0,0,0}\makebox(0,0)[lb]{\smash{$2\pi/3$}}}%
    \put(0.27208312,0.04028976){\color[rgb]{0,0,0}\makebox(0,0)[lb]{\smash{$\pi/3$}}}%
    \put(0.12392589,0.12267081){\color[rgb]{0,0,0}\makebox(0,0)[rb]{\smash{-0.02}}}%
    \put(0.12448761,0.30486396){\color[rgb]{0,0,0}\makebox(0,0)[rb]{\smash{0.02}}}%
    \put(0.4157804,0.06890613){\color[rgb]{0,0,0}\makebox(0,0)[lb]{\smash{}}}%
    \put(0.15291558,0.15270628){\color[rgb]{0,0,0}\makebox(0,0)[lb]{\smash{$L_1$}}}%
    \put(0.54323524,0.23072988){\color[rgb]{0,0,0}\makebox(0,0)[lb]{\smash{$L_3$}}}%
    \put(0.24889384,0.21689001){\color[rgb]{0,0,0}\makebox(0,0)[rb]{\smash{$L_4$}}}%
    \put(0.85196597,0.21636687){\color[rgb]{0,0,0}\makebox(0,0)[lb]{\smash{$L_5$}}}%
    \put(0.1581211,0.2776388){\color[rgb]{0,0,0}\makebox(0,0)[lb]{\smash{$L_2$}}}%
  \end{picture}%
\endgroup%
	\caption{Phase portrait of the Hamiltonian $\Hb(\bZs, \zetas_1, \bzero, \bzero)$  in the coordinates $(\zetas_1, \Zs_1)$.
	 The units and the parameter are chosen such that $\Zs_2 =0$, $m_0=1$, $\upsilon_0 = 2\pi$, $\eps m_1 = 10^{-3}$,  $\eps m_2 = 3 \times 10^{-4}$. The blue area corresponds to tadpole or Trojan orbits while the red one corresponds to horseshoe orbits. See Section \ref{sec:semi_HS} for more details.}
	\label{fig:portrait_phase}
	\end{center}
\end{figure}
%%%%%%%%%%%%%%%%%%%%%%%%%%%%%%%%

%%%
The phase portrait of the integrable Hamiltonian $ \Hb(\bZs, \zetas_1,{\bf 0}, {\bf 0}) $ is displayed in Figure  \ref{fig:portrait_phase}.
This figure being extensively described in  \cite{2013RoPo}, we will limit ourselves to present what will be useful thereafter.
%%%

%%%
Two elliptic  fixed points are present on this phase portrait.
These points, labelled by $L_4$ or $L_5$, coincide with the Lagrange equilateral equilibria, which are linearly stable as long as the planetary masses are small enough%%
	%%%
	\footnote{According to \cite{1843Ga}, when the planetary orbits are circular, the equilateral configurations are linearly stable  if the mass of the three bodies satisfy the relation $27(m_0\eps m_1+m_0\eps m_2 +  \eps m_1\eps m_2) < (m_0 + \eps m_1 +\eps m_2)^2 $. }.
	%%%
They are surrounded by periodic orbits (blue domains) that correspond to semi-fast deformations of Lagrange configurations (the tadpole orbits of Figure \ref{fig:RTTadpole}).
%%%
In the center of the phase portrait, the fixed point labelled by $L_3$ represents the unstable Euler configuration for which the three bodies are aligned and the Sun is between the two planets.
Its stable and unstable manifolds, which coincide (red curve), bound the two previous domains. Outside these separatrices lie the horseshoe orbits (red region).
%%%
Contrarily to the tadpole orbits for which the variation of $\zetas_1$ does not exceed $156^\circ$, along a horseshoe trajectory  the difference of the mean-longitudes $\zetas_1$ oscillate around $180^\circ$ with a very large amplitude of at least $312^\circ$ (see Section \ref{sec:mech_syst}).
It is on this region, more precisely close to the outer edge of the separatrix, that we will focus in the next sections.
%%%

%%%
The outer part of the horseshoe domain is bounded by the separatrices associated with $L_1$ and $L_2$  (beige and brown curves).
Beyond  these manifolds, the top and the bottom light grey areas correspond to non-resonant dynamics where the angle $\zetas_1$ evolves slowly but in a monotonous way.
%%%
The singularity that corresponds to the collision between the planets is located at $\Zs_1 = \zetas_1  = 0$ and is separated from the previous regions by the stable and unstable manifolds originated at $L_1$ and $L_2$.
It is shown in \cite{2016RoNiPo} that  the distance between the singularity and  these structures is of order $\eps^{1/3}$. As mentioned above, in this case, the remainder $H_*$  is at least as large as the perturbation, and this part of the phase portrait is not necessarily relevant. But this is not a problem since, in the following, we will work only in the vicinity of the $L_3$-separatrix.
%%%%%%%%%%%%%%%%%

	\subsection{2d co-orbital tori}
	
	%%% Domaine Da: Figure %%%%%%%%%%%%%%%%%%%%%%%%%%%%%%%
\begin{figure}[ht]
	\begin{center}
	\def\svgwidth{1\textwidth}
%% Creator: Inkscape inkscape 0.92.2, www.inkscape.org
%% PDF/EPS/PS + LaTeX output extension by Johan Engelen, 2010
%% Accompanies image file 'PhaseBlancEntier2.eps' (pdf, eps, ps)
%%
%% To include the image in your LaTeX document, write
%%   \input{<filename>.pdf_tex}
%%  instead of
%%   \includegraphics{<filename>.pdf}
%% To scale the image, write
%%   \def\svgwidth{<desired width>}
%%   \input{<filename>.pdf_tex}
%%  instead of
%%   \includegraphics[width=<desired width>]{<filename>.pdf}
%%
%% Images with a different path to the parent latex file can
%% be accessed with the `import' package (which may need to be
%% installed) using
%%   \usepackage{import}
%% in the preamble, and then including the image with
%%   \import{<path to file>}{<filename>.pdf_tex}
%% Alternatively, one can specify
%%   \graphicspath{{<path to file>/}}
%%
%% For more information, please see info/svg-inkscape on CTAN:
%%   http://tug.ctan.org/tex-archive/info/svg-inkscape
%%
\begingroup%
  \makeatletter%
  \providecommand\color[2][]{%
    \errmessage{(Inkscape) Color is used for the text in Inkscape, but the package 'color.sty' is not loaded}%
    \renewcommand\color[2][]{}%
  }%
  \providecommand\transparent[1]{%
    \errmessage{(Inkscape) Transparency is used (non-zero) for the text in Inkscape, but the package 'transparent.sty' is not loaded}%
    \renewcommand\transparent[1]{}%
  }%
  \providecommand\rotatebox[2]{#2}%
  \ifx\svgwidth\undefined%
    \setlength{\unitlength}{787.27001953bp}%
    \ifx\svgscale\undefined%
      \relax%
    \else%
      \setlength{\unitlength}{\unitlength * \real{\svgscale}}%
    \fi%
  \else%
    \setlength{\unitlength}{\svgwidth}%
  \fi%
  \global\let\svgwidth\undefined%
  \global\let\svgscale\undefined%
  \makeatother%
  \begin{picture}(1,0.49416275)%
    \put(0,0){\includegraphics[width=\unitlength]{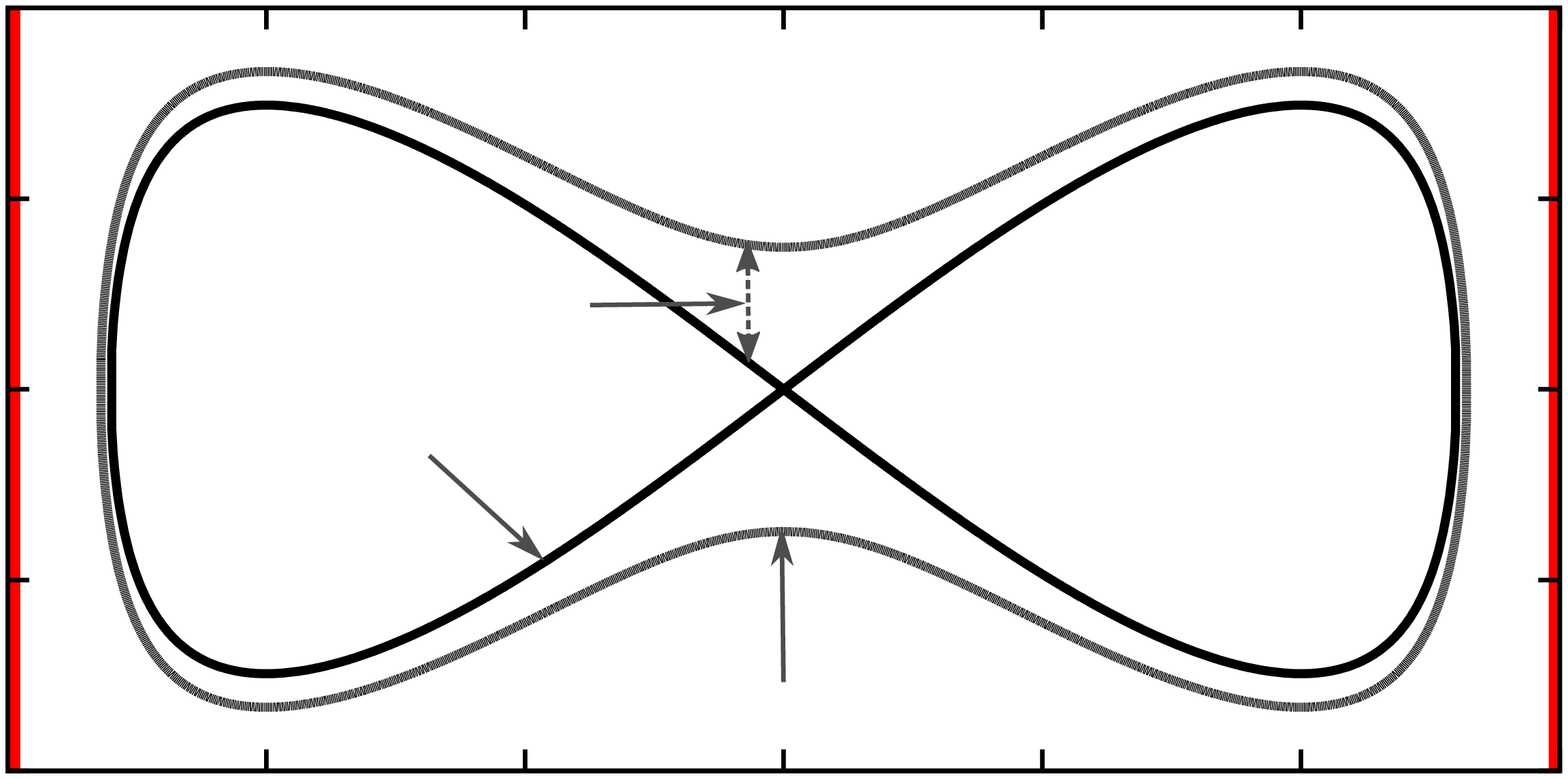}}%
    \put(0.12644491,0.0583348){\makebox(0,0)[rb]{\smash{-0.01}}}%
    \put(0.12644491,0.16026933){\makebox(0,0)[rb]{\smash{-0.005}}}%
    \put(0.12644491,0.26220386){\makebox(0,0)[rb]{\smash{0}}}%
    \put(0.12644491,0.36404313){\makebox(0,0)[rb]{\smash{0.005}}}%
    \put(0.12644491,0.46597766){\makebox(0,0)[rb]{\smash{0.01}}}%
    \put(0.13711469,0.03547098){\makebox(0,0)[b]{\smash{0}}}%
    \put(0.2754408,0.03547098){\makebox(0,0)[b]{\smash{$\pi/3$}}}%
    \put(0.41376691,0.03547098){\makebox(0,0)[b]{\smash{$2\pi/3$}}}%
    \put(0.55199776,0.03547098){\makebox(0,0)[b]{\smash{$\pi$}}}%
    \put(0.69032387,0.03547098){\makebox(0,0)[b]{\smash{$4\pi/3$}}}%
    \put(0.82864998,0.03547098){\makebox(0,0)[b]{\smash{$5\pi/3$}}}%
    \put(0.96697609,0.03547098){\makebox(0,0)[b]{\smash{$2\pi$}}}%
    \put(0.43959768,0.30220881){\color[rgb]{0,0,0}\makebox(0,0)[rb]{\smash{$\cO(\delta\sqrt{\eps})$}}}%
    \put(0.95507793,0.26764345){\color[rgb]{0,0,0}\rotatebox{90}{\makebox(0,0)[b]{\smash{singularity}}}}%
    \put(0.55308538,0.00562874){\color[rgb]{0,0,0}\makebox(0,0)[b]{\smash{$\zeta_{1}$}}}%
    \put(0.02013196,0.27009023){\color[rgb]{0,0,0}\rotatebox{90}{\makebox(0,0)[b]{\smash{$Z_{1}$}}}}%
    \put(0.55147225,0.22731315){\color[rgb]{0,0,0}\makebox(0,0)[b]{\smash{$L_3$}}}%
    \put(0.35802127,0.2365024){\color[rgb]{0,0,0}\rotatebox{-0.38800811}{\makebox(0,0)[rb]{\smash{$\delta=0$}}}}%
    \put(0.55262748,0.0835635){\color[rgb]{0,0,0}\rotatebox{0.51412283}{\makebox(0,0)[b]{\smash{$\delta>0$}}}}%
  \end{picture}%
\endgroup%
	\caption{Phase portrait of the considered Hamiltonian $\sHt_1(Z_1, \zeta_1)=- \eps B(1+\delta)$.
	It approximates the one of the Hamiltonian $\Hb(\bZ, \zeta_1,\bzero, \bzero)$, which is depicted in Figure \ref{fig:portrait_phase}, in the $L_3$-separatrix region.
	The units and the parameter are the same as in Figure \ref{fig:portrait_phase}.
	In this approximation, the $L_1$ and $L_2$ fixed points as well as their separatrices have disappeared while $\{\zeta_1 = 0\}$ became a singularity.
	For $\delta=0$, the separatrix divides the phase portrait in two distinct dynamics:  the two tadpole trajectory domains for $\delta<0$, which are surrounded by the separatrix, and the horseshoe trajectories for $\delta>0$ (grey trajectory).}
	\label{fig:separatrixEntier}
	\end{center}
\end{figure}
%%%%%%%%%%%%%%%%%%%%%%%%%%%%%%%%%%%%%%%%%%
%%%	
%
In Section \ref{sec:red}, we shall introduce a linear transformation that uncouple the fast and semi-fast dynamics.
Moreover, we shall approximate  the semi-fast dynamics in the $L_3$-separatrix region (the two domains surrounded by the separatrix and its outer neighborhood; see Figure \ref{fig:separatrixEntier}) by a simple Hamiltonian proportional to
\bes
	\sHt_1(Z_1,\zeta_1) = -A Z_1^2 + \eps B \cF(\zeta_1)
\ees
where the coefficients $A$, $B$, and the $2\pi$-periodic real function $\cF$ will be defined later.
%%%
With the previous notations, the separatrix is defined by the level curve  $\sHt_1(Z_1,\zeta_1) = h_0 = -\eps B$  while those given by   $\sHt_1(Z_1,\zeta_1) = h_\delta =  -\eps B(1+\delta)$ with $\delta>0$  are the horseshoe orbits surrounding the latter.
%%%
As a consequence, for each $2\pi/\nu_\delta$-periodic trajectory of the differential system associated with the Hamiltonian $\sHt_1$, there exists a $2\pi$-periodic real function $F_\delta$ such that the parametric representation of the latter trajectory reads
\be
	\zeta_1(t) = F_\delta(\nu_\delta t)\, , \quad Z_1(t) = - \frac{\nu_\delta F_\delta'(\nu_\delta t)}{2A}.
	\label{eq:HS_param}
\ee
%%%
Moreover, the function $F_\delta$ parametrizing a horseshoe orbit of energy $h_\delta$ satisfies
\be
	 - \frac{\nu_\delta^2\left(F_\delta'(\nu_\delta t)\right)^2}{4A} + \eps B \cF\left(F_\delta(\nu_\delta t)\right) = -\eps B(1+\delta) \qtext{with} \delta>0.
	 \label{eq:HS_param_H}
\ee
The smaller $\delta$ is, the closer to the separatrix the orbit is.
%%%

%%%
In the same way the fast dynamics will be approached by a quadratic integrable Hamiltonian $\sH_2(Z_2)$ (see Section  \ref{sec:red}).
Hence we are led to consider the Hamiltonian system linked to
\be
	\sHt_1(Z_1,\zeta_1) +\sH_2(Z_2).
	\label{eq:Ham_reference}
\ee
%%%
%%%
Consequently the phase space is foliated in 2-dimensional tori invariant under the Hamiltonian flow linked to (\ref{eq:Ham_reference}).
Hence, we choose as reference tori the 2d tori that are parametrized  in the original variables by
\be
	\left\{\quad\begin{aligned}
	&\lam_1(\btheta) = c_1 +  \theta_2  + (1 - \kappa) F_\delta(\theta_1 )  \\
	&\lam_2(\btheta) = c_2 +  \theta_2  - \kappa F_\delta( \theta_1) \\
	&\Lambda_1(\btheta)  = c_3 + \sqrt\eps G_\delta (\theta_1)\\
	&\Lambda_2(\btheta)  = c_4 - \sqrt\eps G_\delta (\theta_1)  \\
	&x_j(\btheta) =  0
\end{aligned}\right.
\label{eq:2dcoorb}
\ee
with $\btheta \in \TT^2$ and  $\sqrt{\eps}G_\delta = -\nu_{\delta}F'_\delta/2A$.  These objects will be called 2d co-orbital tori.
%
%%%%
In the expression (\ref{eq:2dcoorb}), the function $F_\delta$ and the frequency $\nu_\delta$ parametrize respectively the semi-fast horseshoe orbits and frequency as in (\ref{eq:HS_param}) and (\ref{eq:HS_param_H}), while the $c_j$  are real constant coefficients and $\kappa = m_1/(m_1+m_2)$.
%%%
On each torus, the flow is linear with the frequency
\bes
 	\dot {\btheta} 	=  \bom
 	 						 = (\nu_\delta,\upsilon)  \quad\mbox{where }  \nu_\delta \sim d_1{\frac{\sqrt\eps}{\ln\delta}} \mbox{ and } \upsilon\sim d_2
\ees
for some constants $d_1>0$ and $d_2>0$ when $\delta$ goes to zero.
%%%

%%%
%In order to get rid of the additional small parameter $\delta$, we constrain the latter to satisfy $\eps^{2\rpqh}\leq \delta \leq \eps^\rpqh$ for some exponent $\rpqh >0$. This yields:
%\be
% 	\dot{\btheta} 	= \bom
% 							= ( \nu_\delta,\upsilon ) = \gO\left( \frac{\sqrt{\eps}}{\modu{\ln \eps}}, \,1\right).
%\label{eq:tore_ref}
%\ee

%%%
By an application of KAM theory, we would like to continue the 2d co-orbital tori under the flow of the three-body problem carrying quasi-periodic trajectories with two frequencies: a semi-fast one, corresponding to the averaged motion,  and a fast one.
But to this end, the knowledge of the semi-fast dynamics is not enough, it is also necessary to control the dynamics in the directions that are normal to the 2d co-orbital tori.
These normal directions will be called secular directions.
Hence, we will consider the Hamiltonian
\bes
	\sHt_1(Z_1, \zeta_1) + \sH_2(Z_2) + \sQt(\zeta_1,\bx, \bxt)
\ees
where $\sQt$ is a suitable approximation of the secular dynamics, which is quadratic in the eccentricity variables thanks to the conservation of the D'Alembert rule given by (\ref{eq:DAl}).

$\sQt$ characterizes the linear stability of the $\rC_0$-manifold  in the normal directions $(\bx,\bxt)$ via the derived variational equations in the eccentricity variables.
However, as $\sQt$ is $\zeta_1$-dependent, the variational equations are time-dependent along a 2d co-orbital tori which prevents to express their solutions in a close form.
Hence, we shall transform our  system of canonical coordinates in order to uncouple the semi-fast and secular dynamics, and express the Hamiltonian in a suitable normal form.

\subsection{Reduction to a suitable normal form}
In the horseshoe region, the semi-fast angle $\zeta_1$ does not evolve in a monotonous way which prevents to remove directly the $\zeta_1$-dependency via a second averaging process.

Using the classical integral formulations (see Section \ref{sec:mech_syst}), we will build  semi-fast action-angle variables adapted to  horseshoe trajectories.
Hence we shall prove that there exists a canonical transformation
\bes
\Psi : (J_1, Z_2, \phi_1, \zeta_2, \bx, \bxt) \longmapsto (\bZ, \bzeta, \bx, \bxt)
\ees
such that the considered Hamiltonian reads
\bes
\begin{gathered}
\sH_1(J_1) + \sH_2(Z_2) + \sQ(J_1,\phi_1, \bx, \bxt) \\
\mbox{where} \quad \sH_1(J_1) = \sHt_1(Z_1(J_1, \phi_1), \zeta_1(J_1, \phi_1)),\\
\frac{\rd\sH_1}{\rd J_1}(J_1) = \nu_{\delta(J_1)},
\qtext{and}\sQ(J_1,\phi_1, \bx ,\bxt) = \sQt(\zeta_1(J_1,\phi_1), \bx, \bxt).
\end{gathered}
\ees
%
%The method is based on asymptotic expansions as the trajectory gets closer to the separatrix  (equivalently $\delta$ tends to zero).
However, it gives rise to an important drawback: the expressions of the semi-fast dynamics are no longer explicit which bring additional difficulties in our forthcoming application of KAM theory.

Then, we will proceed to a second averaging process (over the semi-fast angle $\phi_1$ in that case) in order to reject the $\phi_1$-dependency in a general remainder. The method is based on asymptotic expansions as the trajectory gets closer to the separatrix. For the sake of simplicity, as $\delta$ tends to zero, we constrain the energy shift $\delta$ by the relation $\eps^{2\rpqh}\leq \delta \leq \eps^\rpqh$ for some exponent $\rpqh >0$.
We shall prove that there exists a symplectic transformation
\bes
\Psib : (\Js_1, \Zs_2, \phis_1, \zetas_2, \bxs, \bxts) \longmapsto (J_1, Z_2, \phi_1,\zeta_2, \bx, \bxt)
\ees
close to identity and such that, in these variables, the considered Hamiltonian becomes
\bes
\begin{gathered}
\sH(\Js_1) + \sH_2(\Zs_2) + \sQb(\Js_1,\bxs, \bxts) + \sH_*(\Js_1, \Zs_2, \phis_1, \bxs, \bxts)\\
\mbox{where}
\quad	\sQb(\Js_1, \bxs, \bxts ) =\frac{1}{2\pi}\int_0^{2\pi} \sQ(\Js_1,\phis_1, \bxs, \bxts) \rd\phis_1
\end{gathered}
\ees
and $\sH_*$ is supposed to be small with respect to $\sQb$.
More precisely, we will iterate the averaging process until the semi-fast component
$\sHd_*=\sH_* - \sHbast$, where
$$
 \sHbast(\Js_1, \Zs_2, \bxs, \bxts ) =\frac{1}{2\pi}\int_0^{2\pi} \sH_*(\Js_1, \Zs_2, \phis_1, \bxs, \bxts) \rd\phis_1
$$
is exponentially small with respect to $\eps$.

Finally, from the secular Hamiltonian $\sHb$ that reads
\bes
\begin{split}
\sHb(\Js_1, \Zs_2, \bxs, \bxts) = \sH_1(\Js_1) + \sH_2(\Zs_2) &+ \sQb(\Js_1, \bxs, \bxts)+ \sHbast(\Js_1, \Zs_2, \bxs, \bxts)
\end{split}
\ees
we shall deduce the linear stability of the 2d co-orbital tori considered in the formulas (\ref{eq:2dcoorb}).
Indeed, from the conservation of the D'Alembert rule given by (\ref{eq:DAl}) and of the integral $\cC(\Zs_2, \bxs, \bxts)$ given by \eqref{eq:C_int}, we will control the remainder $\sHbast$ that reads
\bes
\sHbast(\Js_1, \Zs_2,\bxs, \bxts) = \sF_0(\Js_1, \Zs_2) + \sum_{j,k\in\{1,2\}} \sF_{(j,k)}(\Js_1, \Zs_2,\bxs,\bxts)\xs_j \xts_k
\ees
and obtain the spectrum of the second order terms in the eccentricities, that is
\bes
\sQb(\Js_1, \bxs, \bxts) + \sum_{j,k\in\{1,2\}} \sF_{(j,k)}(\Js_1, \Zs_2,\bzero, \bzero)\xs_j \xts_k\, .
\ees
The spectrum being simple with purely imaginary eigenvalues associated with the two secular frequencies
\bes
\bgomega = (g_1, g_2)= \cO\left(\eps,\, \frac{\eps}{\modu{\ln \eps}}\right),
\ees
then we will prove that the considered 2d co-orbital tori are normally elliptic.
As a consequence, we will consider the normal form
\bes
\sH_1(\Js_1) + \sH_2(\Zs_2) + \mathscr{F}_0(\Js_1,\Zs_2) + \sum_{j\in\{1,2\}} i g_j(\Js_1,\Zs_2) z_j\zt_j
\ees
where $(\bz, \bzt)$ are the eccentricity variables that diagonalize the secular Hamiltonian $\sQb + \sHbast$.

\subsection{2d tori for the full Hamiltonian in the horseshoe domain}

%%%
Now, let us see how the quasi-periodic orbits associated with the horseshoe region can be built in the full problem, and what will they look like.

As described above, we will develop an integrable approximation of the problem which will enable us to uncouple the fast, semi-fast, and secular dynamics.
It will be proved that, if we choose a horseshoe orbit which lies on the quasi-circular manifold $\rC_0$ and is close enough to the $L_3$-separatrix, its two frequencies (fast and  semi-fast) are respectively of order $(1,\, \sqrt{\eps}/\modu{\ln \eps})$ while it is normally elliptic along the two transversal directions with  frequencies of order $(\eps,\, \eps/\modu{\ln \eps})$.
This yields  four different timescales that will prevent the occurrence of small divisors for $\eps$ small enough.
As a consequence,  we will take advantage of this property, which will be used  to fulfill the Melnikov condition on the frequency map required to apply \cite{1996Po}'s theorem,  to get  2-dimensional tori associated with the horseshoe orbits in the three-body problem.

\begin{theorem}
\label{maintheorem}
There exists a real number $\eps_{*}>0$ such that for all $\eps$ with $0<\eps<\eps_{*}$, the Hamiltonian
flow linked to the planetary Hamiltonian $\cH$ given by \eqref{eq:ham_plan} admits an invariant set which is an union of 2-dimensional $\sC^{\infty}$ invariant tori carrying quasi-periodic trajectories.
These tori  are close, in $\sC^{0}$-topology, to the 2d co-orbital tori introduced above.

\end{theorem}

The quasi-periodic trajectories that come from this application of KAM theory can be described as follows:
\bes
\left\{\quad\begin{aligned}
\lam_1(\btheta) &= c_1 + \theta_2 + (1 - \kappa) F_\delta( \theta_1) + f_1(\btheta; \eps) \\
\lam_2(\btheta) &= c_2 +  \theta_2  - \kappa F_\delta( \theta_1) + f_2(\btheta; \eps) \\
\Lambda_1(\btheta)  &= c_3 + \sqrt{\eps} G_\delta( \theta_1) + f_3(\btheta; \eps) \\
\Lambda_2(\btheta)  &= c_4 + \sqrt{\eps} G_\delta( \theta_1) + f_4(\btheta; \eps) \\
x_j(\btheta) &= f_{5,j}(\btheta; \eps)\,
\end{aligned}\right.
\ees
where  $\dot{\btheta} 	= \bom = ( \nu_\delta,\upsilon )$ and $f_j$ are small perturbative terms.   Indeed, there exists  $C\geq 1$ independent of the small parameters of the problem such that:
\bes
\norm{f_1}\leq C \eps^{\gamma_1},\quad \norm{ f_2}\leq C \eps^{\gamma_1}, \quad\norm{f_3}\leq C \eps^{\gamma_2}, \quad \norm{f_4}\leq C\eps^{\gamma_2}, \quad \norm{f_{5,j}}\leq C \eps^{\gamma_3}
\ees
for the supremum norm on our domain with real exponents $\gamma_j$ such that
\bes
7/40 < \gamma_1 < \gamma_2 <\gamma_3 < 3/4.
\ees
In this expression, the functions $F_\delta$ and $G_\delta$ parametrize the semi-fast horseshoe orbits as in (\ref{eq:2dcoorb}).

%%% MAJ 2018-04-10 (Philippe Version)

%%% Sec3 Notations %%%%%%%%%
%%%%%%%%%%%%%%%%%%%%%%%%%%%%%%
% COORBKAMPLUS
% SEC 2 MAIN RESULT REDUCTION
% MAJ: 2018-04-24 ALEX
%%%%%%%%%%%%%%%%%%%%%%%%%%%%%%

\section{Notations}
\label{sec:notations}

%%% Pour appliquer KAM: approximation integrable %%%
In order to apply KAM theory, we need an integrable approximation of the Hamiltonian of the problem associated with the horseshoe motion and whose frequency map satisfies non-degeneracy properties.
%%%%%%%%%%%%%%%%%%%%%%%%%%%%%%%%%%

%%% Not %%%%%%%%%%%%%%%%%%%%%%%%%%%%
Before going further, let us introduce some useful notations.
%%%%%%%%%%%%%%%%%%%%%%%%%%%%%%%%%%

\medskip

First of all, the vector $(0,0)$ will be denoted $\bzero$ while $\Ent(x)$ will denote the floor function of a real number $x$.

Moreover, for $\bz\in\CC^n$, $\re(\bz)\in\RR^n$, $\im(\bz)\in\RR^n$  are the vectors corresponding respectively to the real part and the imaginary part of $\bz$.
Finally, the magnitude $\modu{\,\cdot\,}$ of the complex vector  $\bz$ is the supremum norm of the magnitude on each complex coordinate, that is
\bes
\modu{\bz}= \sup_{j\in\{1, \ldots, n\}} \modu{z_j}.
\ees

	\subsection{Complex domains and norms}
Let $\cD$ a subset of $\CC^n$ and $f$ a function in $\sC(\cD, \CC^m)$. Then, we will denote  $\norm{f}_{\cD}$ the  supremum norm of $f$ on the domain $\cD$  such that
\bes
	\norm{f}_{\cD}= \sup_{\bz\in\cD} \modu{f(\bz)}.
\ees

Now, let $\cU$ a subset of $\RR^n$.  We will define its associate complexified domain of width $r>0$ such that
\bes
	\cB_r\cU= \left\{ \bz \in \CC^n \,/\,\exists \bx_0 \in \cU  \mbox{ such that } \modu{\bz - \bx_0}\leq r \right\}
={\displaystyle{\bigcup\limits_{\bx_0\in\cU}}}\cB_r\{\bx_0\}
\ees
where for a real vector $\bx_0\in\RR^n$, $\cB_r\{\bx_0\}$ is the complex closed ball of radius $r>0$ centered on $\bx_0$.
Hence, we will denote
\bes
	\cB_r^n = \cB_r\{\bzero_n\}
\ees
the closed ball of  radius  $r>0$ centered on the origin in $\CC^n$.

Let $\cU$ a subset of $\TT^n$. We will define its associated complexified domain of width $s$ such that
\bes
	\cV_s \cU = \left\{ \bz \in\CC^{n} \,/\,  \re(\bz) \in \cU, \quad \modu{\im(\bz)}\leq s \right\}.
\ees

\medskip

%%% Domain complexe + norme %%%%%%%%%
For an interval $ \cS = \left[a, b\right]\subset{\mathbb R}$ and a set $\cIh\subset\TT^2$, if $\rho>0$ and  $\sig>0$, we define the complex domain $\cKh_{\rho,\sig}$ as follows:
\bes
	\cKh_{\rho,\sig}= \cB_\rho\cS\times\cB^1_\rho \times \cV_\sig\cIh \times \cB^4_{\sqrt{\rho\sig}}.
\ees
%%%%%%%%%%%%%%%%%%%%%%%%%%%%
%%% Domaine plus petit %%%%%%%%%%%%%
For $0<p\leq 1$, we also consider the domains
\bes
	\cKh_p = \cKh_{p\rho, p\sig}
\ees
and define the supremum norm on these latter as follows:
\bes
	\norm{\,\cdot\,}_p=\norm{\,\cdot\,}_{\cKh_{p\rho, p\sig}}.
\ees
%%%%%%%%%%%%%%%%%%%%%%%%%%%%

%%% Anisotropic %%%%%%%%%%%%%%%%%%
We will  need to consider the case of anisotropic analyticity widths where for $\brho=(\rho_1, \rho_2)\in \RR_+^*\times \RR_+^*$
and  $\bsig=(\sig_1, \sig_2)\in \RR_+^*\times \RR_+^*$.
The complex domain $\cK_{\brho, \bsig}$ is defined as follows:
\bes
	\cK_{\brho, \bsig} = \cB_{\rho_1}\cS\times\cB^1_{\rho_2} \times \cV_{\sig_1}\TT \times \cV_{\sig_2}\TT\times \cB^4_{\sqrt{\rho_2\sig_2}}
\ees
and its restriction $\cK_{\brho,\bsig,r}$, such as
\bes
\cK_{\brho, \bsig,r} =
\cB_{\rho_1}\cS\times\cB^1_{\rho_2} \times \cV_{\sig_1}\TT \times \cV_{\sig_2}\TT\times \cB^4_r
\ees
for $0<r\leq \sqrt{\rho_2\sig_2}$.
%%% Plus petit %%%%%%%%%%%%%%%%%%%
Thus, for $0<p\leq 1$, we also consider the domains
\bes
\cK_p= \cK_{p\brho, p\bsig} \qtext{and} \cK_{p,r}= \cK_{p\brho, p\bsig, r}
\ees
with the supremum norms
\bes
 \norm{\, \cdot\,}_p = \norm{\, \cdot\,}_{\cK_p} \qtext{and} \norm{\, \cdot\,}_{p,r}= \norm{\, \cdot\,}_{\cK_{p,r}}.
\ees
%%%%%%%%%%%%%%%%%%%%%%%%%%%%

%%% Norme Fonctions %%%%%%%%%%%%%%
Finally, for $1\leq k\leq +\infty$  and a given function $f\in\sC^{k}(\cU,\CC^m)$ where $\cU$ is a compact set in $\CC^n$, we define the $\sC^{k}$-norm $\norm{f}_{\sC^{k}}$ on $\cU$ such that
\be
	\norm{f}_{\sC^{k}} = \sup_{p \leq k}\norm{\frac{\partial^{p} f}{\partial z_1^{p_1} \ldots \partial z_n^{p_n}}}_\cU
\ee
with $(p_j)_{j\in\{1,\ldots,n\}} \in \NN^n$ and $p = \sum_{j=1}^n p_j$.

\subsection{Estimates}
\label{sec:notations_estim}
%%% Estimés %%%%%%%%%%%%%%%%%%%%
In the sequel we do not attempt to obtain estimates with particularly sharp constants.
Actually, we suppress all constants and use the notation
\bes
	x \leqp y, \quad x\pleq y, \qtext{and} x\eqp y
\ees
to indicate respectively that
\bes
x<Cy, \quad Cx<y, \qtext{and} x=Cy
\ees
with some constant $C\geq 1$ independent of the small parameters of the problem.
%%%%%%%%%%%%%%%%%%%%%%%%%%%%

\subsection{Derivatives}
%%% Dérivée %%%%%%%%%%%%%%%%%%%%
Let us now introduce several simplified notations about the derivatives.
Let a function $f(z)$ with $z\in\CC$, we will denote
\bes
f'(z) = \frac{\rd f}{\rd z}(z)
\qtext{and}
f^{(l)}(z) = \frac{\rd^lf}{\rd z^l} (z).
\ees

Let $f(\bw,\bx,\by, \bz)$  a multi-variable function of $\CC^8$ with  $\bw = (w_1,w_2)$, $\bx = (x_1,x_2)$, $\by = (y_1,y_2)$, and $\bz = (z_1,z_2)$. Then, we will denote the partial derivatives
\bes
\begin{aligned}
	&\partial_{w_1}f = \frac{\partial f}{\partial w_1},&\quad
	&\partial_{w_1}^l f = \frac{\partial^l f}{\partial w_1^l},&\\
	&\partial_\bw f = (\partial_{w_j}f)_{j\in\{1,2\}},&\quad
	&\partial_\bw\partial_\bx f= \left(\frac{\partial^2f}{\partial w_j\partial x_k }\right)_{j,k\in\{1,2\}},&\\
	&\partial_{\bw}^2 f= \partial_{\bw}\partial_{\bw}= \left(
		\frac{\partial^2f}{\partial w_j\partial w_k}\right)_{j,k\in\{1,2\}},\quad&
	&\partial_{(\bz,\bw)} f= (\partial_{\bz}f,\partial_{\bw}f),&	
\end{aligned}
\ees
and
\bes
\begin{split}
\partial_{(\bz,\bw)}^2f &= \partial_{(\bz,\bw)}\partial_{(\bz,\bw)}f=
\begin{pmatrix}
\partial_{\bz}\partial_\bz f & \partial_{\bz}\partial_{\bw} f\\
\partial_{\bw}\partial_{\bz} f & \partial_{\bw}\partial_\bw f
\end{pmatrix}.
\end{split}
\ees

Finally, the differential of the function $f$ will be denoted
\bes
\rd f = (\partial_\bw f, \partial_\bx f, \partial_\by f, \partial_\bz f)
\qtext{and}
\rd^lf= \underbrace{\rd\rd\ldots\rd}_{l}f.
\ees

\subsection{Hamiltonian flow}
%%%% Necessaire lemme iteratif de la première moyenne

The Hamiltonian flow at a time $t$ generated by an auxiliary function $g(\bw,\bx,\by,\bz)$ will be denoted $\Phi_t^g(\bw,\bx,\by,\bz)$.
By introducing the Poisson bracket of the two real functions $f(\bw,\bx,\by,\bz)$ and $g(\bw,\bx,\by,\bz)$,  such as
\bes
\Poi{f}{g}  = \partial_{\bw} f \bigcdot \partial_{\bx} g - \partial_{\bw} g \bigcdot \partial_{\bx} f
		  + \partial_{\by} f \bigcdot \partial_{\bz} g - \partial_{\by} g \bigcdot \partial_{\bz} f,
\ees
then the Hamiltonian flow satisfies
\bes
	\frac{\rd}{\rd t}(f\circ \Phi^g_t) = \Poi{f}{g}\circ \Phi_t^g
\ees
and thus the Taylor expansions
\begin{align}
f\circ\Phi_1^g 	&= f + \int_0^1\Poi{g}{f}\circ\Phi_s^g \rd s\label{eq:Taylor0}\qtext{and}\\
f\circ\Phi_1^g	&= f + \Poi{g}{f} + \int_0^1(1-s)\Poi{g}{\Poi{g}{f}}\circ\Phi_s^g \rd s \label{eq:Taylor1}.
\end{align}

%%% Sec4 Construction of an adapted integrable approximation
% EstimÃ©s sur HK et HP et D'Alembert
% PremiÃšre forme normale (Averaging)
\section{Reduction of the Hamiltonian}
\label{sec:red_sec}

%%% Goal %%%%%%%%%%%%%%%%%%%%%%%%%%%
The main goal of this section is to reduce the planetary Hamiltonian
	$H(\bZ, \bzeta, \bx, \bxt) = H_K(\bZ) +H_P(\bZ, \bzeta, \bx, \bxt)$
defined in Section \ref{sec: r11} to the sum of two terms: an integrable Hamiltonian  associated with horseshoe trajectories written in terms of action variables and a remainder whose size is controlled.
Several steps are necessary.

\subsection{A collisionless domain }
\label{sec:collisionless}

%%% Domaine complexe %%%%%%%%%%%%%
First of all, let us  define a complexified domain excluding the collision manifold where the sizes of the Keplerian and perturbation parts will be estimated.
%%%%%%%%%%%%%%%%%%%%%%%%%%%%

%%%  ensemble I %%%% %%%%%%%%%%%%%
For an arbitrary fixed $\Deltah >0$, \textit{independent of the small parameters of the problem}, we define the set $\cIh$ by
\bes
	\cIh = \left\{\bzeta\in  \TT^2  /\, \modu{\zeta_1}\geq\Deltah \right\}
\ees
where ``$\modu{\,\cdot\,}$" denotes the usual distance over the quotient space $\TT=\RR/2\pi\ZZ$.
Remark that the condition on $\zeta_1$ can also be considered with the real variable $\zeta_1\in [\Deltah , 2\pi -\Deltah ]$ since there exists an unique real representative in this segment for an angle $\zeta_1$ with a modulus lowered by $\Deltah$.
Hence, $\cIh$ has the structure of a cylinder in $\RR\times\TT$.
%%%%

If we assume that the planets are on circular exact-resonant orbits ($\bZ =\bx = \bxt= \bzero$; see Section \ref{sec: r11}), the fixed quantity $\Deltah$ corresponds to the minimal angular separation between the two planets which yields to a minimal distance given by
	$\displaystyle\Delta = 2 \min (a_{1,0}, a_{2,0})\sin(\Deltah/2)$.
%%%
Thus, with the notations of Section \ref{sec:notations}, for an arbitrary $\Deltah>0$ independent of the small parameters of the problem, $\rho>0$ and $\sig>0$ small enough that will be specified in the sequel, we can define a complex domain of holomorphy
\bes	
\cKh_{\rho,\sig}= \cB^2_{\rho} \times \cV_{\sig}\cIh \times \cB^4_{\sqrt{\rho\sig}}
\ees
that excludes the collision manifold.
%%% Estimés %%%%%%%%%%%%%%%%%%%%%
In this setting, it will be possible to estimate the size of the transformations and the functions involved in our  resonant normal form constructions.
%%%%%%%%%%%%%%%%%%%%%%%%%%%%%

Hence, we set out the following
%%% Theorem 2: estimés de HK et HP %%%%%%
\begin{lemma}[Estimates on $H_K$ and $H_P$]
	\label{Th:HKHP}
	Assuming that
	\bes
		0<\rho_0 <\sig_0\, , \quad
		\rho_0\pleq 1, \qtext{and}
		\sig_0 \pleq \Deltah,
 	\ees
 the Hamiltonian $H$ is analytic in the domain $\cKh_{\rho_0,\sig_0}$ and satisfies:
	\be
		\norm{H_K}_{\sC^{3}} \leqp 1,\quad
		\norm{H_P}_{\sC^{4}} \leqp\eps .	
		\label{eq:HKHP_bounds}
	\ee
\end{lemma}
%%%%%%%%%%%%%%%%%%%%%%%%%%%%%%%

	\subsection{First averaging}
	\label{sec:first_normal_form}
%%% Goal Averaging 1 %%%%%%%%%%%%%%%%%%
In the first step of the reduction scheme, we average the Hamiltonian $H$ over the fast angle $\zeta_2$ in order to reject the $\zeta_2$-dependency in an exponentially small remainder.
This reduction is provided by the following
%%% Theorem %%%%%%%%%%%%%%%%%%%
\begin{theorem}[First Averaging Theorem]
	\label{Th:Moy1}
	
	For
 	\bes
 		1/7 <\beta < 1/2,\quad (\rho, \sig) = \sig_0(\eps^\beta, 1),
 	\ees
and $\eps$ small enough (i.e. $\eps \pleq 1$), there exists a canonical transformation
\begin{gather}
		\Upsb : \quad \Bigg\{
		\begin{array}{ccc}
			\cKh_{1/3}                				&\longrightarrow 	& \cKh_{1} \\
			(\bZs, \bzetas, \bxs, \bxts)  	&\longmapsto    	&(\bZ,\bzeta,\bx,\bxt )
		\end{array}\nnb\\
\mbox{with}\quad	\cKh_{1/6} \subseteq \Upsb(\cKh_{1/3}) \subseteq \cKh_{1/2}\label{eq:Moy1_emboitement}
\end{gather}
	and
	such that
\be
	\label{eq:Moy1_HamF}	
	\begin{gathered}
		H\circ\Upsb(\bZs,\bzetas, \bxs, \bxts) =
		H_K(\bZs) +\Hb_P(\bZs,\zetas_1,\bxs, \bxts) + H_*(\bZs,\bzetas,\bxs, \bxts) \\
		\mbox{where}\quad
		\Hb_P(\bZs, \zetas_1,  \bxs, \bxts ) 	 =
		\frac{1}{2\pi}\int_0^{2\pi} H_P(\bZs_j, \zetas_1, \zetas_2,\bxs, \bxts) \rd\zetas_2
	\end{gathered}
\ee
	with the following estimates:	
\begin{align}
			\norm{\Hbast}_{1/3} 	&\leqp \eps^{2-\beta},
		\label{eq:Moy1_rem_moy}\\
			\norm{\Hd_*}_{1/3} &\leqp \eps\exp(-\frac{1}{\eps^\alpha}),
		\label{eq:Moy1_rem_exp}
\end{align}	
for
\bes
	\begin{gathered}
		\Hbast(\bZs,\zetas_1, \bxs, \bxts)
		= \frac{1}{2\pi}\int_0^{2\pi} H_*(\bZs,\zetas_1, \zetas_2, \bxs, \bxts)\rd\zetas_2, \quad  \Hd_* = H_* - \Hbast,
	\end{gathered}
 \ees
and 	
\bes
		\alpha = \frac{1 - 2\beta}{5}.
%		\label{eq:Moy1_alpha}
\ees
%
	%%%	
	Moreover, the size of the transformation $\Upsb$ is given by
\be
	\begin{array}{ll}
	\norm{\bZs - \bZ}_{1/3} \leqp \eps,& \quad	
	\norm{\bzetas - \bzeta}_{1/3} \leqp \eps^{1-\beta},\\
	\norm{(\bxs,\bxts) - (\bx,\bxt)}_{1/3} \leqp \eps^{1-\beta/2}.&
	\end{array}	
	\label{eq:Moy1_transf}
\ee
\end{theorem}
%%%%%%%%%%%%%%%%%%%%%%%%%

%%% Reorganisation du Ham %%%%%%%%%
The remainder $\Hd_*$ being exponentially small on $\cKh_{1/3}$, we choose to drop it for the moment in order to focus our reduction on the averaged Hamiltonian given by
\bes
\Hb(\bZs, \zetas_1, \bxs, \bxts) = H_K(\bZs) + \Hb_P(\bZs, \zetas_1, \bxs, \bxts) + \Hbast(\bZs, \zetas_1, \bxs, \bxts).
\ees
%%%%%%%%%%%%%%%%%%%%%%%%%%

%%% D'alembert %%%%%%%%%%%%%%%%%
At last, we have the following crucial property.
\begin{lemma}[D'Alembert rule in the Averaged Problem]
\label{lem:DAl}
 We choose the transformation $\Upsb$  such as the  D'Alembert rule, given by (\ref{eq:DAl}), is preserved (see Lemma \ref{Lem:Moy1}).  Equivalently, the quantity
\be
	\cC(\Zs_2, \bxs, \bxts) = \Zs_2 + i\xs_1\xts_1 + i\xs_2\xts_2,
	\label{eq:Dal_integral_moy}
\ee
associated with the angular momentum $\cCt(\brt_j, \br_j)$, is a first integral of the averaged Hamiltonian $\Hb$.

Furthermore, if a general function $f$, which does not depend on the fast angle $\zetas_2$, satisfies the D'Alembert rule, then there exists a set of function $(f_{(j,k)})_{j,k\in\{1,2\}}$  such that
\be
\begin{gathered}
\begin{aligned}
 f(\bZs, \zetas_1, \bxs, \bxts)&= f_0(\bZs, \zetas_1)+ f_2(\bZs, \zetas_1, \bxs, \bxts)\\
&=f_0(\bZs, \zetas_1)+ \sum_{j,k\in\{1,2\}} f_{(j,k)}(\bZs, \zetas_1, \bxs, \bxts)\xs_j\xts_k
 \end{aligned}\\
 \begin{aligned}
	&\mbox{with}\quad f_0(\bZs, \zetas_1) = f(\bZs, \zetas_1, \bzero, \bzero) \\
 	&\mbox{and} \quad f_{(j,k)}(\bZs, \zetas_1, \bxs, \bxts) - f_{(j,k)}(\bZs, \zetas_1, \bzero, \bzero)  = \cO_2(\norm{(\bxs, \bxts)}).
\end{aligned}
\end{gathered}
\label{eq:Dal_f}
\ee

\end{lemma}
%%%%%%%%%%%%%%%%%%%%%%%%%%%%%

%%% C0 invariant manifold %%%%%%%%%%%%%
This implies that in the expansion of $\Hb$ in the neighborhood of $\{\bxs=\bxts=\bzero\}$  the total degree in $\xs_1$, $\xs_2$, $\xts_1$, $\xts_2$  in the monomials appearing in the  associated Taylor series is even.
Hence, the quasi-circular manifold defined as follows:
\bes
\rC_0 =\left\{ (\bZs, \bzetas, \bxs, \bxts)\in\RR^2\times\TT^2\times\CC^4 \,/\, \bxs=\bxts=\bzero\right\}
%\label{quasi_circ}
\ees
is invariant by the flow of the Hamiltonian $\Hb$.
For more details on the topology of $\rC_0$, see the phase portrait and its description in Section \ref{sec:semi_HS}.

%%%%%%%%%%%%%%%%%%%%%%%%%%%%%%%

%%% Sec4 Construction of an adapted integrable approximation (suite)
% Reduction de Hd
% ActionAngle de H1
%%%%%%%%%%%%%%%%%%%%%%%%%%%%%%
% COORBKAMPLUS
% SEC 4 RED
% MAJ: 2018-04-18 ALEX
%%%%%%%%%%%%%%%%%%%%%%%%%%%%%%	
	
	\subsection{The reduction of $\Hb$}
	\label{sec:red}
	
%%% Goal reduction %%%%%%%%%%%%%%%%%
In the second step, we perform  some reductions in order to get a more tractable expression of the Hamiltonian $\Hb$.
It mainly consists in an expansion of the Hamiltonian in a suitable domain and at an appropriate degree.
%%%%%%%%%%%%%%%%%%%%%%%%%%%%%%%

%%% Ecc %%%%%%%%%%%%%%%%%%%%%%%%%
First of all, regarding the eccentricities, a polynomial expansion of degree two in $\bxs=\bxts = \bzero$ is enough to control the dynamics along the secular directions (i.e. the directions transversal to $\rC_0)$.
%%%%%%%%%%%%%%%%%%%%%%%%%%%%%%%

%%% Actions 1 %%%%%%%%%%%%%%%%%%%%%%%
For the action variables $\bZs$, it is natural to expand in the neighborhood of the exact-resonant actions $(\Lam_{1,0},\Lam_{2,0})$ given in the formula (\ref{eq:exact_res}), that is $\bZs = \bzero$.
Thus, $\Hb_P$ is truncated at degree zero while it is necessary to keep the second order for the Keplerian part.

%%% Actions 2 %%%%%%%%%%%%%%%%%%%%%%%
However, coming from the fact that when $\bZs= \bZ_\star$ with
\bes
	\bZ_\star =  \begin{pmatrix}
	\Lam_{1,0}-\Lam_{1,\star}\\
	\Lam_{1,0}+\Lam_{2,0} -(\Lam_{1,\star}+\Lam_{2,\star})	
	\end{pmatrix}
	\qtext{and} \Lam_{j,\star} = \hm_j\mu_j^{1/2}m_0^{1/6}\upsilon_0^{-1/3},
\ees
the two associated semi-major axes are both equal to the same value given by	
$a_\star= m_0^{1/3}\upsilon_0^{-2/3}$. Consequently,  it is much more convenient to center the expansion of $\Hb_P$ at $\bZs = \bZ_\star$.
This shift generates only small additional terms, as the difference between $\Lam_{j,0}$ and $\Lam_{j,\star}$ satisfies the inequalities
$		0< \Lam_{j,0} - \Lam_{j,\star}  \leqp \eps$.
%
%%% Autre simplifications %%%%%%%%%%%%%%%%%%%
Remark that the reduced mass $\hm_j$ can be replaced by $m_j$ which only adds small terms to the remainder of order $\eps\norm{\bZs}^2$ from $H_K$ and of order $\eps^2$ from $\Hb_P$.
%%%%%%%%%%%%%%%%%%%%%%%%%%%%%%%%%%%

%%% Découplage fast-semifast %%%%%%%%%%%%%%%%
Finally, in order to uncouple the fast and semi-fast action variable $\Zs_1$ and $\Zs_2$, we introduce a new set of action-angle variables via the linear transformation $\Psit$ given by
\bes
		\Psit : \quad \Bigg\{
		\begin{array}{ccc}
				\cKh_{1/6} 				&\longrightarrow 	& \cKh_{1/3} \\
				(\bI,\bvphi,\bw, \bwt) 	&\longmapsto  		&(I_1 +\kappa I_2, I_2,\varphi_1, \varphi_2-\kappa\varphi_1, \bw,\bwt)
		\end{array}
\ees
where
\bes
		\kappa =\frac{m_1}{m_1+m_2}\leq \frac{1}{2} %\qtext{(given by \eqref{massehierarchie})}
\ees
and such that
\be
		\cKh_{1/18}\subseteq\Psit(\cKh_{1/6})\subseteq\cKh_{1/4}\;.\label{eq:Approx_emboitement}
\ee
%%%%%%%%%%%%%%%%%%%%%%%%%%%%%%%%%%%

%%% Theo3 %%%%%%%%%%%%%%%%%%%%%%%%%%%%
All of these successive transformations and estimation of the generate remainder are summarized in the following
\begin{theorem}[Hamiltonian Reduction]
\label{Th:Approx}
Under the assumptions of Theorem \ref{Th:Moy1}, we have the following assertions:
	\begin{enumerate}
		\item  in the coordinates $(\bI,\bvphi, \bw,\bwt)$ the averaged Hamiltonian $\sHt=\Hb \circ\Psit - H_{K}(\bzero) $ can be written
		\be
		\begin{split}
			\sHt(\bI, \varphi_1, \bw, \bwt) =   \sHt_1(I_1,\varphi_1) + \sH_2(I_2) &+ \sQt(\varphi_1, \bw, \bwt) + \sRt(\bI, \varphi_1, \bw, \bwt)
			\label{eq:Approx_reduc}
		\end{split}
		\ee	
		where
		\bes
			\begin{split}
				&\sHt_1(I_1,\varphi_1) 			= \upsilon_0 \left(-AI_1^2 + \eps  B\cF(\varphi_1)\right),\quad
				\sH_2(I_2)								=	\upsilon_0 \left(I_2-E I_2^2\right),\\
				&\sRt(\bI, \varphi_1, \bw, \bwt)=\sRt_0(\bI, \varphi_1)+ \sum_{j,k\in\{1, 2\}} \sRt_{(j,k)}(\bI, \varphi_1, \bw, \bwt)w_j\wt_k,
			\end{split}
		\ees	
		\bes
			\begin{split}
							\sQt(\phi_1,\bw,\bwt) &=  \sum_{j,k\in\{1, 2\}} \sQt_{(j,k)}(\phi_1)w_j\wt_k\\
								&= i\eps\upsilon_0  D\bigg(
					\frac{\cAt(\varphi_1)}{m_1}w_1\wt_1 	
				+ 	\frac{\cBt(\varphi_1)}{\sqrt{m_1m_2}} w_1\wt_2
			  \\ &\phantom{=}+   \frac{\conj(\cBt)(\varphi_1)}{\sqrt{m_1m_2}} \wt_1 w_2
				+	\frac{\cAt(\varphi_1)}{m_2}w_2\wt_2
				\bigg),
			\end{split}
		\ees	
		with
		\be
		\begin{split}
			&\cF(\varphi_1)	=\frac{2}{3} \left(\cos \varphi_1 - \cD(\varphi_1)^{-1}\right),\\
			&\cAt(\varphi_1) 	= \frac{\cD(\varphi_1)^{-5}}{4}\left(5\cos2\varphi_1  - 13 + 8\cos \varphi_1\right) - \cos \varphi_1, \\
			&\cBt(\varphi_1) 	=  e^{-2i\varphi_1}-	\frac{\cD(\varphi_1)^{-5}}{8}\left(e^{-3i\varphi_1} + 16 e^{-2i\varphi_1} - 26e^{-i\varphi_1} + 9 e^{i\varphi_1}\right),\\
		&\mbox{$\conj(\cBt)(\varphi_1)$ is the complex conjugate of $\cBt(\varphi_1)$, }\\
			&\cD(\varphi_1) = \sqrt{2-2\cos \varphi_1},
		\end{split}	\label{eq:Approx_Functions}	
		\ee
		and the parameters:
		\be
		\begin{split}
			&A= \frac{3}{2}\upsilon_0^{1/3}m_0^{-2/3}\left(\frac{1}{m_1} + \frac{1}{m_2}\right), \quad
			B= \frac{3}{2}\upsilon_0^{-1/3}m_0^{2/3} D,\\
			&D= \frac{m_1m_2}{m_0},\quad
			E = \frac{3}{2}\upsilon_0^{1/3}m_0^{-2/3}(m_1 + m_2)^{-1}.
		\end{split}
		\label{eq:Approx_param}
		\ee
		\item The remainder $\sRt$ is bounded by the threshold
		\bes
				\norm{\sRt}_{1/6} \leqp \eps^{3\beta}
		\ees
		and, if we assume $\beta>1/3$, we can ensure
		\bes
				\norm{\sRt_{(j,k)}}_{1/6} \leqp \eps^{2-2\beta} .
		\ees
\end{enumerate}
\end{theorem}

%%% Remarque: taille des restes %%%%%%%%%%%%%%%%%%%%%%
\begin{remark}
On the domain $\cKh_{1/6}$, the size of the remainder $\sRt$ is larger than the one provided by the first averaging given by \eqref{eq:Moy1_rem_moy}.
\end{remark}
%%%%%%%%%%%%%%%%%%%%%%%%%%%%%%%%%%%%%%%%%%

%%% Description du nouvel Ham %%%%%%%%%%%%%%%%%%%%%%
As a consequence, the averaged Hamiltonian $\sHt$ possesses three components.
The first one describes the dynamics on the quasi-circular manifold $\rC_0$.
It is composed of the integrable Hamiltonian $\sH_2$  and the mechanical system $\sHt_1$, respectively associated with the fast and the semi-fast variations.
The second component $\sQt$ which is of order two in eccentricity and depends of the semi-fast angle described the main part of the secular behavior along the two normal-directions.
At last, we have the remainder $\sRt$ whose shape and size  are controlled on the domain $\cKh_{1/6}$.
%%%%%%%%%%%%%%%%%%%%%%%%%%%%%%%%%%%%%%%%%%

\subsection{The mechanical system $\sHt_1$}
\label{sec:mech_syst}

%%% Mech. System Intro % %%%%%%%%%%%%%%%%%%%%%%%
In the third step, we focus our efforts on the semi-fast dynamics in order to build an action-angle coordinate system valid for the horseshoe trajectory region.
According to Theorem \ref{Th:Approx}, the semi-fast component of the Hamiltonian is given by the following mechanical system:
\bes
		\sHt_1(I_1, \varphi_1) = \upsilon_0\left(-AI_1^2 + \eps B \cF(\varphi_1)\right)
%		\label{eq:mech_sys}
\ees
where the real function $\cF$ is defined on $]0,2\pi[$ by \eqref{eq:Approx_Functions} and  $A$, $B$ are two positive constants given by \eqref{eq:Approx_param}.
%%%%%%%%%%%%%%%%%%%%%%%%%%%%%%%%%%%%%%%%%%
%%% Mech Sys. symétrie %%%%%%%%%%%%%%%%%%%%%%%%%%
As
\bes
\sHt_1(-I_1, \varphi_1) = \sHt_1(I_1, \varphi_1) \qtext{and} \sHt_1(I_1, \pi -\varphi_1) = \sHt_1(I_1, \pi + \varphi_1),
\ees
the study of the Hamiltonian $\sHt_1$ and its flow can be reduced to $\RR_+\times]0,\pi]$.
%%%%%%%%%%%%%%%%%%%%%%%%%%%%%%%%%%%%%%%%%%
%%% L3 Hyperbolic fixed point and separatrix %%%%%%%%%%%%%%
As
\bes
	\cF(\varphi_1) = -1 + \frac{7}{24}(\varphi_1 - \pi)^2 +\cO_4(\varphi_1-\pi),
\ees
the point of coordinates $(0,\pi)$ is a hyperbolic fixed point whose energy level equals $h_0 = \eps \upsilon_0 B \cF(\pi) =  -\eps  \upsilon_0B$.
%%%%%%%%%%%%%%%%%%%%%%%%%%%%%%%%%%%%%%%%%%

%%% HS definition %%%%%%%%%%%%%%%%%%%%%%%%%%%%%%%
With these notations, the horseshoe orbits that we are interested  in  are the level curves of $\sHt_1$ which fulfill the relation
\bes
	\sHt_1(I_1, \varphi_1) = h_\delta = - \eps \upsilon_0 B(1+\delta) \qtext{with} \delta >0.
\ees
%%%%%%%%%%%%%%%%%%%%%%%%%%%%%%%%%%%%%%%%%%

%%% Domaine Da: Figure %%%%%%%%%%%%%%%%%%%%%%%%%%%%%%%
\begin{figure}[h!]
	\begin{center}
	\def\svgwidth{1\textwidth}
%% Creator: Inkscape inkscape 0.92.2, www.inkscape.org
%% PDF/EPS/PS + LaTeX output extension by Johan Engelen, 2010
%% Accompanies image file 'PhaseBlanc.eps' (pdf, eps, ps)
%%
%% To include the image in your LaTeX document, write
%%   \input{<filename>.pdf_tex}
%%  instead of
%%   \includegraphics{<filename>.pdf}
%% To scale the image, write
%%   \def\svgwidth{<desired width>}
%%   \input{<filename>.pdf_tex}
%%  instead of
%%   \includegraphics[width=<desired width>]{<filename>.pdf}
%%
%% Images with a different path to the parent latex file can
%% be accessed with the `import' package (which may need to be
%% installed) using
%%   \usepackage{import}
%% in the preamble, and then including the image with
%%   \import{<path to file>}{<filename>.pdf_tex}
%% Alternatively, one can specify
%%   \graphicspath{{<path to file>/}}
%%
%% For more information, please see info/svg-inkscape on CTAN:
%%   http://tug.ctan.org/tex-archive/info/svg-inkscape
%%
\begingroup%
  \makeatletter%
  \providecommand\color[2][]{%
    \errmessage{(Inkscape) Color is used for the text in Inkscape, but the package 'color.sty' is not loaded}%
    \renewcommand\color[2][]{}%
  }%
  \providecommand\transparent[1]{%
    \errmessage{(Inkscape) Transparency is used (non-zero) for the text in Inkscape, but the package 'transparent.sty' is not loaded}%
    \renewcommand\transparent[1]{}%
  }%
  \providecommand\rotatebox[2]{#2}%
  \ifx\svgwidth\undefined%
    \setlength{\unitlength}{791.5663147bp}%
    \ifx\svgscale\undefined%
      \relax%
    \else%
      \setlength{\unitlength}{\unitlength * \real{\svgscale}}%
    \fi%
  \else%
    \setlength{\unitlength}{\svgwidth}%
  \fi%
  \global\let\svgwidth\undefined%
  \global\let\svgscale\undefined%
  \makeatother%
  \begin{picture}(1,0.49148058)%
    \put(0,0){\includegraphics[width=\unitlength]{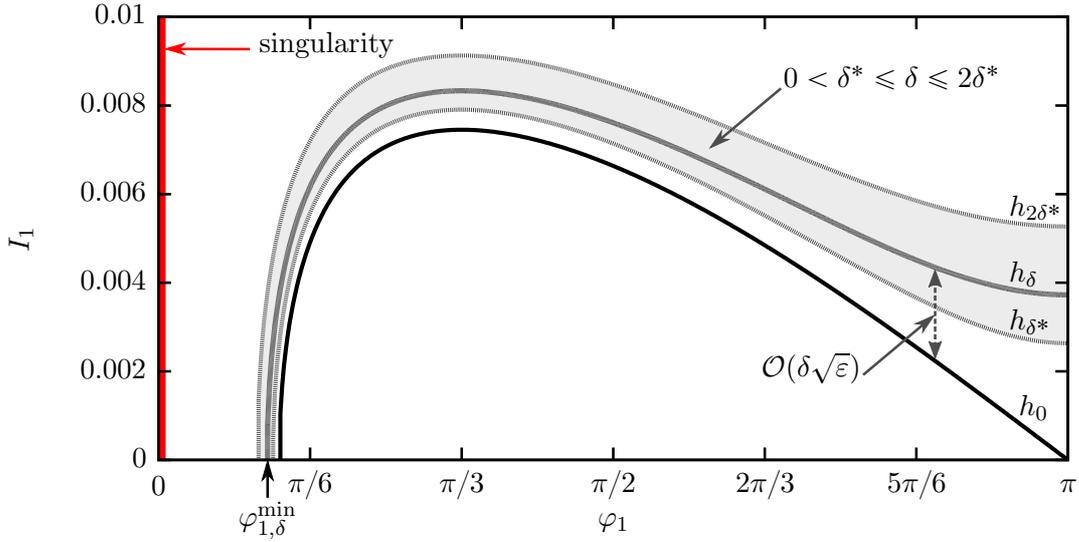}}%
    \put(0.12127301,0.05952265){\makebox(0,0)[rb]{\smash{0}}}%
    \put(0.13864554,0.03384459){\makebox(0,0)[rb]{\smash{0}}}%
    \put(0.12127301,0.14062766){\makebox(0,0)[rb]{\smash{0.002}}}%
    \put(0.12127301,0.22173268){\makebox(0,0)[rb]{\smash{0.004}}}%
    \put(0.12127301,0.30274295){\makebox(0,0)[rb]{\smash{0.006}}}%
    \put(0.12127301,0.38384797){\makebox(0,0)[rb]{\smash{0.008}}}%
    \put(0.12127301,0.46495299){\makebox(0,0)[rb]{\smash{0.01}}}%
    \put(0.26946021,0.03678291){\color[rgb]{0,0,0}\makebox(0,0)[b]{\smash{$\pi/6$}}}%
    \put(0.23106716,0.0053582){\color[rgb]{0,0,0}\makebox(0,0)[b]{\smash{$\varphi_{1,\delta}^{\min}$}}}%
    \put(0.54452812,0.00534843){\color[rgb]{0,0,0}\makebox(0,0)[b]{\smash{$\varphi_{1}$}}}%
    \put(0.01446731,0.26837452){\color[rgb]{0,0,0}\rotatebox{90}{\makebox(0,0)[b]{\smash{$I_{1}$}}}}%
    \put(0.40703555,0.03678292){\makebox(0,0)[b]{\smash{$\pi/3$}}}%
    \put(0.54461088,0.03678292){\makebox(0,0)[b]{\smash{$\pi/2$}}}%
    \put(0.68209147,0.03678292){\makebox(0,0)[b]{\smash{$2\pi/3$}}}%
    \put(0.8196668,0.03678292){\makebox(0,0)[b]{\smash{$5\pi/6$}}}%
    \put(0.95724213,0.03678292){\makebox(0,0)[b]{\smash{$\pi$}}}%
    \put(0.91216744,0.10613906){\color[rgb]{0,0,0}\makebox(0,0)[lb]{\smash{$h_0$}}}%
    \put(0.90337455,0.28947808){\color[rgb]{0,0,0}\makebox(0,0)[lb]{\smash{$h_{2\delta^*}$}}}%
    \put(0.90337455,0.18538042){\color[rgb]{0,0,0}\makebox(0,0)[lb]{\smash{$h_{\delta^*}$}}}%
    \put(0.90342883,0.22883881){\color[rgb]{0,0,0}\makebox(0,0)[lb]{\smash{$h_{\delta}$}}}%
    \put(0.76877842,0.14232916){\color[rgb]{0,0,0}\makebox(0,0)[rb]{\smash{$\cO(\delta\sqrt{\eps})$}}}%
    \put(0.69869555,0.40672517){\color[rgb]{0,0,0}\makebox(0,0)[lb]{\smash{$0<\delta^*\leq\delta\leq 2\delta^*$}}}%
    \put(0.2232501,0.43814453){\color[rgb]{0,0,0}\makebox(0,0)[lb]{\smash{singularity}}}%
  \end{picture}%
\endgroup%
	\caption{Phase portrait of the mechanical system $\sHt_1(I_1,\varphi_1)$ depicted in the domain $\RR_+\times]0,\pi[$ (the entire phase portrait is represented in Figure \ref{fig:separatrixEntier}). $\delta=0$ corresponds to the separatrix (black curve) that surrounds the tadpole domain ($\delta<0$) while $\delta >0$ corresponds to the horseshoe trajectories (grey curves).
	The grey region represents the  horseshoe trajectories such as $\delta^* \leq \delta \leq 2\delta^*$ where adapted action-angle variables are built.}
	\label{fig:separatrix}
	\end{center}
\end{figure}
%%%%%%%%%%%%%%%%%%%%%%%%%%%%%%%%%%%%%%%%%%
%
%%% Redefinition d'une courbe de niveau pour HS %%%%%%%%%%%%
Along a $h_\delta$-level curve, the action $I_1$ can be expressed as a function of $\varphi_1$ for $\varphi_{1,\delta}^{\min} \leq \varphi_1 \leq \pi$, where $\varphi_{1,\delta}^{\min}$ is the angle corresponding in one of the two intersections of the level curve with the axis $I_1 = 0$ (see Fig.\ref{fig:separatrix}).
%%%%%%%%%%%%%%%%%%%%%%%%%%%%%%%%%%%%%%%%%%
%%% Angle min %%%%%%%%%%%%%%%%%%%%%%%%%%%%%%%%
This angle, which is also the minimal value of  $\varphi_1$  along a $h_\delta$-level curve, verifies
\be
	\varphi_{1,\delta}^{\min} = 2 \arcsin\left(\frac{\sqrt{2} -1}{2}\right) - c_0\delta + \cO(\delta^2) \qtext{where} 1/4< c_0 <1/3.
	\label{eq:phi_min}
\ee
%%%%%%%%%%%%%%%%%%%%%%%%%%%%%%%%%%%%%%%%%%
%%% HS: Oscillation avec une large amplitude %%%%%%%%%%%%%%
By symmetry reasons, the  amplitude of variation of $\varphi_1$ around $\pi$, for a given value of $h_\delta$, is greater than $2\pi - 2\varphi_{1,\delta}^{\min} > 312^\circ$.
%%%%%%%%%%%%%%%%%%%%%%%%%%%%%%%%%%%%%%%%%%

%%% Period %%%%%%%%%%%%%%%%%%%%%%%%%%%%%%%%%%
When $\delta>0$, the orbit of energy $h_\delta$ is periodic and the corresponding period is given by the expression
\bes
\begin{gathered}
		T_\delta = \frac{2}{\upsilon_0 \sqrt{\eps A B}}\int_{\varphi_{1,\delta}^{\min}}^\pi \frac{\rd\varphi_1}{\sqrt{U_\delta(\varphi_1)}} \qtext{with}\\
		\label{eq:period_delta}
		U_\delta(\varphi_1) =\eps^{-1}AB^{-1}( I_1(\varphi_1))^2 = 1 + \delta + \cF(\varphi_1).
\end{gathered}
\ees
%%%%%%%%%%%%%%%%%%%%%%%%%%%%%%%%%%%%%%%%%%

%%% Comportement asymptotique de la fréquence %%%%%%%%%%%
As the orbit approaches the separatrix ($\delta$ tends to zero), its period $T_\delta$ tends to infinity.
More precisely,

\begin{lemma}[Semi-fast Frequency]
\label{lem:semifast}
if $0<\delta^*$ is small enough ($\delta^*\pleq 1$), then for all $\delta\in[\delta^*, 2\delta^* ]$ the asymptotic expansions of the semi-fast frequency and of its first derivative read
\be
	\nu_\delta 	= 2\pi T_\delta^{-1}
					=  \frac{\upsilon_0  \sqrt{\eps}K}{\modu{\ln \delta}}\left(1 + \hh_0(\delta)\right)
	\qtext{and}				
	\nu'_\delta =   \frac{\upsilon_0\sqrt{\eps}K}{\delta \modu{\ln \delta}^2}\left(1 +  \hh_1(\delta)\right)
	\label{eq:asymp_nu}
\ee
where $(\hh_j)_{j\in\{0,1\}}$ are analytic functions over $[\delta^*, 2\delta^* ]$ that satisfy the relations
\bes
\modu{\hh_j(\delta)} \leqp \modu{\ln \delta^*}^{-1}\qtext{and}\displaystyle K = \sqrt{\frac{7\pi^2}{6}AB}=\sqrt{\frac{21\pi^2}{8} \frac{m_1+m_2}{m_0} }.
\ees
\end{lemma}

%%%%%%%%%%%%%%%%%%%%%%%%%%%%%%%%%%%%%%%%%%

%%% Domaine Da %%%%%%%%%%%%%%%%%%%%%%%%%%%%%%%
We define a subset of the horseshoe region in order to build the adapted action-angle variables.
Thus, let us consider the domain $\fD_{*}$ defined as
\be
\fD_{*} = \left\{ \begin{array}{c}(I_1,\varphi_1) \in \RR\times]0,2\pi[ \qtext{such that} \sHt_1(I_1,\varphi_1) =  h_\delta \\\mbox{with}\quad \delta^* \leq \delta \leq 2\delta^* \end{array}\right\}.
\label{eq:Deps}
\ee	
This set, which corresponds to the grey region in Fig.\ref{fig:separatrix}, contains the horseshoe orbits of the mechanical system which are close to the separatrix.
%%%%%%%%%%%%%%%%%%%%%%%%%%%%%%%%%%%%%%%%%%

%%% Construction de variable AA %%%%%%%%%%%%%%%%%%%%%%
In this domain we can build a system of action-angle variables denoted $(J_1,\phi_1)$ such that
\be
\begin{gathered}
\sH_1(J_1)= \sHt_1\circ\fF(J_1, \phi_1) =h_\delta
\qtext{and}
\sH_1'(J_1) = \nu_\delta\\
\mbox{where } \quad\fF=(\fF_1,\fF_2) :\ \Bigg\{
		\begin{array}{ccc}
		 \fF^{-1}\!\left(\fD_{*} \right)		&\longrightarrow 	& \fD_{*}\\
			(J_1,\phi_1) 	&\longmapsto  		&(I_1, \varphi_1)
		\end{array}\\
		 {\rm with}\quad \fF^{-1}\!\left(\fD_{*} \right) =\cS_{*}\times{\mathbb T}\ ;\ \cS_{*}=\left[a, b\right]\subset{\mathbb R}.
		\label{eq:AA_transf}
\end{gathered}
\ee
%%%%%%%%%%%%%%%%%%%%%%%%%%%%%%%%%%%%%%%%%%

 If we restrict our attention to an arbitrary energy level corresponding to a fixed shift of energy $\delta$ that belongs to the segment $\left[\delta^* , 2\delta^*\right]$, the transformation in action-angle variables can be defined explicitly by the classical integral formulation:
\bes
	\begin{split}
	J_1 		&= 4 \sqrt{\eps BA^{-1}} \int_{\varphi_{1,\delta}^{\min}}^\pi \sqrt{U_\delta(s)} \rd s \qtext{and}
	\phi_1 	= -\frac{\nu_\delta}{2\upsilon_0 \sqrt{\eps AB}}\int_{\varphi_{1,\delta}^{\min}}^{\varphi_1} \frac{\rd s}{\sqrt{U_\delta(s)}}\,
\end{split}
\ees
where $\delta$ and $\varphi_{1,\delta}^{\min}$ are functions of $(I_1,\varphi_1)$.

%%%%%%%%%%%%%%%%%%%%%%%%%%%%%%%%%%%%%%%%%%

%%% Domaine d'holomorphie %%%%%%%%%%%%%%%%%%%%%%%%
As we look for a complex domain of holomorphy for the integrable Hamiltonian $\sH_1$, we have the following

\begin{theorem}[Semi-fast Holomorphic Extension]
\label{Th:largeurs}
For $\delta^*$ small enough ($\delta^*\pleq 1$), the transformation $\fF$ can be extended holomorphically over $\cB_{\rhoh_1}\cS_{*}\times \cV_{\sigh_1} \TT$ with $\rhoh_1 = \sqrt{\eps}(\delta^*)^{\ph}$ and $\sigh_1 =(\delta^*)^{\ph}$ for some positive exponent $\ph$.
Moreover, the extended function is $C$-Lipschitz with $C\peq 1/\sqrt{\delta^*}$.

%In the sequel, we will denote the exponent:
 %\be
%\rpq=\ph\rpqh
 %\label{eq:expo_larg_analy}
% \ee

\end{theorem}

\begin{remark}
Rough estimates lead to $\ph=11/2$ which is far to be optimal. Moreover, without loss of generality, we can assume that $\delta^*$ is small enough that in the domain $\cB_{\rhoh_1}\cS_{*}$ there is a diffeomorphism between $J_1$ and $\delta$. We will also use the notations $h_\delta$, $\nu_\delta$, $\nu'_\delta$ for the energy and the frequency on the complex domain $\cB_{\rhoh_1}\cS_{*}$. From a more general point of view, the polynomial link between the energy shift $\delta^*$ and the analyticity widths is arbitrary and certainly not optimal.  Another possibility would be to leave these quantities independent during the calculations and to fix them at the end, in view of the constraints obtained.
\end{remark}
%%%%%%%%%%%%%%%%%%%%%%%%%%%%%%%%%%%%%%%%%%

	\subsection{The Hamiltonian in semi-fast action-angle variables}

%%% AA dans le full problem %%%%%%%%%%%%%%%%%%%%%%%%%
Going back to the averaged Hamiltonian $\sHt$ considered in (\ref{eq:Approx_reduc}), the introduction of action-angle variables of the mechanical system leads to the following expressions.

\begin{theorem}[Semi-fast Action-Angle variables]
\label{Th:AA}
 With the notations of Section \ref{sec:notations}, for $\delta^*$ small enough ($\delta^*\pleq 1$) there exists a canonical transformation
\begin{gather*}
		\Psi: \quad \Bigg\{
		\begin{array}{ccc}
			\cK_{\brho,\bsig}		&\longrightarrow 	& \cKh_{1/6}\\
			(\bJ,\bphi,\bw, \bwt) 	&\longmapsto  		&	(\bI,\bvphi,\bw, \bwt)
		\end{array}\nnb\\
	\mbox{where} \quad 	(I_1,\varphi_1,I_2, \varphi_2) = (\fF(J_1,\phi_1), J_2,\phi_2)
\end{gather*}
and
\bes
 	\begin{gathered}
 		\cK_{\brho,\bsig} = \cB_{\rho_1}\cS_{*} \times \cB^1_{\rho_2}  \times \cV_{\sig_1}\TT \times \cV_{\sig_2}\TT\times \cB^4_{\sqrt{\rho_2\sig_2}}\qtext{such that}\\
		\rho_1 \peq \sqrt{\eps}(\delta^*)^{\ph},\quad  \rho_2 \peq \eps^\beta , \quad \sig_1 \peq(\delta^*)^{\ph}, \quad  \sig_2 \peq  1,\\
		\mbox{and} \quad 1/3 <\beta<1/2.
	\end{gathered}
\ees
	
 Then, the transformed Hamiltonian $\sH= \sHt\circ \Psi$ is analytic, satisfies the D'Alembert rule, and reads
\bes
	\begin{split}
		\sH(\bJ, \phi_1, \bw , \bwt) = \sH_1(J_1) + \sH_2(J_2) &+ \sQ(J_1, \phi_1, \bw , \bwt)+  \sR(\bJ, \phi_1,\bw,\bwt)
	\end{split}
\ees
where
\bes
			\begin{gathered}
			\sH_1(J_1) = h_{\delta},\quad \sH_1'(J_1) = \nu_{\delta}, \quad \sH_1''(J_1) = -\frac{\nu_{\delta}' \nu_{\delta}}{\eps \upsilon_0 B}, \\
							\sH_2(J_2)=	\upsilon_0 \left(J_2-E J_2^2\right),\\
								   \sR(\bJ, \phi_1,\bw,\bwt) = \sR_0(J_1, \varphi_1) + \sum_{ j,k\in\{1,2\}} \sR_{(j,k)}(\bJ, \phi_1,\bw,\bwt)w_j\wt_k,
			\end{gathered}
\ees
\bes
			\begin{split}
				  \sQ(J_1,\phi_1,\bw,\bwt) 	&=  \sum_{j,k \in\{1,2\}} \sQ_{(j,k)}(J_1, \phi_1)w_j\wt_k\\
									&= i\eps\upsilon_0  D\bigg(
					\frac{\cA(J_1,\phi_1)}{m_1}w_1\wt_1 	
				+ 	\frac{\cB(J_1,\phi_1)}{\sqrt{m_1m_2}} w_1\wt_2
			  \\ 	&\phantom{=}+   	\frac{\conj(\cB)(J_1,\phi_1)}{\sqrt{m_1m_2}} \wt_1 w_2
				+	\frac{\cA(J_1,\phi_1)}{m_2}w_2\wt_2
				\bigg)
			\end{split}
\ees	
with $\cA = \cAt\circ\fF_2$ and $\cB = \cBt\circ\fF_2$.
Moreover, the following bounds are satisfied:
\be
	\begin{gathered}
	\begin{aligned}
		&\norm{\sH_1^{(l)}}_{\cK_{\brho,\bsig}} \leqp\eps^{\frac{2-l}{2}}(\delta^*)^{-l\ph}&  &(l\in\{0,\ldots, 4\}),&\\
 		&\norm{\partial_{J_1}^l \sQ_{(j,k)}}_{\cK_{\brho,\bsig}} \leqp \eps^{\frac{2-l}{2}}(\delta^*)^{-l\ph}&\quad &(l\in\{0,1,2\}),&
				 	 \label{eq:AAestimates}
 	\end{aligned}\\
		\norm{\sR}_{\cK_{\brho,\bsig}} \leqp \eps^{3\beta},\quad
		\norm{\sR_{(j,k)}}_{\cK_{\brho,\bsig}} \leqp \eps^{2-2\beta},
	\end{gathered}
\ee
and
\bes
\begin{aligned}
&\norm{\partial_{J_1} \fF_1}_{\cK_{\brho, \bsig} 	}	\leqp (\delta^*)^{-\ph},& \quad
&\norm{\partial_{\phi_1} \fF_1}_{\cK_{\brho, \bsig} }	\leqp \sqrt{\eps} (\delta^*)^{-\ph},&\\
&\norm{\partial_{J_1} \fF_2}_{\cK_{\brho, \bsig} 	}	\leqp \sqrt{\eps}^{-1} (\delta^*)^{-\ph},& \quad
&\norm{\partial_{\phi_1} \fF_2}_{\cK_{\brho, \bsig} }	\leqp  (\delta^*)^{-\ph}.
%\label{eq:AAtransfo}
\end{aligned}
\ees

\end{theorem}
%%%%%%%%%%%%%%%%%%%%%%%%%%%%%%%%%%%%%%%%%%

%%% Diffeo J <-> delta %%%%%%%%%%%%%%%%%%%%%%%%%%%
\begin{remark}
\label{remark_diffeo}
In the complex domain $\cK_{\brho, \bsig}$, there exists a diffeomorphism between  the shift of energy $\delta$ and the semi-fast action $J_1$.
\end{remark}
%%%%%%%%%%%%%%%%%%%%%%%%%%%%%%%%%%%%%%%%%

%%% Sec4 Construction of an adapted integrable approximation (suite)
% Seconde Moyenne sur Ht
% Main part Frequence sec.
	\subsection{Second averaging}

In the fourth step, we average the Hamiltonian $\sH$ over the semi-fast angle $\phi_1$ in order to reject the $\phi_1$-dependency up to an exponentially small remainder.

Up to now, $\delta$ and $\eps$ were two independent small parameters. However, in order to simplify the calculations in the following we link the bounds in energy level to $\eps$ such that
\be
	\delta^* = \eps^\rpqh
	\label{eq:delta_eps}
\ee
where $\rpqh$ is a positive exponent that will be determined in the sequel.
Hence, the analyticity widths of the considered domain of holomorphy $\cK_{\brho, \bsig}$ are equal to
\bes
\begin{gathered}
	\rho_1 \peq \sqrt{\eps}\eps^\rpq,\quad  \rho_2 \peq \eps^\beta , \quad \sig_1 \peq\eps^\rpq, \quad  \sig_2 \peq  1,\\
	\rpq := \ph \rpqh, \qtext{and} 1/3 <\beta < 1/2.
\end{gathered}
\ees

Then, using the notations of Section \ref{sec:notations} we restrict the domain of holomorphy of the Hamiltonian $\sH$ to the complex domains $\cK_{p}= \cK_{p\brho, p\bsig}$ for $0<p\leq 1$ with
\bes
\begin{gathered}
\rho_1 \peq \sqrt{\eps} \eps^{5\rpq},\quad  \rho_2 \peq \eps^\beta,  \quad \sig_1 \peq \eps^\rpq, \quad \sig_2 \peq 1,\\
\rpq = 2\ph\rpqh, \qtext{and} 1/3<\beta<1/2
\end{gathered}
\ees
where the following bounds on the semi-fast frequency are valid:
\bes
\frac{\sqrt{\eps}}{\modu{\ln \eps}} \leqp \modu{\nu_\delta}=\modu{\sH_1'(J_*)}  \leqp \frac{\sqrt{\eps}}{\modu{\ln \eps}}
\ees
for any $\delta\in\left[\delta^* ,2\delta^*\right]$ (or equivalently for any $J_*\in\cS_{*}$) and for $J_1\in\cB_{\rho_1}\cS_{*}$ there exists $J_*\in\cS_{*}$ such that
\bes
\norm{\sH_1'(J_1) - \sH_1'(J_*)}_{p} \leqp \sqrt{\eps}\eps^{3\rpq}.
%\label{eq:est_H'1}
\ees
The latter estimates come from  (\ref{eq:AAestimates}) which allows to bound  $\sH''_1$.
 Consequently, we have
\be
\frac{\sqrt{\eps}}{\modu{\ln \eps}} \leqp  \modu{\sH_1'(J_1)} \leqp \frac{\sqrt{\eps}}{\modu{\ln \eps}}
\label{eq:bound_nu}
\ee
uniformly over $\cK_p$.

In this setting, one can normalize once again the Hamiltonian in order to eliminate the semi-fast angle $\phi_1$ according to the following
\begin{theorem}[Second Averaging Theorem]
\label{Th:Moy2}
For
\be
	\rpq = \frac{3\beta - 1}{15} \qtext{with} 4/9<\beta< 1/2
 	\label{eq:Moy2_q}
\ee
and $\eps$ small enough $(\eps \pleq 1)$, there exists a canonical transformation
\begin{gather*}
		\Psib : \quad \Bigg\{
		\begin{array}{ccc}
			\cK_{7/12}                				&\longrightarrow 	& \cK_{1} \\
			(\bJs, \bphis, \bws, \bwts)  	&\longmapsto    	&(\bJ,\bphi,\bw,\bwt )\nnb
		\end{array}\nnb\\
		\mbox{with} \quad \cK_{5/12} \subseteq \Psib(\cK_{7/12})  \subseteq \cK_{3/4}
%\label{eq:Moy2_emboitement}
\end{gather*}
such that $\sH\circ\Psib = \sHb + \sHd_*$
where
\bes
\begin{split}
	\sHb(\bJs,  \bws, \bwts) = \sH_1(\Js_1) &+ \sH_2(\Js_2) +  \sQb(\Js_1, \bws, \bwts) + \sF(\bJs, \bws, \bwts)
	%\label{eq:Moy2_SecHam}
\end{split}
\ees
 is the secular Hamiltonian with
\bes
\begin{gathered}
\sH_1(\Js_1) = h_{\deltas}, \quad \sH_1'(\Js_1) = \nu_{\deltas}, \quad \sH_1''(\Js_1) = -\frac{\nu_{\deltas}' \nu_{\deltas}}{\eps \upsilon_0 B},  \\
\sH_2(\Js_2)=	\upsilon_0 \left(\Js_2-E \Js_2^2\right) , \\
\sF(\bJs, \bws, \bwts)=\sF_0(\bJs) + \sum_{j,k\in\{1,2\}} \sF_{(j,k)}(\bJs, \bws, \bwts)\ws_j\wts_k,	\\
\sQb(\Js_1,\bws, \bwts)=\frac{1}{2\pi}\int_0^{2\pi} \sQ(\Js_1,\phis_1,\bws, \bwts)\rd\phis_1,
\end{gathered}
\ees
and $\sHd_*$ is the remainder that contains the $\phis_1$-dependency such that
\bes
	\sHb_*^\dagger(\bJs,\bws, \bwts) 	= \frac{1}{2\pi}\int_0^{2\pi} \sHd_*(\bJs,\phis_1,\bws, \bwts)\rd\phis_1= 0.
\ees
Moreover, for $0<p<7/12$ the choice for $\rpq$ given in (\ref{eq:Moy2_q}) yields the following upper bounds:
\be
	\begin{gathered}
	\begin{aligned}
	 	 &\norm{\sH_1^{(l)}}_{p} \leqp \eps^{\frac{2-l}{2} - l\rpq}&\quad & (l\in\{0,\ldots, 4\}),&\\
%		%	
  		&\norm{\sQb_{(j,k)}^{(l)}}_{p} \leqp  \eps^{\frac{2-l}{2} - l\rpq}& \quad & (l\in\{0,1,2\}),&
	\end{aligned}\\
			\norm{\sF_0}_{p} \leqp \eps^{3\beta},\quad\norm{\sF_{(j,k)}}_{p} \leqp \eps^{2-2\beta},
				\label{eq:Moy2_bounds1}
	\end{gathered}
\ee
\be
	\begin{aligned}
		&\norm{\partial_{J_1}\sF_0}_{p} \leqp \eps^{3\beta-\frac{1}{2}-5\rpq},&\quad
%		%	
		&\norm{\partial_{J_2}\sF_0}_{p} \leqp \eps^{2\beta},&\\
		&\norm{\partial^2_{J_1}\sF_0}_{p} \leqp \eps^{3\beta-1-10\rpq},&\quad
		&\norm{\partial^2_{J_2}\sF_0}_{p} \leqp \eps^\beta,&\\
		&\norm{\partial_{J_1}\partial_{J_2}\sF_0}_{p} \leqp \eps^{2\beta-\frac{1}{2}-5\rpq},&
%		%		%		
		&\norm{\partial_{\bJ}\sF_{(j,k)}}_{p} \leqp \eps^{1/3},&\\
%		%
		&\norm{\partial^2_{\bJ} \sF_{(j,k)}}_{p} \leqp \eps^{-1/3},&
		&\norm{\partial^2_{(\bw, \bwt)} \sF_{(j,k)}}_{p} \leqp \eps^{2-3\beta},&
	\end{aligned}
	\label{eq:Moy2_bounds2}
\ee
and
\be
	 \norm{\sHd_*}_{p} \leqp \eps \exp(-\frac{1}{\eps^{\rpq}}).
	\label{eq:Moy2_Remav}
\ee
Finally, we can bound the size of the transformation by
\bes
\begin{aligned}
	&\norm{\Js_1 - J_1}_{p} \leqp \modu{\ln \eps} \eps^{3\beta-\frac{1}{2}-2\rpq},\\
	&\norm{\phis_1 - \phi_1}_{p} \leqp \modu{\ln \eps} \eps^{3\beta-1-6\rpq},\\
	&\norm{\phis_2 - \phi_2}_{p} \leqp \modu{\ln \eps} \eps^{2\beta-\frac{1}{2}-\rpq},\\
	&\norm{(\bws, \bwts) - (\bw, \bwt)}_{p} \leqp \modu{\ln \eps} \sqrt{\eps}\eps^{-\rpq} \norm{(\bw, \bwt)}_{p},\\
	&\mbox{while 	$\Js_2= J_2$.}
	\end{aligned}
%\label{eq:Moy2_transf}
\ees
\end{theorem}
The remainder $\sHd_*$ being exponentially small, we drop it for the moment in order to focus on the secular Hamiltonian $\sHb$.
\begin{remark}
In the same way as in the remark \ref{remark_diffeo}, on the complex domain $\cK_{p}$ there exists a diffeomorphism between the shift of energy $\deltas$ and the semi-fast action $\Js_1$.
\end{remark}

\subsection{The normal frequencies}
In the fifth step, we focus our effort on the secular dynamics.
%%% Partie principale de la Dynamique normale %%%%%%%%%%%%%
From the estimates of the previous theorem, the main part of the secular dynamics is given by  $\sQb$ whose coefficients $\cAb$ and $\cBb$ are the result of the averaging of $\cAt\circ\fF_2$ and $\cBt\circ\fF_2$ with respect to the semi-fast angle $\phi_1$.
%%%%%%%%%%%%%%%%%%%%%%%%%%%%%%%%%%%%%%%%%%
%%% Moyenne par rapport à l'autre angle %%%%%%%%%%%%%%%%%
The action-angle transformation $\fF$ being not explicit, we perform the average with respect to the initial angle $\varphi_1 = \fF_2(J_1, \phi_1)$.
More specifically, we consider the average at a semi-fast action $J_*\in\cS_{*}$,
hence on a level curve corresponding to $\deltas\in\left[\delta^*,2\delta^*\right]$ such that
\be
\begin{gathered}
	\begin{aligned}
	\cAb(J_*) 	&=  \frac{\nu_{\deltas}}{\upsilon_0\sqrt{\eps\pi^2 AB}}\int_{\varphi_{1,\deltas}^{\min}}^\pi \frac{\cAt(\varphi_1)\rd\varphi_1}{\sqrt{U_{\deltas}(\varphi_1)}} \,,\\
	\cBb(J_*) 	&=  \frac{\nu_{\deltas}}{\upsilon_0\sqrt{\eps\pi^2 AB}}\int_{\varphi_{1,\deltas}^{\min}}^\pi \frac{\re\left(\cBt(\varphi_1)\right)\rd\varphi_1}{\sqrt{U_{\deltas}(\varphi_1)}}
	\end{aligned}\\
	\mbox{with}\quad
	\frac{\rd\phis_1}{\rd t} =\sH_1'(J_*) = \nu_{\deltas}.
	\label{eq:moy_AB}
\end{gathered}
\ee
%
%%%%%%%%%%%%%%%%%%%%%%%%%%%%%%%%%%%%%%%%%%
%%% Calcul des valeurs propres %%%%%%%%%%%%%%%%%%%%%%
From these expressions, we deduce the asymptotic expansion of the pure imaginary eigenvalues of $(\sQb_{(j,k)}(J_*))_{j,k\in\{1, 2\}}$ that we denote $i\gt_{j,\deltas}$ for $j\in\{1,2\}$ where $\gt_{1,\deltas}$ and $\gt_{2,\deltas}$ correspond to the main part of the two secular frequencies.
%
%%%%%%%%%%%%%%%%%%%%%%%%%%%%%%%%%%%%%%%%%%
%
%%% Expansion des gt %%%%%%%%%%%%%%%%%%%%%%%%%%%%
Hence, we have the following
 \begin{theorem}[Secular Frequencies]
\label{Th:normal_freq}
	The asymptotic expansions of the main part of the secular frequencies as $\deltas$ tends to zero (with our definition of $\fD_*$ given by (\ref{eq:Deps})) are given by
	\be
		\begin{split}
		&\gt_{1,\deltas} = \eps \upsilon_0 \frac{m_1 + m_2}{m_0}\left( \frac{7}{8} + \hh_2(\deltas)\right),\\
		&\gt_{2,\deltas} = \eps \upsilon_0 \frac{m_1 + m_2}{m_0} \left(\frac{c_2}{\modu{\ln \deltas}} + \hh_3(\deltas) \right),
	\end{split}
	\label{eq:main_secular}
\ee
	where
\bes
	c_2 <0,  \quad \modu{\hh_2(\deltas)}\leqp \modu{\ln \delta^*}^{-1} \qtext{and} \modu{\hh_3(\deltas)}\leqp \modu{\ln \delta^*}^{-2}.
\ees
 \end{theorem}
%%
%%
%%% Sec4 Construction of an adapted integrable approximation (suite)
% Diag
% Birkhoff
%%% On veut controller la dynamique normale %%%%%%%%%%%%%%
%
Then, we need to reduce the quadratic part
%
%
 %%% Partie quadratic du Hamiltonien %%%%%%%%%%%%%%%%%%%
\be
\sQb(\Js_1,\bws, \bwts) + \sum_{j,k\in\{1, 2\}} \sF_{(j,k)}(\bJs,\bzero, \bzero)\ws_j \wts_k
\label{eq:quad_part}
\ee
 to a diagonal form for $\bJs \in \cB_* = \cB_{p\rho_1}\cS_{*} \times \cB^1_{p\rho_2}$ for some $0<p<7/12$.
Since we link the shift in energy $\deltas$ with the mass ratio $\eps$ via a power law as in  (\ref{eq:delta_eps}), the differences  between the coefficients of $(\sQb_{(j,k)}(\Js_1))_{j,k\in\{1,2\}}$ and $(\sQb_{(j,k)}(J_*))_{j,k\in\{1,2\}}$ for $\Js_1 \in \cB_{p\rho_1}\cS_{*}$ are negligible with respect to the  eigenvalues  (\ref{eq:main_secular}).
Indeed, the estimates \eqref{eq:Moy2_bounds1} of Theorem \ref{Th:Moy2} together with the mean value theorem provide the following:
\bes
\norm{\sQb_{(j,k)}(\Js_1) - \sQb_{(j,k)}(J_*)}_p \leqp \eps^{\frac{31}{30}}\qtext{since} 4/9<\beta<1/2.
\ees
In the same way,  the estimates \eqref{eq:Moy2_bounds2} imply that  the coefficients $(\sF_{(j,k)})_{j,k\in\{1, 2\}}$  are of size $\eps^{2-2\beta}$ and thus, are also negligible with respect to the  eigenvalues  (\ref{eq:main_secular}) over $\cB_*$.
Consequently, for all $\bJs\in\cB_*$, the main part of the eigenvalues in the quadratic form \eqref{eq:quad_part} are given by the eigenvalues of $\sQb(J_*,\bws, \bwts)$ for some $J_*\in\cS_*$.

We denote by $ ig_j(\bJs)$ the eigenvalues of (\ref{eq:quad_part}) for  $\bJs \in \cB_*$. Since these quantities are perturbation of $\gt_{1,\deltas}$ and $\gt_{2,\deltas}$, which are different for $\eps$ small enough, the spectrum of (\ref{eq:quad_part})  is simple.
On the real domain $\cB_*\cap{\mathbb R}^2$, the angular momentum  $\cC(\Js_2, \bws, \bwts)$ given by \eqref{eq:Dal_integral_moy} being an integral of $\sH\circ \Psib$ considered in Theorem \ref{Th:Moy2}, the manifold $\rC_0$ is normally stable.  These two properties imply that the two perturbed frequencies are also purely imaginary numbers, or equivalently $g_j(\bJs)$ is real, for $\bJs \in \cB_*\cap{\mathbb R}^2$.

  In the complex domain $\cB_*$, we have
\bes
	 g_j(\bJs) = \gt_{j,\deltas} + f_j(\bJs) \qtext{with} \vert\vert f_j\vert\vert_{p}\leqp \eps^{2-2\beta}.
%\label{eq:def_g_j}
\ees
Consequently, as $\delta^* =\eps^\rpqh\leq\deltas\leq 2\eps^\rpqh =2\delta^*$, for $\varepsilon$ small enough we have
\be
	\eps \leqp \modu{g_1(\bJs)} \leqp \eps \qtext{and} \frac{\eps}{\modu{\ln\eps}} \leqp\modu{ g_2(\bJs) }\leqp \frac{\eps}{\modu{\ln \eps}}\label{boundg1g2}
\ee	
on the complex domain $\cB_*$.
%%%%%%%%%%%%%%%%%%%%%%%%%%%%%%%%%%%%%%%%%%%
%
%%%% Estimés sur les fréquences %%%%%%%%%%%%%%%%%%%%%%
%%%%%%%%%%%%%%%%%%%%%%%%%%%%%%%%%%%%%%%%%%%	

%%% Diagonalisation %%%%%%%%%%%%%%%%%%%%%%%%%%%%%

%%% Alex

Since the spectrum is simple, there exists a symplectic transformation which is linear with respect to $\bws$, $\bwts$  and diagonalizes the quadratic form (\ref{eq:quad_part}).
In the same way, the eigenspaces of \eqref{eq:quad_part} are close to those of $\sQb(\Js_1, \bws, \bwts)$ which correspond to a non-singular transformation depending of $\cAb(\Js_1)$ and $\cBb(\Js_1)$.
Hence, we have the following
\begin{theorem}[Diagonalization]
 With the notations of Section \ref{sec:notations}, for
 \bes
 0<r\pleq \eps^{\frac{1}{4} + 3\rpq} \qtext{with} \rpq = \frac{3\beta-1}{15} \qtext{and} 4/9<\beta<1/2
 \ees
 (which is strictly smaller that $\sqrt{\rho_2\sig_2}$ with the considered values of $\beta$), there exists $0<p<7/12$ and  a canonical transformation
\bes
		\Xi : \quad \Bigg\{
		\begin{array}{ccc}
			\cK_{p,r}						&\longrightarrow 	& \cK_{7/12}	\\
			(\bGam,\bpsi,\bz, \bzt) 	&\longmapsto  		&	(\bJs,\bphis,\bws, \bwts)
		\end{array}
\ees
	which is linear with respect to $\bz$ and $\bzt$ with
\bes	
	 \bJs =\bGam , \quad  \bphis =\bpsi +\cG_2(\bGam, \bz, \bzt) \qtext{where} \cG_2= \cO_2(\norm{(\bz, \bzt)}),
\ees
such that the secular Hamiltonian $\sHb\circ\Xi = \sHc  + \sRc $ reads
\be
	\begin{split}
		\sHc &= \sH_1+ \sH_2 + \sF_0 +  \sum_{j=1,2} i g_j (\bGam ) z_j\zt_j \qtext{and}
\label{eq:frorm_norm0}\\
			\sRc(\bGam, \bz, \bzt) &= \cO_4(\norm{(\bz, \bzt)}) \qtext{with} \norm{\sRc}_{p,r} \leqp \eps^{2-3\beta} r^4.
	\end{split}
\ee	
\label{Th:diag}
\end{theorem}
As a consequence, we set out the following
\begin{Corollary}
\label{cor1}
Taking into account the exponentially small  remainders in Theorems \ref{Th:Moy1} and \ref{Th:Moy2}, the planetary Hamiltonian $\cH$ given in (\ref{eq:ham_plan})  reads
\bes
\begin{gathered}
\check{\rH}(\bGam,\bpsi,\bz, \bzt)  = \sHc(\bGam,\bz, \bzt)  + \sRc(\bGam, \bz, \bzt)  + \sHc_{*}(\bGam,\bpsi,\bz, \bzt)\\
\mbox{with} \quad \norm{\sHc_{*}}_{p,r} \leqp  \eps \exp(-\frac{1}{\eps^{\alpha}})\\
\qtext{where} \alpha = \frac{1-2\beta}{5} \qtext{and} 4/9 < \beta < 1/2\, .
\end{gathered}
\ees
\end{Corollary}

\section{Application of a P\"oschel  version of KAM theory}
\label{sec:KAM}

As said in the introduction, we apply \citeauthor{1996Po} version  of KAM theory for the persistence of lower dimensional normally elliptic invariant tori  \citep{1996Po}. More precisely we implement a formulation of  \citeauthor{1996Po}'s  theorem \citep{1996Po} given in Proposition 2.2 of  \cite{BiCheVa2003}, which is a summary of Theorems A, B and Corollary C in \cite{1996Po} for the finite-dimensional case.
In the co-orbital case, we have to be cautious about the dependance with respect to the small parameter $\eps$ of the constants involved in these statements. Indeed, some quantities, such as the analyticity width with respect to the semi-fast angle, are singular in our problem.

For the sake of clarity, we will now try to be as close as possible to the notations used in \cite{BiCheVa2003}.
P{\"o}schel's theorem  requires a parametrized normal form that can be written in our case as
\be
  \rN(\by,\bz,\bzt;\bxi) =   \sum_{j=1,2} \omega_j(\bxi) y_j + \sum_{j=1,2} i\Om_j(\bxi)z_j\zt_j
  \label{eq:Poechel_NF}
\ee
where $\bom$ are the internal frequencies, which depends on the 2d parameter $\bxi$ belonging a complex set $\Pi$ defined later and $\bOm$ are  the normal or secular  frequencies.
The 2d tori given by the quasi-circular manifold  $\rC_0 = \{\bz=\bzt=\bzero\}$ are invariant under the flow of the normal form given by (\ref{eq:Poechel_NF}) and are normally elliptic.

Let us now consider the Hamiltonian
\be
  \rH(\by,\bpsi,\bz, \bzt; \bxi) =  \rN(\by,\bz, \bzt; \bxi) + \rP(\by,\bpsi,\bz, \bzt; \bxi)
  \label{eq:planHam}
\ee
with
\bes
\left\{\quad\begin{aligned}
 \omega_j(\bxi) &=   \sH_j'(\xi_j) + \partial_{\Gam_j} \sF_0(\bxi) \\
  \Om_j(\bxi)     & = g_j(\bxi)
%\label{eq:def_freq}
\end{aligned}\right.
\ees
and
\be
  \rP(\by,\bpsi,\bz, \bzt; \bxi) = \check{\rH}(\bxi + \by, \bpsi, \bz, \bzt)-\rN(\by,\bz, \bzt; \bxi)- \sHc(\bxi,\bzero, \bzero),
\label{eq:perturbation}
\ee
$\check{\rH}$ and $\sHc$ being respectively the planetary Hamiltonian considered in the corollary \ref{cor1} and the integrable approximation of Theorem \ref{Th:diag}.

We will need estimates on the Lipschitz norm of a function $f$ defined over the domain $\Pi$:
\bes
\modu{f }_\Pi^{\rm Lip}  = \sup\limits_{\bxi \neq \bxi'\in\Pi} \frac{\modu{f(\bxi) - f(\bxi')}}{\modu{\bxi - \bxi'}}\, .
\ees
Hence $\modu{ f}_\Pi^{\rm Lip}\leq\norm{ \rd f}_\Pi$ for a differentiable function.
Especially, we consider the upper bound:
\bes
\vert \bom \vert_\Pi^{\rm Lip}  + \vert \bOm \vert_\Pi^{\rm Lip} \leq  M.
\ees
Moreover,  \citeauthor{1996Po} reasoning requires that the internal frequency map $\bom$ is a diffeomorphism onto its image $\bom(\Pi)$ (more precisely $\Pi=\cB^2_\rho$ where $\rho$ is determined in Proposition \ref{prop1}). Thus, we consider the following upper bound:
\bes
\vert \bom^{-1} \vert_\Pi^{\rm Lip} \leq L.
\ees

In order to ensure  the persistence of normally elliptic tori, we  have to check Melnikov's condition for multi-integers of length bounded by
\bes
K_0 = 16L M .
\ees
More precisely, we have to prove the existence of a constant $\gam_0>0$ such that
\bes
\begin{split}
&\min\limits_{\bxi\in\Pi} \left\{ \,\vert\Om_1(\bxi)\vert,  \,\vert\Om_2(\bxi)\vert,  \, \vert\Om_1(\bxi) - \Om_2(\bxi)\vert\,  \right\} \geq \gamma_0 \qtext{and}\\
&\min\limits_{\bxi\in\Pi}  \vert \bom(\bxi)\bigcdot \bk + \bOm(\bxi)\bigcdot \bl\vert \geq \gamma_0  \quad
\forall \, 0 <\vert \bk\vert  \leq K_0, \; \vert \bl\vert \leq 2
\end{split}
\ees
(see Proposition \ref{prop2} for more details).

The planetary Hamiltonian $\rH$ defined in (\ref{eq:planHam}) is analytic over the domain
\be
  D(\rb,\sb) =   \left\{ (\by,\bpsi,\bz, \bzt) \in \CC^8\  /\ \modu{\by} < \rb^2, \, \bpsi \in \cV_{\sb}\TT^2,\,\modu{(\bz,\bzt)} <  \rb\right\} \label{eq:domain_Dbar}
\ee
such that $0< \rb<r$ and $ 0<\sb \pleq  \sigma_1$.
The thresholds of Proposition 2.2 in \cite{BiCheVa2003} concern the size of the perturbation $\rP$ measured  using the norm of its associated Hamiltonian vector field $X_\rP$.  More precisely, on the domain $D(\rb,\sb) $, we consider the following norms:
\bes
\begin{split}
 \norm{X_\rP}_{\rb, D(\rb,\sb)}  &=  \sup\limits_{D(\rb,\sb)\times\Pi}
 \left(
 \vert\partial_{\by} \rP \vert  + \frac{1}{\rb^2}\vert\partial_{\bpsi} \rP \vert  + \frac{1}{\rb}\left( \vert\partial_{\bz} \rP \vert  +\vert\partial_{\bzt} \rP \vert  \right)
 \right), \\
  \norm{X_\rP}_{\rb, D(\rb,\sb)}^{\rm Lip}  &=  \sup\limits_{D(\rb,\sb)}
 \left(
 \vert\partial_{\by} \rP \vert_\Pi^{\rm Lip}  + \frac{1}{\rb^2}\vert\partial_{\bpsi} \rP \vert_\Pi^{\rm Lip}  + \frac{1}{\rb}\left( \vert\partial_{\bz} \rP \vert_\Pi^{\rm Lip} +\vert\partial_{\bzt} \rP \vert_\Pi^{\rm Lip}  \right)
 \right)
\end{split}
\ees
where, for a function $f$ defined over $D(\rb,\sb)\times\Pi$, we define
\bes
\vert f \vert_\Pi^{\rm Lip} = \sup\limits_{\bxi \neq \bxi'\in\Pi} \frac{\norm{ f(\,\cdot\,;\bxi) - f(\,\cdot\,;\bxi') }_{D(\rb,\sb)}}  {\vert \bxi - \bxi'\vert},
\ees
hence $\vert f \vert_\Pi^{\rm Lip}\leq\vert\vert\partial_{\bxi} f\vert\vert_{D(\rb,\sb)\times\Pi}$ for a differentiable function.

 Now, we can state the theorem which ensures the existence of invariant tori.

%%%%%%%%%%%%%%%%%%%%%%%%%%%
\begin{theorem}
\label{th:lastone}
With the  previous notations, there exists a large enough parameter $\tau>0$ such that for $\gamma \in ] 0, \gamma_0/2]$,  if
\be
  \epsilon =\norm{X_\rP}_{\rb, D(\rb,\sb)}   +  \frac{\gamma}{M\gamma_0}\norm{X_\rP}_{\rb, D(\rb,\sb)}^{\rm Lip}
  \pleq \frac{c\gamma}{L^aM^a} \sigma_1^b ,
\label{Dernier_Seuil}
\ee
where $a= \tau +1$, $b= 2\tau +4$, and $c>0$ is a constant depending only on $\tau$, then the following holds. There exists a non-empty Cantor set of parameters  $\Pi_{*} \subset \Pi$ (more precisely, the measure of its complement $\Pi\backslash\Pi_{*}$ goes to zero with $\gamma$) and a Lipschitz continuous family of tori embedding
\bes
{\bf T}: \quad \Bigg\{
\begin{array}{ccc}
   \TT^2\times\Pi_{*}   &  \longrightarrow    & D_{ \rb}  \\
     (\btheta, \bxi)    &  \longmapsto     & ( \by(\btheta, \bxi),  \bpsi(\btheta, \bxi),\bz(\btheta, \bxi), \bzt(\btheta, \bxi))
\end{array}
\ees
with
\bes
D_{\rb }=]-\rb^2,\rb^2[\times]-\rb^2,\rb^2[\times\TT^2\times\cB^4_{\rb} \mbox{ and } \bzt(\btheta, \bxi)=-i{\bar\bz}(\btheta, \bxi),
\ees
a Lipschitz homeomorphism $\bom_{*}$ on $\Pi_{*}$ such that, for any $\bxi \in \Pi_{*}$, the image ${\bf T}(\TT^2,\bxi)$ is a real-analytic (elliptic) $\rH$-invariant 2-dimensional torus, on which the flow linked to $\rH$ is analytically conjugated to the linear flow $\btheta\mapsto\btheta+\bom_{*}t$. Moreover, the embedding ${\bf T}(\TT^2,\bxi)$ for $\bxi\in\Pi_{*}$ is $\epsilon /\gamma$-close to the torus $\{\bGam+\bxi=\bz =\bzt ={\bf 0}\}$ with the notations of the corollary \ref{cor1}

\end{theorem}

\begin{remark}
The previous theorem corresponds to Proposition 2.2 in \cite{BiCheVa2003} but we have to specify the dependance with respect to the parameters which appear in threshold (\ref{Dernier_Seuil}) since these constants goes to zero in our case. This is obtained by going back to the original paper of \cite{1996Po} where the exponent ``$a$" comes from
Corollary C, the exponent ``$b$" is defined below formula (6) and (17) of \cite{1996Po}, finally the parameter $\tau$ is defined in formula (22).
\end{remark}

As it was specified above, we have to be cautious with the fact that the involved constants degenerate when $\eps$ goes to zero. This is overcome by constraining $\eps$ to be inside an interval $[\eps_0/2, \eps_0]$ for any arbitrary $\eps_0>0$ and in Section \ref{sec:appendixKAM}, we prove that the main threshold \eqref{Dernier_Seuil} is satisfied if $\eps_0$ is small enough ($\eps_0\pleq 1$).

Consequently, for mass ratio $\eps$ small enough, we find the desired quasi-periodic horseshoe orbits.

\section{Extensions, comments and prospects}
\label{sec:Comments}
As we have seen, the obtained quasi-periodic motions suffer two limitations: they correspond to 2-dimensional tori but not Lagrangian tori in this 4-degree of freedom system and they are close to the $L_3$-separatrix.
\medskip

Concerning the first item, our initial goal was to obtain Lagrangian tori following the reasonings of Herman and F\'ejoz in the $N$-body planetary problem with large gaps between the planets \citep{2004Fe}. To ensure the existence of Lagrangian invariant tori, we have to prove that the frequency map associated with the Hamiltonian $\sHc$, introduced in Theorem \ref{Th:diag}, satisfies R$\ddot{\rm u}$ssmann  non-degeneracy condition  (i.e. its image should not be included in any hyperplane of $\RR^4$).
Hence, we have to consider the map $\rF(\bGam)=(\bom(\bGam), \check{\bOm}(\bGam))$ with:
  \bes
   \omega_j(\bGam)=\sH'_j(\Gam_j)+ \partial_{\Gam_j}\sF_0(\bGam)  \qtext{and}
      \check{\Omega}_j(\bGam)= \gt_{j,\deltas(\Gam_1)} + \check{f}_j(\bGam)
      \ees
 where $(\gt_{j, \deltas})_{(j\in\{1,2\})}$ are the normal frequencies associated with the averaged quadratic part $\left(\sQb_{(j,k)}(\Gam_1)\right)_{j,k\in\{1,2\}}$, that are introduced in Theorem \ref{Th:normal_freq}. The functions $\check{f}_j$ are small quantities generated by the remainder $\left(\sF_{(j,k)}(\bGam, \bzero, \bzero))\right)_{j,k\in\{1,2\}}$.
   If we consider the approximate frequency map
\bes
\rF_0(\bGam)=(\sH'_1(\Gam_1),\sH'_2(\Gam_2),\gt_{1,\deltas(\Gam_1)},\gt_{2,\deltas(\Gam_1)})
\ees
    we can prove that
\bes
{\rm det}(\rF_0,\partial_{\Gam_2}\rF_0,\partial_{\Gam_1}\rF_0,\partial^2_{\Gam_1}\rF_0)\neq 0.
\ees
But our approach does no t give enough control on the remainder $(\partial_{\Gam_j}\sF_0(\bGam) , \check{f}_j(\bGam))$ to ensure the same property on the complete frequency map $\rF(\bGam)$. Hence we don't have enough information in our approximation which has to be refined in order to prove R$\ddot{\rm u}$ssmann non-degeneracy condition for Lagrangian tori.
We believe that a possible way to overcome this issue would be to consider an integrable approximation truncated at higher order in semi-fast action. Indeed in this case, the integrable approximation would not be a mechanical system as in \cite{2015MeNeTr}.
From a more general point of view, estimates that may be useful in our context can certainly be found in the work of \citeauthor{BiChe2015} (\citeyear{BiChe2015}, \citeyear{BiChe2017}) where a general KAM theory of secondary Lagrangian tori is studied in the case of a perturbed mechanical system.

   We can also consider a nearly-invariant Lagrangian tori where the solutions are almost quasi-periodic for a very long time. More specifically, the horseshoe orbits which are $\varepsilon$-close to the $L_3$-separatrix  have four frequencies (fast, semi-fast and two normal frequencies) which are respectively of order $(1, \sqrt{\eps}/\modu{\ln\eps}, \eps, \eps/\modu{\ln \eps})$.
 These four different timescales, which prevent the occurrence of small divisors for $\eps$ small enough, allows to reduce the secular Hamiltonian $\sHc$ introduced in Theorem \ref{Th:diag} to a Birkhoff normal form up to an arbitrary order.
 Using Theorem 5.5 of \cite{MR980547} or Proposition 1 of \cite{MR1419016}, it is possible to get the following statement.

 \begin{theorem}

    The estimates \eqref{boundg1g2} and Theorem \ref{Th:diag} ensure that for an arbitrary $L\in\NN^*$, there exists $\eps_L>0$ such that for any $\eps <\eps_L$ we have
\be
		\frac{\eps}{\modu{\ln \eps}} \leqp \modu{l_1 g_1 (\bGam ) + l_2 g_2(\bGam )}
\label{eq:SmallDivisor}
\ee
	for any $(l_1,l_2)\in\ZZ^2$ of length $0<\modu{l_1} + \modu{l_2} \leq L$.	

	Hence, if we impose
\bes	
r<r_0\peq \eps^{-1-\beta/2}\left(\frac{\eps^{\beta/2}}{\modu{\ln\eps}}\right)^L,
%\label{borne_excentricite}
\ees
then for any $\eps<\eps_L$ and  $p'<p$ small enough,  there exists a canonical transformation
\bes
		\overline{\Xi}: \quad \Bigg\{
		\begin{array}{ccc}
			\cK_{p',r}		&\longrightarrow 	& \cK_{p,r}	\\
			(\bGams,\bpsis,\bzs, \bzts) 	&\longmapsto  		&	(\bGam,\bpsi,\bz, \bzt)
		\end{array}
\ees
with
\bes	
	 \bGam =\bGams , \quad  \bpsi =\bpsis +{\mathcal{G}}_4(\bGams, \bzs, \bzts)\quad \mbox{where ${\mathcal{G}}_4= \cO_4(\norm{(\bzs, \bzts)})$}
\ees
such that the transformed secular Hamiltonian
\bes
\sHc\circ\overline{\Xi} (\bGams,\bzs, \bzts)=\rN^{(L)} (\bGams,\bpsis,\bzs, \bzts)+ \rR_*^{(L+1)}(\bGams,\bzs, \bzts)
\ees
is reduced in a Birkhoff normal form up to order $L$ in $(\bzs, \bzts)$.

As a consequence, we have
\bes
\rN^{(L)}(\bGams,\bzs, \bzts) =  \sum_{s\in\{1,\ldots, \Ent(L/2)\}} \rN^{(s)}(\bGams,\bzs, \bzts)
\ees
where $\rN^{(s)}$ is a homogeneous  polynomial of degree $s$ in $\zs_1\zts_1$ and $\zs_2\zts_2$ while
the remainder $\rR_*^{(L+1)}$  is of order $L+1$ in $(\zs_1,\zts_1,\zs_2,\zts_2)$ and $\norm{
\rR_*^{(L+1)}}_{p,r} \leqp r^L$.

\label{Th:Birk}
	\end{theorem}
By using the action-angle variables $(\Theta_j,\theta_j)_{j\in\{1,\ldots,4\}}$ such that
 \bes
	\begin{gathered}
	\Theta_j = \Gams_j,\quad \theta_j = \psis_j, \quad \zs_{j} 	=  \sqrt{\Theta_{j+2}} e^{-i\theta_{j+2}}, \quad
		 \zts_{j} = -i\sqrt{\Theta_{j+2}} e^{i\theta_{j+2}},
	\end{gathered}
\ees	
 we obtain a nearly-invariant Lagrangian tori over polynomially long times with respect to $\varepsilon$ at any order.
  Actually, in view of estimates ($\ref{eq:SmallDivisor}$), we can certainly push these reasonings in order to obtain a time of stability of order $\eps^{\ln{\eps}}$.

%\medskip

 Another possible direct extension of our work comes from the fact that our reasonings are valid in the vicinity of the separatrices arising from $L_3$.  Actually, we have considered orbits which surround these separatrices but we can also consider orbits which are inside one of these loops (see Figure \ref{fig:portrait_phase}). In that case, the frequencies have the same order as in the present paper and the same strategy would certainly allows to ensure the existence of quasi-periodic tadpole orbits which are far from the equilateral Lagrange configurations.

%\medskip

 Concerning prospects, in this paper we have shown the existence of invariant tori which are polynomially (with respect to $\eps$) close to  the $L_3$-separatrix.  This is only a  trick that leads us to derive the frequency map, which is, in this case, close to the one evaluated at the $L_3$ equilibrium. Actually our mechanical approximation, defined in Section \ref{sec:mech_syst}, is valid well beyond this separatrix: numerical simulations show that this kind of model is able to approach accurately Janus-Epimetheus actual motion \cite[see][]{RoRaCa2011}. Thus, it would be interesting to build quasi-periodic trajectories with initial conditions close to those of these satellites.

%\medskip

Another natural extension of  our result would be to consider the spatial three-body problem. Using Jacobi reduction, which allows one to eliminate inclinations and ascending nodes for a given value of the angular momentum \cite[see][]{Ro1995}, the spatial problem can be reduced to a four degrees of freedom Hamiltonian system as it is the case for the planar case.
Once reduced, the spatial problem should have properties that are similar to those of the Hamiltonian
studied in the present paper.

%%%%%%%%%%%%%%%%%%%
% Appendix
% Première moyenne
\appendix
\section{Proofs}
\label{sec:proof}

\subsection{Theorem \ref{Th:HKHP}: Estimates on $H_K$, $H_P$}
\label{sec:Proof_HKHP}

By the real analyticity of the transformation in Poincar\'e resonant  complex variable $\Upst\circ \Ups$, there exists $\rho_0>0$ and $\sig_0>0$ such that the differential of its complex extension,
\be
	\Upst\circ\Ups : \quad \Bigg\{
	\begin{array}{ccc}
		\cKh_{\rho_0, \sig_0}               			 &\longrightarrow & \CC^8   \\
		(\bZ, \bzeta, \bx, \bxt)  &\longmapsto     &(\brt_1, \br_1, \brt_2, \br_2 )\, ,
	\end{array}\nnb\\
\ee
admits a norm uniformly bounded on the collisionless domain  $\cKh_{\rho_0, \sig_0} $ (defined in Section \ref{sec:collisionless})  by a constant $C>0$ independent of $\eps$.

In the following, we will denote $D_{\rho_0,\sig_0}$ the image of $\cKh_{\rho_0, \sig_0}$ by the transformation $\Upst\circ \Ups$.

Hence, as
$\norm{(\bZ, \bzeta, \bx, \bxt)}_{\cKh_{\rho_0,\sig_0}} \leq \rho_0 + \sig_0 + 2\sqrt{\rho_0\sig_0}$
then
$$\norm{\br_j - \re(\br_j)}_{D_{\rho_0,\sig_0}} \leq C(\rho_0 + \sig_0 + 2\sqrt{\rho_0\sig_0}) .$$
Thus, one has
\bes
\begin{split}
\norm{\br_1 - \br_2}_{D_{\rho_0,\sig_0}}  &\geq \norm{\re(\br_1) - \re(\br_2)}_{D_{\rho_0,\sig_0}} - \sum_{j\in\{1,2\}}\norm{\br_j - \re(\br_j)}_{D_{\rho_0,\sig_0}} \\
								&\geq \Delta - 8C\sig_0 \geq \frac{\Delta}{2}
\end{split}
\ees
since $\rho_0<\sig_0$ and $\sig_0 \pleq \Deltah$ where $\Deltah$ is an arbitrary fixed value on $\TT$ such that the minimum distance $\Delta$ between two planets in circular motion is reached (see Section \ref{sec:collisionless} for more details).

Consequently,
$ \partial_{\br_j}^l \norm{\br_1- \br_2}_{D_{\rho_0, \sig_0}}^{-1} \leqp \Delta^{-l-1} \leqp 1$
and
\bes
\norm{H_P}_{\sC^{4}} \leqp \frac{\eps}{\Delta^5}\leqp \eps \qtext{on the domain ${\cKh_{\rho_0, \sig_0}}$}
\ees
 as $\Delta$ (resp. $\Deltah$) does not depend on the small parameter $\eps$.

Finally, since $\rho_0 \pleq 1$ then there exists a constant $c>0$ such that $$c\leq \norm{\Lam_{1,0} + Z_1}_{\cK_{\rho_0, \sig_0}} \qtext{and}c\leq \norm{\Lam_{2,0} + Z_2-Z_1}_{\cK_{\rho_0, \sig_0}} $$ which implies that
\bes
\norm{H_K}_{\sC^{4}} \leqp \frac{1}{c^6} \leqp 1   \qtext{on the domain ${\cKh_{\rho_0, \sig_0}}$.}
\ees

\subsection{Theorem \ref{Th:Moy1}: First Averaging Theorem}
 \label{sec:Proof_Moy1}
First of all, we define an iterative lemma of averaging.
Let us introduce some notations:   $(\xi_k)_{k\in\{1,2,3\}}$ are given positive numbers such that
\bes
	0<\xi_1 < \rho, \quad  0 < \xi_2 <\sig, \quad 0<\xi_{3}<\sqrt{\rho\sigma}
\ees
and, for $0\leq r\leq 1$, we denote $\fKh_r$ the domain such as
\bes
	\fKh_r=\cB^2_{\rho-r\xi_1}\times \cV_{\sigma-r\xi_2}\cIh \times \cB^4_{\sqrt{\rho\sig} - r\xi_{3}} .
\ees

Hence, we set out the following
\begin{lemma}[First Iterative Lemma]
\label{Lem:Moy1}
Let $\rho^-$, $\sig^-$, $\xi_1$, $\xi_2$ be fixed positive real numbers that depend on the small parameter $\eps$ and
\be
	\begin{split}
	\rho^+ &= \rho^- - \xi_1>0,\quad
	\sig^+ = \sig^- - \xi_2>0,  \\
	\xi_{3}&=
	\sqrt{\rho^-\sig^-} - \sqrt{\rho^+\sig^+} .
	\end{split}
	\label{eq:LemMoy1_cond0}
\ee
Let $H^-$ be a Hamiltonian of the form
\bes
\begin{split}
	H^-(\bZ, \bzeta, \bx, \bxt) = H_K(\bZ) + \Hb_P(\bZ, \zeta_1, \bx, \bxt) &+  {H_{*}^{0,-}}(\bZ, \zeta_1, \bx, \bxt) \\
													  &+ {H_{*}^{1,-}}(\bZ, \bzeta, \bx, \bxt)
\end{split}
\ees
which is analytic on the domain $\fKh_0^-=\cKh_{\rho^-, \sig^-}$ and such that
\bes
	{\Hb_{*}^{1,-}}(\bZ, \zeta_1, \bx, \bxt) = \frac{1}{2\pi}\int_0^{2\pi} {H_{*}^{1,-}}(\bZ, \zeta_1, \zeta_2, \bx, \bxt)\rd\zeta_2 =0 .
\ees
Let  $\eta^-$, $(\mu_l^-)_{l\in\{0,1,2,3\}}$ be fixed positive real numbers, which depend on $\eps$, such that
\be
	\begin{split}
	\norm{{H_{*}^{1,-}}}_{\fKh_0^-}\leq\eta^-, \quad 	
	\norm{{H_{*}^{0,-}}}_{ \fKh_0^-}\leq \mu_0^-
	\end{split}
	\label{eq:LemMoy1_Binit1}
\ee
and
\bes
	\begin{aligned}
	&\norm{\partial_{\bZ}\big(\Hb_P + {H_{*}^{0,-}}\big)}_{ \fKh_0^-}\leq \mu_1^-,&\quad
	&\norm{\partial_{\bzeta}\big(\Hb_P +{H_{*}^{0,-}}\big)}_{\fKh_0^-}\leq\mu_2^-,&\\
	&\norm{\partial_{(\bx,\bxt)}\big(\Hb + {H_{*}^{0,-}}\big)}_{ \fKh_0^-}\leq\mu_3^- .&
	%\label{eq:LemMoy1_Binit2}
	\end{aligned}
\ees

If we assume that
\be
	\eta^- \pleq \xi_1 \xi_2
	\label{eq:LemMoy1_CCond}
\ee
then there exists a canonical transformation
\begin{gather}
	\Upsb^+ : \quad \Bigg\{
	\begin{array}{ccc}
		\fKh_1^-                			 &\longrightarrow & \fKh_0^-   \\
		(\bZs, \bzetas, \bxs, \bxts)  &\longmapsto     &(\bZ, \bzeta, \bx, \bxt )
	\end{array}\nnb\\
\mbox{with} \quad 	\fKh_{2/3}^- \subseteq \Upsb^+(\fKh_{1/2}^-) \subseteq \fKh_{1/3}^- \label{eq:LemMoy1_emboitement}	
\end{gather}
and such that, in the new variables, the Hamiltonian
$H^+= H^- \circ \Upsb^+$ can be written
\bes
	\begin{split}
	H^+	=& H_K + \Hb_P + {H_{*}^{0,-}}  + H_*^+ \\
		=& H_K + \Hb_P + {H_{*}^{0,+}} + {H_{*}^{1,+}}
	\end{split}
\qtext{with}
	\left\{\begin{array}{l}
		{H_{*}^{0,+}} = {H_{*}^{0,-}} + \Hb_{*}^+ \\
		{H_{*}^{1,+}} = H_*^+      - \Hb_*^+
	\end{array} \right.
\ees
and
\bes
	\Hb_{*}^{+}(\bZs, \zetas_1, \bxs, \bxts) = \frac{1}{2\pi}\int_0^{2\pi} H_{*}^{+}(\bZs, \zetas_1, \zetas_2, \bxs, \bxts)\rd\zetas_2 .
\ees

Furthermore, we have the thresholds
\be
	\begin{split}
	\norm{{H_{*}^{1,+}}}_{\fKh_1^-}\leq \eta^+ , \quad \norm{{H_{*}^{0,+}}}_{ \fKh_1^-}\leq\mu_0^+ ,
	\end{split}
	\label{eq:LemMoy1_Bfinal1}
\ee
and
\be
	\begin{aligned}
	&\norm{\partial_{\bZ}\big(\Hb_P + {H_{*}^{0,+}}\big)}_{ \fKh_1^-}\leq \mu_1^+ ,& \quad
	&\norm{\partial_{\bzeta}\big(\Hb_P + {H_{*}^{0,+}}\big)}_{\fKh_1^-}\leq\mu_2^+ ,&\\
	&\norm{\partial_{(\bx, \bxt)}\big(\Hb_P + {H_{*}^{0,+}}\big)}_{ \fKh_1^-}\leq \mu_3^+,&
	\end{aligned}
	\label{eq:LemMoy1_Bfinal2}
\ee
with the following quantities:
\be
\begin{gathered}
 	\eta^+ \eqp \eta^-\left( \theta^+ + \frac{\rho^-}{\xi_2} \right), \quad 	\mu_0^+ - \mu_0^- \eqp \eta^-\theta^+ , \\	
\begin{aligned}
	\mu_l^+-  \mu_l^- &\eqp \eta^-\frac{\theta^+}{\xi_l} \quad (l\in\{1,2,3\}),\\
	\mbox{and}\quad \theta^+ &= \frac{\mu_1^-}{\xi_2} + \frac{\mu_2^-}{\xi_1} + \frac{\mu_3^-}{\xi_{3}} + \frac{\eta^-}{\xi_1 \xi_2}.
	\label{eq:LemMoy1_Bfinal3}
\end{aligned}
\end{gathered}
\ee

\end{lemma}
\begin{proof}
We define $\Upsb^+:\fKh_1^-\longrightarrow \fKh_0^- $
that is the time-one map  of the Hamiltonian flow generated by the auxiliary function $\chi^+$, i.e. $\Upsb^+=\Phi_1^{\chi^+}$ with
\bes
	{\chi^+}(\bZ,\bzeta, \bx, \bxt) = \frac{2\pi}{\upsilon_0}\int_0^1 s {H_{*}^{1,-}}(\bZ,\zeta_1, \zeta_2 + 2\pi s, \bx, \bxt)\rd s
\ees
such that
\be
\begin{gathered}
\Poi{\chi^+}{\upsilon_0 \Zs_2} + H_*^{1,-} = 0
\qtext{and}\\
 \overline{\chi}^+(\bZ, \zeta_1,\bx, \bxt) = \frac{1}{2\pi}\int_0^{2\pi} {\chi^+}(\bZ, \zeta_1, \zeta_2,\bx, \bxt) \rd\zeta_2 = 0 .
\end{gathered}
\label{eq:LemMoy1_prop}
\ee
Thus, in the new variables, the Hamiltonian  reads
\be
	\begin{split}
	H^+  &=H^-\circ\Upsb^+
		=H^- + H^-\circ\Phi_1^{\chi^+} - H^-\\
		&= H^- + \upsilon_0\Zs_2\circ\Phi_1^{\chi^+} - \upsilon_0\Zs_2 + (H^- - \upsilon_0\Zs_2)\circ\Phi_1^{\chi^+}-H^- + \upsilon_0\Zs_2 \\
		&= H_K + \Hb_P + H_*^{0,-} + \underbrace{H_*^{1,-} + \Poi{\chi^+}{\upsilon_0\Zs_2}}_{(*)} + H_*^+	\nnb\end{split}
\ee
with the remainder
\be
	\begin{split}
	H_*^+ 	& = \int_0^1(1-s) \Poi{\chi^+}{\Poi{\chi^+}{\upsilon_0\Zs_2}}\circ\Phi_s^{\chi^+} \rd s + \int_0^1\Poi{\chi^+}{H^- - \upsilon_0\Zs_2}\circ\Phi^{\chi^+}_s\rd s\\
			& = \int_0^1 \Poi{{\chi^+}}{H_K-\upsilon_0 \Zs_2 + \Hb_P + {H_{*}^{0,-}} + s{H_{*}^{1,-}}}\circ \Phi_s^{\chi^+} \rd s\nnb
	\end{split}
\ee
that is given by the equations \eqref{eq:Taylor0} and \eqref{eq:Taylor1} while $(*)$ is equal to zero by \eqref{eq:LemMoy1_prop}.

We have to estimate the size of $H_*^+$ to prove the thresholds \eqref{eq:LemMoy1_Bfinal1} and \eqref{eq:LemMoy1_Bfinal2}.
Firstly, by the conditions \eqref{eq:LemMoy1_Binit1}, we have
$\norm{{\chi^+}}_{\fKh_0^-} \leqp  \frac{\eta^-}{\upsilon_0} \leqp \eta^-\nnb$
as $\upsilon_0=\cO(1)$.
One then applies the Cauchy inequalities to obtain the partial derivatives
\be
	\norm{\partial_{\bZ}  {\chi^+}}_{\fKh_{1/2}^-} \leqp \frac{\eta^-}{\xi_1} , \quad
	\norm{\partial_{\bzeta} {\chi^+}}_{\fKh_{1/2}^-} \leqp \frac{\eta^-}{\xi_2} , \quad
	\norm{\partial_{(\bx,\bxt)}  {\chi^+}}_{\fKh_{1/2}^-} \leqp \frac{\eta^-}{\xi_{3}}, \nnb
\ee
and deduces the estimates on the Poisson brackets
\be
	\norm{\Poi{{\chi^+}}{H_K-\upsilon_0 \Zs_2} }_{\fKh_{1/2}^-} \leqp  \frac{\eta^-\rho^-}{\xi_2}\nnb
\ee
(by the threshold $\norm{\partial_{\bZ} H_K - (0,\upsilon_0)}_{\fKh_0^-} \leqp \rho^- $ given by \eqref{eq:HKHP_bounds} and the mean value theorem),
\be
	\norm{\Poi{{\chi^+}}{{H_{*}^{1,-}}} }_{\fKh_{1/2}^-}  \leqp \frac{(\eta^-)^2}{\xi_1\xi_2}\nnb
\ee
(as \eqref{eq:LemMoy1_cond0} implies that 	$(\xi_{3})^2 \geq \xi_1 \xi_2$), and
\be
	\norm{\Poi{{\chi^+}}{\Hb_P + {H_{*}^{0,-}}} }_{\fKh_{1/2}^-}           \leqp \eta^-\Big(\frac{\mu^-_1}{\xi_2} + \frac{\mu^-_2}{\xi_1} + \frac{\mu^-_3}{\xi_{3}}\Big).\nnb
\ee
As a consequence, the remainder of the transformation $\Upsb^+$ is bounded such that
\be
	\norm{H^+_*}_{\fKh_{1/2}^-}    \leqp  \eta^- \Big(\theta^+  + \frac{\rho^-}{\xi_2}\Big)\nnb
\ee
where $\theta^+$ is given by \eqref{eq:LemMoy1_Bfinal3}.
Moreover, taking into account that $\overline{\chi}^+=0$ (given by \eqref{eq:LemMoy1_prop}), we have
\be
	\Hb_{*}^+(\bZ, \zeta_1, \bx, \bxt) = \frac{1}{2\pi} \int_0^{2\pi} \int_0^1 s \Poi{{\chi^+}}{{H_{*}^{1,-}}}\circ\Phi_s^{\chi^+}(\bZ, \zeta_1, \tau, \bx, \bxt) \rd s \rd\tau\nnb
\ee
and therefore
\be
	\norm{\Hb_{*}^+}_{\fKh_{1/2}^-} \leqp \frac{(\eta^-)^2}{\xi_1\xi_2} \leqp \eta^- \theta^+ .\nnb
\ee
Hence,  if we denote
${H_{*}^{0,+}}  = {H_{*}^{0,-}} + \Hb_{*}^+$ and
${H_{*}^{1,+}}  = H_*^+      - \Hb_*^+$ then the triangle inequality gives the estimates (\ref{eq:LemMoy1_Bfinal1}) and (\ref{eq:LemMoy1_Bfinal2}) (together with the Cauchy inequalities for the last).

Finally, by the equation \eqref{eq:Taylor0} and the Cauchy inequalities, we can estimate the size of the transformation $\Upsb^+$.
Hence, the condition \eqref{eq:LemMoy1_CCond} provides the following estimates
\be
\begin{gathered}
	\norm{\bZs - \bZ}_{\fKh_{1/2}^-} \leqp  \frac{\eta^-}{\xi_2} \leq \frac{\xi_1}{6} ,\quad
	\norm{\bzetas - \bzeta}_{\fKh_{1/2}^-} \leqp  \frac{\eta^-}{\xi_1} \leq \frac{\xi_2}{6} , \\
	\mbox{and}\quad
	\norm{(\bxs,\bxts) - (\bx,\bxt)}_{\fKh_{1/2}^-} \leqp  \frac{\eta^-}{\xi_3}\leq \frac{\xi_{3}}{6}\nnb
\end{gathered}
\ee
 which yields \eqref{eq:LemMoy1_emboitement}.
\end{proof}

\medskip

Now, in order to prove Theorem \ref{Th:Moy1}, one applies a first time Lemma \ref{Lem:Moy1}.
Thus, we define the following
\bes
	(\xi_1, \xi_2, \xi_3) = \frac{\sigma_0}{3}(\eps^\beta, 1, \eps^{\beta/2}) \qtext{for} 1/7<\beta<1/2
 \ees
such that  $\fKh_r = \cKh_{1-\frac{r}{3}}$ for $0\leq r\leq 1$.
By Theorem \ref{Th:HKHP} and the notations of Lemma \ref{Lem:Moy1}, the Hamiltonian $H$ is analytical on $\cKh_1$ and of the form
\be
	H(\bZ,\bzeta,\bx,\bxt) = H_K(\bZ) + \Hb_P(\bZ,\zeta_1,\bx,\bxt)  +  \big[H_P - \Hb_P\big](\bZ,\bzeta,\bx,\bxt)\nnb
\ee
with
\be
	\eta^- \leqp \eps ,\quad \mu^-_0 = 0 \qtext{and} \mu_l^-\leqp \eps \qtext{for $l\in\{1,2,3\}$.}\nnb
\ee
Hence, the condition \eqref{eq:LemMoy1_CCond} is fulfilled and Lemma \ref{Lem:Moy1} provides the existence of the transformation $\Upsb^{0}: \cKh_{2/3} \longrightarrow \cKh_{1}$
 such that
 $$H^{0}=H\circ \Upsb^{0}=  H_K + \Hb_P +  H_{*}^{0,0} + H_{*}^{1,0}$$
with the following thresholds:
\be
	\norm{H_{*}^{1,0}}_{2/3}\leq \eta^{0}\leqp \eps^{1-\beta} , \quad
	\norm{H_{*}^{0,0}}_{2/3}\leq\mu_0^{0} \leqp \eps^{2-\beta},
	\label{eq:Moy1_B1}
\ee
and
\bes
	\begin{aligned}
	&\norm{\partial_{\bZ}\big(\Hb_P + H_{*}^{0,0}\big)}_{2/3}\leq \mu_1^{0}\leqp \eps ,& \quad
	&\norm{\partial_{\bzeta}\big(\Hb_P + H_{*}^{0,0}\big)}_{2/3}\leq\mu_2^{0} \leqp \eps ,&\\
	&\norm{\partial_{(\bx, \bxt)}\big(\Hb_P + H_{*}^{0,0}\big)}_{ 2/3}\leq \mu_3^{0} \leqp \eps  .&
	\end{aligned}
	%\label{eq:Moy1_B2}
\ees
Moreover, by the equation \eqref{eq:Taylor0} and the Cauchy inequalities, one has:
\bes
\begin{gathered}	
	\norm{\bZs - \bZ}_{{2/3}} \leqp \eps ,\quad
	\norm{\bzetas - \bzeta}_{{2/3}} \leqp \eps^{1-\beta} ,\\
	\mbox{and}\quad\norm{(\bxs,\bxts) - (\bx,\bxt)}_{{2/3}} \leqp \eps^{1-\beta/2} .
	%\label{eq:Moy1_transf0}
\end{gathered}
\ees

\medskip

Then, we apply iteratively Lemma \ref{Lem:Moy1} to reduce the fast component of the Hamiltonian until an exponentially small size with respect to $\eps$.
To do so, let $s$ be a non-zero integer such that $s = \Ent(\eps^{-\alpha}) + 1$
where
\bes
	\alpha = \frac{1-2\beta}{5} \qtext{for} 1/7<\beta<1/2 .
\ees
We define
\bes
	(\xi_1, \xi_2, \xi_3) = \frac{\sig_0}{3s}(\eps^\beta, 1, \eps^{\beta/2})
\ees
as well as the sequences $(\rho^{j})_{j\in\{0,1,\ldots,s\}}$, $(\sig^{j})_{j\in\{0,1,\ldots,s\}}$ with
\be
	(\rho^{j}, \sigma^{j}) = \frac{2s-j}{3s} (\rho,\sig) \qtext{for} j\in\{1,\ldots,s\}\nnb
\ee
such that	$\fKh^j_r = \cKh_{\frac{2}{3}-\frac{j + r}{3s}}$ for $0\leq r\leq 1$.

Replacing the notation $ ^-$ and $ ^+$ of Lemma \ref{Lem:Moy1} by $^{j-1}$ and $^{j}$ and
 assuming that for all $0< j\leq s$ the following condition (associated with \eqref{eq:LemMoy1_CCond}) is fulfilled:
\be
	\eta^{j-1} \pleq \eps^{\beta-2\alpha} ,
	\label{eq:Moy1_condLem}	
\ee
an iterative application of Lemma \ref{Lem:Moy1} to the Hamiltonian $H^{0}$ provides a sequence of canonical transformations $(\Upsb^{j})_{j\in\{1,\ldots,s\}}$ such that  $H^0\circ\Upsb^{1}\circ\ldots\circ\Upsb^{s}$ is equal to the Hamiltonian of the formula \eqref{eq:Moy1_HamF} with
\be
	H_*(\bZs,\bzetas,\bxs, \bxts) = H_{*}^{0,s}(\bZs,\zetas_1,\bxs, \bxts) + H_{*}^{1,s}(\bZs,\bzetas,\bxs, \bxts) .\nnb
\ee

In order to complete the proof, let us consider $n\in\{1,\ldots,s\}$ such that the sequences $(\eta^{j})_{j\in\{1,\ldots,n\}}$ and $(\mu_l^{j})_{j\in\{1,\ldots, n\}}$ satisfy the following induction hypothesis:
\bes
	\eta^{j}\leq\eta^{j-1}\exp(-1) \qtext{and}
	\mu_l^{j} - \mu_l^{j-1} \leq \frac{\mu_l^{0}}{s} \quad (l\in\{1,2,3\})  .
	%\label{eq:Moy1_hyprec}
\ees

For $n=1$,  as $0<\alpha<1/7<\beta<1/2$, \eqref{eq:Moy1_condLem} is satisfied and $\theta^{1} \leqp\eps^{1-\beta-\alpha}$ implies that
\be
	\eta^{1} 	 \leqp \eta^{0} \eps^{\beta-\alpha}
			\leq \eta^{0}\exp(-1)\nnb
\ee
for $\eps\pleq \exp(-\frac{1}{\beta-\alpha})$
and
\be 	
	\mu_l^{1} - \mu_l^{0} 	\leqp \eps^{2- \beta-2\alpha}
						\leqp \frac{\mu_l^{0}}{s}\eps^{1-\beta -3\alpha}
						\leq \frac{\mu_l^{0}}{s}\quad (l\in\{1,2,3\})\nnb
\ee
for $\eps \pleq 1$.

For a fixed integer $n$, $(\eta^j)_{j\in\{0,\ldots,n\}}$ is decreasing
while
\be
\mu_l^{n} 	\leq \mu_l^{n-1} + \frac{\mu_l^{0}}{s}
			\leq \ldots
			\leq \mu_l^{0} + n\frac{\mu_l^{0}}{s}	
			\leq 2 \mu_l^{0}\quad (l\in\{1,2,3\}),	\nnb
\ee
then the induction is immediate.
Indeed, \eqref{eq:Moy1_condLem} is satisfied  and
$ \theta^{n+1} \leq \theta^{0} \leqp \eps^{1-\beta-\alpha}$ implies that
\be
\begin{gathered}
	\eta^{n+1}  	\leqp \eta^{n} \eps^{\beta-\alpha}
				\leq	\eta^{n} \exp(-1) \qtext{and}\\
	\mu_l^{n+1} - \mu_l^{n} 	\leqp\frac{\mu_l^{0}}{s}\eps^{1-\beta-3\alpha}
								\leq \frac{\mu_l^{0}}{s}\quad (l\in\{1,2,3\}) .\nnb
\end{gathered}
\ee

As a consequence, we have
\be
	\begin{gathered}
	\eta^{s} 	\leqp \eta^{s-1}\exp(-1)
			\leq 	\ldots
			\leqp \eta^{0}\exp(-s)
			\leqp	\eps^{1+\beta} \exp (-\frac{1}{\eps^{\alpha}})\\
	\mbox{and} \quad \mu_l^{s}
				\leqp \eps  \quad (l\in\{1,2,3\})\nnb
	\end{gathered}
\ee
which prove (\ref{eq:Moy1_rem_exp}).
Likewise,
\be
	\mu_0^{s}  - \mu_0^{0}	\leq  \mu_0^{s}  - \mu_0^{s-1} + \ldots + \mu_0^{1}  - \mu_0^{0}
						\leqp \eta^{0} \theta^{1}
						\leqp \eps^{2-\alpha}\nnb
\ee
and then
$\mu_0^{s} 					\leqp \eps^{2-\beta} $ proves  (\ref{eq:Moy1_rem_moy}).

At last, and in the same way as for the first application of Lemma \ref{Lem:Moy1}, for each transformation $\Upsb^j$ with $j\in\{1, \ldots, s\}$, the equation \eqref{eq:Taylor0} and Cauchy inequalities lead to
\be
	\begin{gathered}
	\norm{\bZs - \bZ}_{{1/3}} \leqp \eps^{1+\beta- \alpha}  ,\quad
	\norm{\bzetas - \bzeta}_{{1/3}} \leqp \eps^{1-\alpha} , \\
	\mbox{and}\quad
	\norm{(\bxs,\bxts) - (\bx,\bxt)}_{{1/3}} \leqp \eps^{1+\beta/2- \alpha} \nnb
	\end{gathered}
\ee
Consequently, the size of the transformation $\Upsb$ is dominated by that of the transformation $\Upsb^{0}$ which provides  the estimates \eqref{eq:Moy1_transf} and yields \eqref{eq:Moy1_emboitement}.

\subsection{Lemma \ref{lem:DAl}: D'Alembert rule in the Averaged Problem}
\label{sec:Proof_Lem_DAl}

The D'Alembert rule, given by (\ref{eq:DAl}), derives from the preservation of the angular momentum denoted $\cCt = \sum_{j\in\{1,2\}} \brt_j \times \br_j$. By the transformation in the resonant Poincar\'e complex variables $\Upst\circ \Ups$, we have $ \cC(\bZ,\bzeta,\bx,\bxt ) =  \cCt \circ \Upst \circ \Ups(\bZ,\bzeta,\bx,\bxt ) = Z_2 + ix_1\xt_1 + ix_2\xt_2$.
$\cCt$ being an integral of the motion, it turns out that
\be
0 = \{\cCt , \cH\} =  \{ \cC  , \cH \circ \Upst \circ \Ups \} =  \{ \cC , H \} .
\label{eq:poiss_nul}
\ee
Injecting the expansion (\ref{eq:Fourier_Taylor}) in (\ref{eq:poiss_nul}) we get
\be
\begin{split}
	0 &= \{ \cC ,  \sum_{(k,\bp,\bpt) \in \sD}\hspace{-12pt} f_{k,\bp,\bpt}(\bZ,\zeta_1)x_1^{p_1}x_2^{p_2}\xt_1^{\pt_1}\xt_2^{\pt_2}\exp (i k\zeta_2) \} \\
 & = - i  \hspace{-14pt}\sum_{(k,\bp,\bpt) \in \sD} \hspace{-12pt} ( k + p_1 -\pt_1 + p_2 - \pt_2 )
 f_{k,\bp,\bpt}(\bZ,\zeta_1)x_1^{p_1}x_2^{p_2}\xt_1^{\pt_1}\xt_2^{\pt_2}\exp (i k\zeta_2) .
\end{split}\nnb
\ee
As a consequence, one has
\bes
 k + p_1 -\pt_1 + p_2 - \pt_2 = 0.
 \ees

In order to  prove Lemma 1,   it only needs to be shown that the expression of $\cCb =   \cC\circ \Upsb$ is equal to
 $\cC(\bZs,\bzetas, \bxs, \bxts) =  \Zs_2 + i\xs_1\xts_1 + i\xs_2\xts_2$.
As the averaging transformation $ \Upsb$ is generated by the composition of the transformations $(\Phi_1^{\chi^{j}})_{j\in\{0,\ldots,s\}}$ (see Section \ref{sec:Proof_Moy1}), the result holds if $\{ {\mathcal \chi^{j}} , \cCb \} = 0$.

At first iteration, the generating function $\chi^{0}$ reads
$$
\chi^{0}(\bZ,\bzeta,\bx,\bxt )  = \frac{2\pi}{\upsilon_0}\int_0^1 s \left[H_P - \Hb_P\right]_{(\bZ,\zeta_1,\zeta_2 + 2\pi s,\bx,\bxt)} \rd s .
$$
As $H_P$ satisfies the D'Alembert rule, one has
$$
\{ \chi^{0} , \cC \} =  \frac{2\pi}{\upsilon_0}\int_0^1 s\{ H_P - \Hb_P, \cC\}\rd s = 0 ,
$$
which leads to: $ \cC\circ\Phi_1^{\chi^{0}} = \cC$.
The same holds true for the other iterations.

Finally, let a real function  $f$ that satisfies the D'Alembert rule and does not depend on the fast angle $\zetas_2$.
Hence, the total degree in $\xs_1$, $\xs_2$, $\xts_1$, $\xts_2$ in the monomials appearing in the Taylor expansion of $f$ in neighborhood of $\bxs= \bxts = \bzero$ is even.
As a consequence $f$ can be decomposed such as $f= f_0 +f_2$ with the properties \eqref{eq:Dal_f}.

\subsection{Theorem \ref{Th:Approx}: Reduction}
The Hamiltonian of Theorem \ref{Th:Approx} is obtained by a suitable  expansion of the averaged Hamiltonian $\Hb$ in the neighborhood of the quasi-circular manifold $\rC_0$.

First of all, by Lemma \ref{lem:DAl},  $\Hb$ and $\Hb_*$ can be decomposed respectively such as
\be
\begin{split}
	\Hb_P(\bZs, \zetas_1, \bxs, \bxts) &= \Hb_{P,0}(\bZs, \zetas_1) + \sum_{j,k \in\{ 1,2\}} \Hb_{P,(j,k)}(\bZs, \zetas_1, \bxs, \bxts) \xs_j\xts_k \\
\mbox{and}\quad
\Hb_*(\bZs, \zetas_1, \bxs, \bxts)  &= \Hb_{*,0}(\bZs, \zetas_1)  + \sum_{j,k \in\{ 1,2\}} \Hb_{*,(j,k)}(\bZs, \zetas_1, \bxs, \bxts) \xs_j\xts_k .\nnb
\end{split}
\ee

Regarding the eccentricities, a polynomial expansion of $\Hb$ of the degree two with respect to $\bxs=\bxts=\bzero$ provides
\be
\begin{gathered}
	\Hb_{P,(j,k)} = \Hb_{P,(j,k)}( \,\cdot \,,\,  \cdot \,,\bzero, \bzero) + R^{1}_{P,(j,k)} \\
\begin{aligned}
&\mbox{with} \quad R^1_{P,(j,k)}(\bZs, \zetas_1, \bxs, \bxts) =g_{(j,k)}(1) - g_{(j,k)}(0)\\
&\mbox{and} \quad g_{(j,k)}(t) = \Hb_{P,(j,k)} (\bZs, \zetas_1, t\bxs, t\bxts) .\nnb
\end{aligned}
\end{gathered}
\ee
The size of the remainder involved in this approximation is estimated thanks to the mean value theorem applied on the function
$g_{(j,k)}$ for $(t, \bZs, \bzetas, \bxs, \bxts)\in[0,1]\times \cK_{1/3}$ together with the bound  \eqref{eq:HKHP_bounds} of Theorem \ref{Th:HKHP}.
Hence, this yields
\bes
	\norm{R^{1}_{P,(j,k)} }_{1/3} \leqp \eps^{1+\beta} .%\label{eq:R1P}
\ees

Now, we consider the expansion of $\Hb$ with respect to the exact resonant action $\bZs = \bzero$.
The Keplerian part can be written:
\be
\begin{gathered}
H_K= H_K(\bzero) + \upsilon_0\Zs_2 + Q+ R^{2}_K + R^{3}_K\\
\begin{aligned}
\mbox{with} \quad R^2_K(\bZs) &= H_K -H_K(\bzero) - \upsilon_0\Zs_2 - \Qt(\bZs),\\
\mbox{and} \quad R^3_K(\bZs) &= \Qt(\bZs) - Q(\bZs)\nnb
\end{aligned}
\end{gathered}
\ee
where the quadratic form $\Qt$ reads
\be
\begin{gathered}
	\Qt(\bZs) = -\upsilon_0 \At\left(\Zs_1^2 + (1-\kappat)(-2\Zs_1\Zs_2 + \Zs_2^2)\right) \qtext{with}\\
\At 			= \frac{3}{2}\upsilon_0^{1/3} \left(\hm_1^{-1}\mu_1^{-2/3} + \hm_2^{-1}\mu_2^{-2/3}\right),\quad
1- \kappat 	= \frac{\hm_2^{-1}\mu_2^{-2/3}}{\hm_1^{-1}\mu_1^{-2/3} + \hm_2^{-1}\mu_2^{-2/3}} ,
\end{gathered}\nnb
\ee
and its approximation
\be
\begin{gathered}
	Q(\bZs) = -\upsilon_0A\left(\Zs_1^2 + (1-\kappa)(-2\Zs_1\Zs_2 + \Zs_2^2)\right)\qtext{with}\\
A		=\frac{3}{2}\upsilon_0^{1/3}m_0^{-2/3}\left(\frac{1}{m_1} + \frac{1}{m_2} \right),\quad
1- \kappa 	= \frac{m_2}{m_1+m_2} .\nnb
\end{gathered}
\ee
The application of the Taylor formula on the function $g(t)= H_K(t\bZ)$ for $(t,\bZ)\in[0,1]\times\cB^2_{\rho/3}$ leads to
\be
	R^2_K(\bZs) = \int_0^1\frac{(1-t)^2}{2}g^{(3)}(t) \rd t\nnb
\ee
and, together with the bound  \eqref{eq:HKHP_bounds}, provides the estimates
\bes
	\norm{R^{2}_K}_{1/3} \leqp \eps^{3\beta} .%\label{eq:R2K}
\ees
Regarding the estimate of $R^3_K$, as
\be
\hm_j = m_j + \cO(\eps), \quad \mu_j = m_0 + \cO(\eps), \label{eq:Approx_masse}
\ee
then $\At - A = \cO(\eps)$ and $\kappat- \kappa = \cO(\eps)$ provide the following bound:
\bes
	\norm{R^3_K}_{1/3} \leqp \eps^{3\beta}\label{eq:R3K} \qtext{since $\beta<1/2$.}
\ees
In the case of the the perturbation part, one can split $\Hb_{P,0}$ and  $(\Hb_{P,(j,k)})_{1\leq j,k,\leq 2}$ in the sum of three terms as follows:
\be
\begin{gathered}
	\Hb_{P,0} =G_0 + R^{2}_{P,0} + R^{3}_{P,0} ,\\
\begin{aligned}
	&\mbox{with}\quad R^2_{P,0}(\bZs, \zetas_1) 	= \Hb_{P,0}(\bZs, \zetas_1) - \Hb_{P,0}(\bzero,\zetas_1)\\
&\mbox{and} \quad R^3_{P,0}(\zetas_1) 			= \Hb_{P,0}(\bzero,\zetas_1) - G_0(\zetas_1),\nnb
\end{aligned}
\end{gathered}
\ee
and
\be
\begin{gathered}
	 \Hb_{P,(j,k)}(\, \cdot\,, \,\cdot\, ,\bzero, \bzero) = G_{(j,k)} + R^{2}_{P,(j,k)}  + R^{3}_{P,(j,k)}\\
	 \begin{aligned}
&\mbox{with} \quad R^2_{P,(j,k)}(\bZs, \zetas_1)	= \Hb_{P,(j,k)}(\bZs, \zetas_1, \bzero, \bzero) - \Hb_{P,(j,k)}(\bzero, \zetas_1, \bzero, \bzero)\\
&\mbox{and} \quad R^3_{P,(j,k)}(\zetas_1) 		= \Hb_{P,(j,k)}(\bzero, \zetas_1, \bzero, \bzero)- G_{(j,k)}(\zetas_1) \nnb
\end{aligned}
\end{gathered}
\ee
where
\be
\begin{split}
	G_0(\zetas_1) &= \eps m_1m_2 \left( -\frac{1}{D_0(\zetas_1)} + \frac{\cos \zetas_1}{\sqrt{a_{1,0}a_{2,0}}}\right)
\end{split}\nnb
\ee
and
\be
\left(G_{(j,k)}\right)_{1\leq j,k\leq 2} = i\eps \frac{m_1m_2}{\sqrt{m_0}}
 \begin{pmatrix}
	 \displaystyle\frac{A_0}{m_1a_{1,0}^{1/2}}		&	 \displaystyle\frac{B_0}{\sqrt{m_1m_2}(a_{1,0}a_{2,0})^{1/4}}	 \\
	\displaystyle\frac{\conj(B_0)}{\sqrt{m_1m_2}(a_{1,0}a_{2,0})^{1/4}}	& \displaystyle\frac{A_0}{m_1a_{1,0}^{1/2}}\end{pmatrix}\nnb
\ee
 with
 \be
 \begin{split}
&A_0= \frac{a_{1,0}a_{2,0}}{4D_0^5}\left( a_{1,0}a_{2,0}(5\cos 2 \zetas_1 - 13) + 4(a_{1,0}^2 + a_{2,0}^2)\cos\zetas_1 \right) - \frac{\cos \zetas_1}{\sqrt{a_{1,0}a_{2,0}}} , \\
&B_0 = - \frac{a_{1,0}a_{2,0}}{8D_0^5}\left(a_{1,0}a_{2,0}\left(e^{-3i\zetas_1} - 26 e^{-i\zetas_1} + 9 e^{i\zetas_1} \right)+ 8(a_{1,0}^2 + a_{2,0}^2)e^{-2i\zetas_1}\right)\\
	&\phantom{B_0 =} + \frac{e^{-2i\zetas_1}}{\sqrt{a_{1,0}a_{2,0}}} ,\\
&\mbox{ $\conj(B_0)$ that is the complex conjugate of $B_0$, }\\
&D_0  = \sqrt{a_{1,0}^2 + a_{2,0}^2 - 2a_{1,0}a_{2,0}\cos \zetas_1} .
\end{split}\nnb
\ee
 With similar reasonings as for the eccentricities, we use the mean value theorem to evaluate the remainder in the truncation at order 0 of $\Hb_{P,0}$ and $(\Hb_{P,(j,k)})_{1\leq j,k,\leq 2}$.
 Hence, this yields
\bes
\begin{gathered}
	\norm{R^{2}_{P,0}}_{1/3}  \leqp \eps^{3\beta}  \qtext{since $\beta <1/2$}\\
	\mbox{and} \quad \norm{R^{2}_{P,(j,k)} }_{1/3} \leqp \eps \rho \leqp \eps^{1+\beta} .
\end{gathered}%\label{eq:R2P}
\ees
Moreover the following estimates:
\bes
	\norm{R^{3}_{P,0}}_{1/3} \leqp \eps^2 , \quad
	\norm{R^{3}_{P,(j,k)} }_{1/3} \leqp \eps^2,
%\label{eq:R3P}
\ees
are obtained with the approximation of the formula \eqref{eq:Approx_masse}.

Finally, in order to get a more tractable expression, one can shift the perturbation parts to
\be
\bZs =\bZ_\star = \left(
		\begin{array}{c}
			\Lam_{1,0} - \Lam_{1,\star} \\
			\Lam_{1,0} + \Lam_{2,0} - (\Lam_{1,\star} + \Lam_{1,\star})
		\end{array}
		\right)\nnb
\ee
with $\Lam_{j,\star} = \hm_j \mu_j^{1/2}m_0^{1/6}\upsilon_0^{-1/3}$
where the two associated semi-major axes are both equal to the same value given by $a_\star = m_0^{1/3} \upsilon_0^{-2/3}$.
This yields
\be
G_0= \eps \upsilon_0B \cF + R^4_{P,0} \qtext{and}
G_{(j,k)} = \sQt_{(j,k)}  + R^4_{P,(j,k)}\nnb
\ee
with the following thresholds:
\bes
	\norm{R^{4}_{P,0} }_{1/3} \leqp \eps^2
\qtext{and}
	\norm{R^{4}_{P,(j,k)} }_{1/3} \leqp \eps^2
%	\label{eq:R4P}
\ees
that are estimated thanks to the bound $0\leq \Lam_{j,0} - \Lam_{j,\star} \leqp \eps$.

As a consequence,
\begin{lemma}
the averaged Hamiltonian can be written
\bes
\begin{split}
\Hb(\bZs, \zetas_1, \bxs, \bxts) = H_K(\bzero) +& \upsilon_0\left(\Zs_2 - A Q(\bZs) + \eps B \cF(\zetas_1)\right) 	\\+& \sQt(\zetas_1, \bxs, \bxts)
					+ R(\bZs, \zetas_1, \bxs, \bxts)
\end{split}
\ees
with $R(\bZs, \zetas_1, \bxs, \bxts) =R_0(\bZs, \zetas_1) + \sum_{j,k \{1,2\}} R_{(j,k)}(\bZs, \zetas_1, \bxs, \bxts)\xs_j\xts_k $ such that
\be
\begin{split}
R_0 		&=  R_K^2 + R_{P,0}^2 +  R_K^3 + R_{P,0}^3 + R_{P,0}^4 + \Hb_{*,0}\\
R_{(j,k)} 	&= R_{P,(j,k)}^1  + R_{P,(j,k)}^2 + R_{P,(j,k)}^3 + R_{P,(j,k)}^4 + \Hb_{*,(j,k)} \nnb
\end{split}
\ee
and
\bes
	\norm{R}_{1/3}\leqp \eps^{3\beta} .
\ees
Moreover, if we assume $\beta>1/3$, we can ensure that
\bes
	\norm{R_{(j,k)}}_{1/6} \leqp \eps^{2-2\beta} .
\ees
\end{lemma}
Remark that this last bound comes from the threshold
\bes
	\norm{\Hb_{*,(j,k)}}_{1/6} \leqp \eps^{2-2\beta}
\ees
that is obtained by application of the Cauchy inequalities.

In order to uncouple the fast and semi-fast degrees of freedom, we perform the symplectic linear transformation $\Psit(\bI, \bvphi, \bw, \bwt)= (\bZs, \bzetas, \bxs, \bxts)$ which diagonalizes the quadratic form $Q$. This leads to   the Hamiltonian $\sHt$ and its remainder $\sRt = R\circ \Psit$. The inclusions   \eqref{eq:Approx_emboitement} are ensured since $\kappa \leq 1/2$.

 \subsection{Lemma \ref{lem:semifast}: Semi-fast Frequency}

 Let us first prove the expression (\ref{eq:phi_min}) which gives the lower bound of $\varphi_1$ along a $h_\delta$-level curve.

A straightforward calculation shows that $\phicrit$ is given by the smallest positive root of the polynomial equation $4X^3 -(5+3\delta)X +1 =0$, where $X = \sin(\phicrit/2)$. It follows that $\phicrit$ is an analytic function of $\delta$  in a neighborhood of $0$, which satisfies
\bes
	\phicrit = 2 \arcsin\left(\frac{\sqrt{2} -1}{2}\right) - \frac{3(\sqrt2 -1) }{(3\sqrt2 -2)\sqrt{1+2\sqrt2}}\delta + \cO(\delta^2).
\ees

In order to prove the relations (\ref{eq:asymp_nu}), let us begin to derive an asymptotic expansion of the integral  $\cI_\delta = \displaystyle{\int}_{\phicrit}^\pi \frac{\rd\varphi}{\sqrt{U_\delta(\varphi)}}$  involved in the expression (\ref{eq:period_delta}).
$\cI_\delta$ can be splitted in three different terms:

\bes
\begin{split}
&\cI_\delta = \cI_\delta^{1} + \cI_\delta^{2} + \cI_\delta^{3} \qtext{with} \\
&\cI_\delta^{1} = \int_{\phicrit}^\frac{\pi}{3} \frac{\rd\varphi}{\sqrt{U^{\phantom{0}}_\delta(\varphi)}}, \quad
\cI_\delta^{(2)} =  \int_{\frac{\pi}{3}}^{\pi}\frac{\rd\varphi}{\sqrt{U_\delta^0(\varphi)}}  \qtext{and}\\
&\cI_\delta^{3} = \int_{\frac{\pi}{3}}^{\pi}
                  \left(
                      \frac{1}{\sqrt{U_\delta^{\phantom{0}}(\varphi)}}  - \frac{1}{\sqrt{U_\delta^0(\varphi)}}
                  \right) \rd\varphi \\
&\mbox{where} \quad U_\delta^0(\varphi) =  \delta + \frac{7}{24}(\varphi -\pi)^2 .
\end{split}
\ees

As $U_\delta(\phicrit) =0$, Taylor formula leads to
\bes
\begin{gathered}
\cI_\delta^{1} = \left( \frac{\pi}{3} -  \phicrit \right)  \int_0^1 \frac{\rd u}{\sqrt{u}\sqrt{G_\delta(u)}}\qtext{where}\\
G_\delta(u) = \int_0^1 \cF'\left(
      \phicrit + \left( \frac{\pi}{3} -  \phicrit \right)uv
    \right) \rd v .
\end{gathered}
\ees
As $G_\delta(u) > G_\delta(1)$ and $G_0(1) >1$ and if $\delta>0$ is small enough, one has
$$
  \int_0^1 \frac{\rd u}{\sqrt{u}\sqrt{G_\delta(u)}} \leqp  \int_0^1 \frac{\rd u}{\sqrt{u}} .
$$
As a consequence, $\cI_\delta^{1}$ is analytic with respect to $\delta$.

The integral expression $\cI_\delta^{2}$ can be calculated explicitly as
\bes
\cI_\delta^{2} = \sqrt{\frac{24}{7}} {\rm arcsinh}\left( \sqrt{\frac{7}{54}}\frac{\pi}{\sqrt{\delta}} \right)
=  \sqrt{\frac67} \vert \ln\delta\vert + \sI_\delta^{2}
\ees
where $\sI_\delta^{2}$ is analytic in $\delta$.

All that remains is to estimate the size of $\cI_\delta^{3}$ and of its first derivative.  First of all, $U_\delta$ being an infinitely differentiable function of $\varphi\in [\pi/3, \pi]$ satisfying the additional relations:
$$
U_\delta(\pi) = 1 + \delta + \cF(\pi) = \delta ,  \quad  \frac{\rd U_\delta}{\rd\varphi}(\pi) =  \frac{\rd^3 U_\delta}{\rd\varphi^3}(\pi)=0  \qtext{and}
  \frac{\rd^2 U_\delta}{\rd\varphi^2}(\pi)=\frac{7}{12} ,
$$
 Taylor formula leads to
$
\vert U_\delta(\varphi) - U_ \delta^0(\varphi) \vert  \leqp \vert\varphi - \pi\vert^4.
$
From the inequalities
\be
U_\delta^0(\varphi) \leqp U_\delta(\varphi) \leqp U_\delta^0(\varphi) , \quad
\delta \leq U_\delta^0(\varphi)  , \quad
  \vert \varphi - \pi\vert^2  \pleq U_\delta^0(\varphi) , \nnb
\ee
that hold for $(\delta,\varphi) \in [\delta^*,2\delta^*]\times[\pi/3,\pi]$, one can derive the following relations:
\bes
\left\vert  \frac{1}{\sqrt{U_\delta^{\phantom{0}}}}  - \frac{1}{\sqrt{U_\delta^0}}   \right\vert  \leqp 1
\qtext{and}
\left\vert  \frac{\rd^p }{\rd\delta^p}  \left( \frac{1}{\sqrt{U_\delta^{\phantom{0}}}}  - \frac{1}{\sqrt{U_\delta^0}} \right)  \right\vert  \leqp \frac{1}{\sqrt{\delta^*}^{2p-1}} .
\ees
It follows that $\cI_\delta^{3}$ is analytic on $[\delta^*,2\delta^*]$ and that its first derivative is bounded by
\be
\left\vert  \frac{\rd \cI_\delta^{3}}{\rd\delta} \right\vert  \leqp \frac{1}{\sqrt{\delta^*}} . \nnb
\ee

As a consequence
\bes
T_\delta =  \frac{2\pi}{\upsilon_0  \sqrt{\eps}K} \modu{\ln\delta}  \left[1 + g(\delta) \right]
\ees
with $\modu{g(\delta)} \leqp \vert\ln\delta^*\vert^{-1}$  and $\modu{g'(\delta)}\leqp (\delta^*)^{-1}\modu{\ln\delta^*}^{-2}$.
As
\bes
 \nu_\delta = \frac{\upsilon_0  \sqrt{\eps}K}{\vert\ln\delta\vert}  \left[1 -  g(\delta)  + \frac{ g(\delta)^2}{1+ g(\delta)}\right],
\ees
we get the expressions (\ref{eq:asymp_nu}).

\subsection{Theorem \ref{Th:largeurs}: Semi-fast Holomorphic Extension}

We consider the mechanical system
\bes
	\sHt_1(I_1, \varphi_1) = \upsilon_0\left(-AI_1^2 + \eps B \cF(\varphi_1)\right)
\ees
where $A$, $B$ are two positive constants and the real function $\cF$ is defined on $]0,2\pi[$ by \eqref{eq:Approx_Functions}.
%%%%%%%%%%%%%%%%%%%%%%%%%%%%%%%%%%%%%%%%%%
%%% Mech Sys. symétrie %%%%%%%%%%%%%%%%%%%%%%%%%%

On the domain $\fD_*$, defined as
\be
\fD_* = \left\{ \begin{array}{c}(I_1,\varphi_1) \in \RR\times]0,2\pi[ \qtext{such that}  \sHt_1(I_1,\varphi_1) =  h_\delta \\\mbox{with}\quad \delta^* \leq \delta \leq 2\delta^* \end{array}\right\}
\nnb
\ee
for some $\delta^*>0$, we can build a system of action-angle variables denoted $(J_1,\phi_1)$ such that
\bes
\sH_1(J_1)= \sHt_1\circ\fF(J_1, \phi_1) =h_\delta
\qtext{and}
\sH_1'(J_1) = \nu_\delta  .
\ees
The transformation in action-angle variables, which will be denoted $\fG$, satisfies
\be
		\fG : \quad \Bigg\{
		\begin{array}{ccc}
		  \fD_{*}	&\longrightarrow 	& \cS_*\times \TT\\
			(I_1,\varphi_1) 	&\longmapsto & (J_1, \phi_1)
	\end{array}\nnb
\ee
with $ \cS_* = \left[a, b\right]$ for some $a< 0 <b$.
We also denote $\fF ={\fG}^{-1}$ the inverse of the action-angle transformation as in (\ref{eq:AA_transf}).
%%%%%%%%%%%%%%%%%%%%%%%%%%%%%%%%%%%%%%%%%%

%%% Explicite transformation %%%%%%%%%%%%%%%%%%%%%%%%

We rewrite the Hamiltonian in a suitable form for the complex extension,
 \be
 \begin{gathered}
 	\sHt_1(I_1, \varphi_1) 	=-\eps\upsilon_0 B(1+\rh(I_1, \varphi_1))\\
	\mbox{with} \quad \rh(I_1, \varphi_1) = \frac{A}{\eps B}I_1^2-1-\cF(\varphi_1 ),\nnb
\end{gathered}
\ee
and the transformation ${\fG}$ can be defined explicitly by a classical integral formulation.
The action is given by
\be
\begin{gathered}
	J_1 = 4 \sqrt{\eps\frac{B}{A}}\left(\pi -\varphi_{(I_1,\varphi_1)}^{\min}\right)\int_{0}^{1}\sqrt{\cU^\pi_{(I_1, \varphi_1) }(x)} \rd x\\
	\mbox{with} \quad \cU^{\theta}_{(I_1,\varphi_1)}(x) = 1 + \cF\left((1-x)\varphi_{(I_1,\varphi_1)}^{\min} + x\theta \right) + \rh(I_1, \varphi_1)\nnb
\end{gathered}
\ee
where  $\varphi_{(I_1,\varphi_1)}^{\min}=\cF^{-1}(-1-\rh(I_1,\varphi_1))$ and $J_{1}$ is the action linked to an energy curve corresponding to an arbitrary shift of energy $\delta\in ]\delta^* ,2\delta^* [$.
The lower angle $\varphi_{(I_1,\varphi_1)}^{\min}$ is well defined since $\cF'(\varphi_0^{\min})\neq 0$ where $\varphi_0^{\min}=2\arcsin\left(\frac{\sqrt{2} -1}{2}\right)$ is the minimal value of the angle $\varphi_1$ along the separatrix hence $\cF(\varphi_0^{\min})=-1$, consequently $\cF^{-1}$ is analytic around $-1$.
Concerning the angle $\phi_1$, we have to consider the time of transit from  the point $( 0,\varphi_{(I_1,\varphi_1)}^{\min})$ to $(I_1,\varphi_1)$ which is given by
\be
\begin{gathered}
	\tau (I_1,\varphi_1)=\frac{2}{\sqrt{\eps AB}}\left(\varphi_1 -\varphi_{(I_1,\varphi_1)}^{\min}\right)\int_{0}^{1}
\frac{\rd x}{\sqrt{\cU^{\varphi_1}_{(I_1,\varphi_1)}(x)}}\\
\mbox{and}\quad \phi_1 =\displaystyle\frac{\pi}{2}\frac{\tau (I_1,\varphi_1)}{\tau (I_1,\pi)} .
\end{gathered}\nnb
\ee

%%%%%%%%%%%%%%%%%%%%%%%%%%%%%%%%%%%%%%%%%%

%%% Domaine d'holomorphie %%%%%%%%%%%%%%%%%%%%%%%%
Now, we look for the complex domain of holomorphy of the integrable Hamiltonian $\sH_1$.
We first consider the complex domain
\bes
D_{*,\rhoh } = \left\{ (I_1,\varphi_1 )\in\CC^2  /  \exists(I_1^*,\varphi_1^*)\in\fD_{*}    \mbox{with:}
\begin{array}{l}
\modu{I_1 - I_1^*} \leq \sqrt{\eps} \rhoh\\
\modu{\varphi_1 - \varphi_1^*} \leq \rhoh
\end{array}\right\}
 \ees
for $\rhoh>0$ and $\eps$ small enough ($\rhoh \pleq 1$, $\eps \pleq 1$).

In order to disentangle the dependance of the complex domain $D_{*,\rhoh }$ with respect to $\delta^*$ and $\eps$, we perform the following scalings:
\bes
	I_1 =\sqrt{\eps}\Ih_1 \qtext{and} J_1 =\sqrt{\eps}\Jh_1 \qtext{for} (\Ih_1,\varphi_1)\in\fDh_{*}
\ees
with
\be
	\fDh_{*} = \left\{ \begin{array}{c}(\Ih_1,\varphi_1) \in \RR\times]0,2\pi[ \qtext{such that}  \rhh(\Ih_1,\varphi_1) =  \delta \\\mbox{with}\quad \delta^* \leq \delta \leq 2\delta^* \end{array}\right\}\nnb
\ee
for the real analytic function
\bes
	\rhh(\Ih_1, \varphi_1)=\frac{A}{B}\Ih_1^2-1-\cF(\varphi_1),
\ees
and we consider the complex extension $\Dh_{*,\rhoh } = \cB_{\rhoh}\fDh_{*}$ with
\bes
\Dh_{*,\rhoh } =  \left\{ (\Ih_1,\varphi_1 )\in\CC^2  /  \exists(\Ih_1^*,\varphi_1^*)\in\fDh_{*}    \mbox{with:}
\begin{array}{l}
\modu{\Ih_1 - \Ih_1^*} \leq \rhoh\\
\modu{\varphi_1 - \varphi_1^*} \leq \rhoh
\end{array}\right\}
 \ees
where $\rhoh>0$ is small enough ($\rhoh \pleq 1$).

Likewise, we have the following real analytic functions:
\be
\Jh_1(\Ih_1, \varphi_1) = 4 \sqrt{\frac{B}{A}}\left(\pi -\varphih_{(\Ih_1,\varphi_1)}^{\min}\right)
\displaystyle\int_{0}^{1}\sqrt{\cUh^\pi_{(\Ih_1, \varphi_1)}(x)} \rd x\nnb
\ee
with
\be
\begin{gathered}
\begin{aligned}
	\cUh^{\theta}_{(\Ih_1,\varphi_1)}(x)\! &=\! 1\! +\! \cF\left((1-x)\varphih_{(\Ih_1,\varphi_1)}^{\min}\! + x\theta \right)\! + \rhh(\Ih_1, \varphi_1),\\
\varphih_{(\Ih_1,\varphi_1)}^{\min}\!\! &=\! \cF^{-1}(-1 - \rhh(\Ih_1, \varphi_1)),\nnb
\end{aligned}
\end{gathered}
\ee
and
\be
\phi_1 (\Ih_1,\varphi_1)= \displaystyle\frac{\pi}{2} \frac{\varphi_1 -\varphih_{(\Ih_1,\varphi_1)}^{\min}}
{\pi -\varphih_{(\Ih_1,\varphi_1)}^{\min}} \frac{\displaystyle\int_{0}^{1}\frac{\rd x}{\sqrt{\cUh^{\varphi_1}_{(\Ih_1, \varphi_1)}(x)}}}{\displaystyle\int_{0}^{1}\frac{\rd x}{\sqrt{\cUh^{\pi}_{(\Ih_1, \varphi_1)}(x)}}} .\nnb
\ee
Hence we consider the transformation
\be
		\fGh : \quad \Bigg\{
		\begin{array}{ccc}
		  \fDh_{*}		&\longrightarrow	& \cSh_* \times \TT\\
			(\Ih_1,\varphi_1) 	&\longmapsto& (\Jh_1, \phi_1)
		\end{array}\nnb
\ee
with
\be
\cSh_* = \left[\frac{a}{\sqrt{\eps}}, \frac{b}{\sqrt{\eps}}\right] \qtext{for some $a<0 <b$}\nnb
\ee
which corresponds to the action-angle variables for the mechanical system
\bes
		\sHh_1(\Ih_1, \varphi_1) = \upsilon_0\left(-A \Ih_1^2 + B \cF(\varphi_1)\right)
\ees
and its inverse mapping will be denoted $\fFh=\fGh^{-1}$.
Moreover, these transformations are independent of $\varepsilon$.

By classical theorem of complex analysis, $\fFh$ (resp. $\fGh$) can be extended in a unique way to a map $F$ (resp. $G$) holomorphic on a complex set
\bes
\begin{split}
\cB_r\cSh_*\times\cV_{s}\TT =
	\left\{\begin{array}{c}
		 (\Jh_1,\phi_1 )\in\CC^2\,  / \, \exists\Jh_1^*\in\cSh_* \qtext{such that}\\
		\modu{\Jh_1- \Jh_1^*} \leq r ,\quad  \re(\phi_1)\in\TT\\ \mbox{and}\quad  \modu{\im(\phi_1)}\leq s
		 \end{array}\right\}
\end{split}
 \ees
 and $G$ holomorphic over the set $\Dh_{*,\rhoh}$
for some $r>0$, $s>0$ and $\rhoh >0$ small enough.
We want to compute a lower bound on the analyticity widths $r$, $s$.
For $\rhoh >0$, we denote
\bes
 	\norm{\Jh_1}_{\rhoh}  =\sup\limits_{\Dh_{*,\rhoh }} \modu{\Jh_1( \Ih_1,\varphi_1)}, \quad
	\norm{\phi_1}_{\rhoh} = \sup\limits_{\Dh_{*,\rhoh}}\modu{\phi_1(\Ih_1,\varphi_1)},\nnb
\ees
moreover, we consider
   \bes
   \tM=\norm{\Jh_1}_{\rhoh} +\norm{\phi_1}_{\rhoh} .
   \ees
Finally, since the real mapping $\fGh$ is symplectic, it is non-degenerate at each point of the domain $\fDh_{*}$ and we denote
\bes
    \tL=\sup\limits_{\fDh_{*}}  \modu{ \rd\fGh^{-1}_{(\Ih_1,\varphi_1)}} .
\ees

By a standard application of the Lipschitz inverse function theorem  \citep[see][]{2013Ga}, we obtain the main estimate of this section.

 \begin{theorem}

Suppose that $U$ is an open subset of a Banach space $(E,\vert\vert .\vert\vert )$ and that $g : U\rightarrow E$ is a Lipschitz  mapping  with constant $K<1$.

Let $f(x)=x+g(x)$.
If the closed ball $\cB_\varepsilon \{x\}$ centered at $x\in E$ of radius $\varepsilon$ is contained in $U$, then
\bes
\cB_{(1-K )\varepsilon }\{f (x)\}\subseteq f(\cB_\varepsilon \{x\})\subseteq \cB_{(1+K )\varepsilon } \{f (x)\} .
\ees

The mapping $f$ is a homeomorphism of $U$ onto $f^{-1} (U)$, the inverse mapping $f^{-1}$ is a Lipschitz mapping with constant $(1- K)^{-1}$ and $f(U)$ is an open subset of $E$.
 \label{Lipschitz inverse function theorem}
\end{theorem}

More precisely, we use Theorem \ref{Th:HKHP} and Cauchy inequalities applies on $G$ which yields

\begin{theorem}

With the previous notations, if
\bes
	 r\peq\frac{\rhoh^2}{\tL^2 \tM}\qtext{and} s\peq\frac{\rhoh^2}{\tL^2 \tM},
\ees
then $G$ admits an inverse mapping $F$ which is holomorphic on $\cB_{r}\cSh_* \times \cV_s\TT$ and $F$ is $C$-Lipschitz with $C\eqp \tM$.

\end{theorem}

 Hence, in order to estimate the analyticity widths in action-angle variables for the considered mechanical system, we have to compute the dependance w.r.t. the quantity $\delta^*$ of the analyticity width $\rhoh$ in the original variables $(\Ih_1,\varphi_1)$, the upper bounds $\tL$ on the real domain $\fDh_{*}$ and $\tM$ on the complex domain $D_{*,\rhoh }$. In order to bound $\tL$, we use the fact that $\fGh$ is symplectic on the real domain $\fDh_{*}$, hence the coefficients of the Jacobian matrix linked to $\rd \fGh^{-1}$ are given by the derivatives of $\fGh$ that we estimate by an application of Cauchy inequalities over $D_{*,\rhoh} $. We obtain
 \bes
    	\tL\eqp (\delta^*)^{-3/2} .
 \ees

   Concerning the quantities $\rhoh$ and $\tM$ on the complex domain $D_{*,\rhoh }$, rough estimates ensure that if we choose the analyticity width $\rhoh \eqp \delta^*$ for $\delta^*$ small enough ($\delta^* \pleq 1$), we can ensure the upper bound
\bes
\tM \leqp \sqrt{\delta^*}^{-1}.
\ees

Plugging these estimates in the latter theorem ensure that $G$ admits an inverse mapping $F$ which is holomorphic on $\cB_r\cSh_*\times\cV_s\TT$ for
\bes
	0<r\pleq (\delta^{*})^{11/2} \qtext{and}
	0<s\pleq  (\delta^{*})^{11/2} .
\ees

Going back to the initial variables,  if we denote $F=(F_1,F_2)$, then the extended transformation in action-angle coordinates in the complex plane is given by
$$(\sqrt{\eps} F_1(J_1/\sqrt{\eps}, \phi_1), F_2(J_1/\sqrt{\eps}, \phi_1))$$
 and we  obtain the analyticity widths of Theorem \ref{Th:largeurs}.

Moreover, $F$ is $C$-Lipschitz with $C\eqp \sqrt{\delta^*}^{-1}$ and the distance to the real domain of the image is bounded by $\sqrt{\eps}(\delta^{*})^{\ph-1/2}$ for $I_1$ and by $(\delta^{*})^{\ph-1/2}$ for $\varphi_1$ hence these quantities are bounded by $\sqrt{\eps} (\delta^*)^5$ and $(\delta^*)^5$ for $\ph = 11/2$.

\subsection{Theorem \ref{Th:AA}: Semi-fast Action-Angle variables}
The existence of the transformation $\Psi$ is immediate by  application of Lemma \ref{lem:semifast} to the averaged Hamiltonian $\sHt$ considered in \eqref{eq:Moy1_HamF}.

Finally, the two last thresholds in \eqref{eq:AAestimates} are deduced by an application of the Cauchy inequalities.

\subsection{Theorem \ref{Th:Moy2}: Second Averaging Theorem}
In the same way as for the First Averaging Theorem, we define firstly an iterative lemma of averaging.
Let us introduce some notations:
$(\xi_k)_{k\in\{1,\ldots, 5\}}$ are given positive numbers such that
\bes
	0<\xi_j < \rho_j, \quad  0 < \xi_{2+j} <\sig_j \qtext{for $j\in\{1,2\}$,}
	0<\xi_{5}<\sqrt{\rho_2\sigma_2},
\ees
and for $0\leq r\leq 1$, we denote $\fK_r$, the domain such as
\be
	\fK_r=\cB_{\rho_1-r\xi_1}\cS_*\times\cB^1_{\rho_2-r\xi_2}\times \cV_{\sigma_1-r\xi_3}\TT\times \cV_{\sigma_2-r\xi_4}\TT \times \cB^4_{\sqrt{\rho_2\sig_2} - r\xi_{5}}\nnb .
\ee

 Moreover, we will consider $\nu_0$ a lower bound for the semi-fast frequency $\sH'_1$ on the complex domain $\cK_p$ and according to (\ref{eq:bound_nu}), we can choose
 \bes
\nu_0\peq \frac{\sqrt{\eps}}{\modu{\ln \eps}}
\ees
with our polynomial dependence of $\delta^*$ with respect to $\varepsilon$.

Hence, we set out the following:
\begin{lemma}[Second Iterative Lemma]

\label{Lem:Moy2}
Let $\brho^{-}$, $\bsig^{-}$, $(\xi_k)_{k\in\{1,\ldots,4\}}$ be fixed positive real numbers that depend on the small parameter $\eps$ and
\bes
	\begin{aligned}
	&\brho^{+} =\brho^- - (\xi_1,\xi_2),&\quad
	&\bsig^{+} = \bsig^- - (\xi_3,\xi_4), \\
	&\xi_{5}= \sqrt{\rho_2^-\sig_2^-} - \sqrt{\rho_2^{+}\sig_2^{+}} &
	&\mbox{such as  $0<\rho_j^{+}$,
	$0 < \sig_j^{+}$ for $j\in\{1,2\}$.}
	\end{aligned}	%\label{eq:LemMoy2_cond0}
\ees
Let $\sH^{-}$ be a Hamiltonian of the form
\bes
\begin{split}
	\sH^{-}(\bJ, \phi_1, \bw, \bwt) = &\sH_1(J_1) + \sH_2(J_2) + \sQb(J_1, \bw, \bwt)\\
	+&\sH_{*}^{0,-}(\bJ, \bw, \bwt)
													  + \sH_{*}^{1,-}(\bJ, \phi_1, \bw, \bwt)
\end{split}
\ees
with
$\sH_*^{l,-} = \sH_{*,0}^{l,-} + \sum_{j,k\in\{1,2\}} \sH_{*,(j,k)}^{l,-}w_j\wt_k$ for $l\in\{0,1\}$
(given by \eqref{eq:Dal_f}), which satisfies the D'Alembert rule, is analytic on the domain $\fK^{-}_0=\cK_{\brho^{-}, \bsig^{-}}$ and such that
\be
	{\sHb_{*}^{1,-}}(\bJ, \bw, \bwt) = \frac{1}{2\pi}\int_0^{2\pi} {\sH_{*}^{1,{-}}}(\bJ, \phi_1, \bw, \bwt)\rd\phi_1 =0 .\nnb
\ee
Let  $(\eta_l^{-})_{l\in\{0,2\}}$ and $(\mu_{l,m}^{-})_{l\in\{0,2\},m\in\{0,1,2\}}$ be fixed positive real numbers, which depend on $\eps$, such that :
\be
	\begin{aligned}
	&\norm{{\sH_{*,0}^{1,{-}}}}_{\fK_0^{-}}\leq\eta_0^{-} , &\quad 	
	&\norm{{\sH_{*,(j,k)}^{1,{-}}}}_{\fK_0^{-}}\leq\eta_2^{-} , &\\	
	&\norm{{\sH_{*,0}^{0,{-}}}}_{ \fK_0^{-}}\leq \mu_{0,0}^{-} ,& \quad
	&\norm{{\sH_{*,(j,k)}^{0,{-}}}}_{ \fK_0^{-}}\leq \mu_{2,0}^{-} ,&
	\label{eq:LemMoy2_Binit1}
	\end{aligned}
\ee
and
\bes
	\norm{\partial_{J_m}\sH_{*,0}^{0,{-}}}_{ \fK_0^{-}}\leq \mu_{0,m}^{-},\quad
	\norm{\partial_{J_m}\sH_{*,(j,k)}^{0,{-}}}_{ \fK_0^{-}}\leq \mu_{2,m}^{-}
	\quad(m\in\{1,2\}).
	%\label{eq:LemMoy2_Binit2}
\ees

If we assume that
\be
	\begin{aligned}
	&\frac{\eta_0^{-} + \eta_2^{-}\rho_2^{-}\sig_2^{-}}{\nu_0}\pleq  \xi_1\xi_3 , \quad
	&\frac{\eta_0^{-} + \eta_2^{-}\rho_2^{-}\sig_2^{-}}{\nu_0}\pleq  \xi_2\xi_4 , \\
	&\frac{\eta_0^{-} + \eta_2^{-}\rho_2^{-}\sig_2^{-}}{\nu_0}\pleq  (\xi_5)^2 ,
	\end{aligned}\label{eq:LemMoy2_CCond}
\ee
then there exists a canonical transformation
\begin{gather*}
	\Psib^{+} : \quad \Bigg\{
	\begin{array}{ccc}
		\fK^{-}_1                			 &\longrightarrow & \fK^{-}_0   \\
		(\bJs, \bphis, \bws, \bwts)  &\longmapsto     &(\bJ, \bphi, \bw, \bwt )
	\end{array}\nnb\\
	\mbox{with} \quad \fK^{-}_{2/3} \subseteq \Psib^{+}(\fK^{-}_{1/2}) \subseteq \fK^{-}_{1/3} %\label{eq:LemMoy1_emboitement}
\end{gather*}
and such that, in the new variables, the Hamiltonian
$\sH^{+}= \sH^{-} \circ \Psib^{+}$ satisfies the D'Alembert rule and can be written
\be
	\begin{split}
	\sH^{+}	=& \sH_1 + \sH_2 + \sQb + {\sH_{*}^{0,{-}}}  + \sH_*^{+} \\\nnb
				=& \sH_1 + \sH_2 + \sQb + {\sH_{*}^{0,{+}}}  + \sH_{*}^{1,+}
	\end{split}
	\qtext{with}
	\left\{\begin{array}{l}
		{\sH_{*}^{0,{+}}} = \sH_{*}^{0,{-}} + \sHb_{*}^{+} \\
		{\sH_{*}^{1,{+}}} = \sH_*^{+}      - \sHb_*^{+}
	\end{array} \right.
\ee
such that
$\sH_*^{l,+} = \sH_{*,0}^{l,+} + \sum_{j,k\in\{1, 2\}} \sH_{*,(j,k)}^{l,+}\ws_j\wts_k$ for $l\in\{0,1\}$
(given by \eqref{eq:Dal_f})
and
\be
	{\sHb_{*}^{+}}(\bJs, \bws, \bwts) = \frac{1}{2\pi}\int_0^{2\pi} {\sH_{*}^{{+}}}(\bJs, \phis_1, \bws, \bwts)\rd\phis_1\nnb .
\ee

Furthermore, we have the thresholds
\be
	\begin{aligned}
	&\norm{{\sH_{*,0}^{1,{+}}}}_{\fK_0^{-}}\leq\eta_0^{+} , &\quad 	
	&\norm{{\sH_{*,(j,k)}^{1,{+}}}}_{\fK_0^{-}}\leq\eta_2^{+} , &\\	
	&\norm{{\sH_{*,0}^{0,{+}}}}_{ \fK_0^{-}}\leq \mu_{0,0}^{+} ,& \quad
	&\norm{{\sH_{*,(j,k)}^{0,{+}}}}_{ \fK_0^{-}}\leq \mu_{2,0}^{+} ,&
	\label{eq:LemMoy2_Bfinal1}
	\end{aligned}
\ee
and
\be
	\norm{\partial_{J_m}\sH_{*,0}^{0,{+}}}_{ \fK_0^{-}}\leq \mu_{0,m}^{+},\quad
	\norm{\partial_{J_m}\sH_{*,(j,k)}^{0,{+}}}_{ \fK_0^{-}}\leq \mu_{2,m}^{+}\quad (m\in\{1,2\}),
	\label{eq:LemMoy2_Bfinal2}
\ee
with the following quantities:
\be
	\begin{aligned}
 	&\eta_0^+ \eqp \eta_0^-\frac{\theta_0^+}{\nu_0},& \quad 	
	&\eta_2^+ \eqp \eta_2^-\left(
		\frac{\theta_0^+}{\nu_0} +
		\frac{\eta_0^-}{\eta_2^-}\frac{\theta_1^+}{\nu_0} +
		\rho_2^-\sig_2^- \frac{\theta_1^++ \theta_2^+}{\nu_0}\right),&\\
	&\mu_{0,0}^+ - \mu_{0,0}^- \eqp \eta_0^-\frac{\gam_0^+}{\nu_0},&\quad	
	&\mu_{2,0}^+ - \mu_{2,0}^- \eqp \eta_2^-\left(
		\frac{\gam_0^+ }{\nu_0} +
		\rho_2^-\sig_2^-\frac{\gam_2^+}{\nu_0}\right),& \\ 	
	&\mu_{0,m}^+-  \mu_{0,m}^- \eqp \eta_0^-\frac{\gam_0^+}{\nu_0\xi_m},&\quad
	&\mu_{2,m}^+-  \mu_{2,m}^- \eqp \eta_2^-\left(
		\frac{\gam_0^+}{\nu_0\xi_m} +
		\rho_2^-\sig_2^-\frac{\gam_2^+}{\nu_0\xi_m}\right),&		
	\label{eq:LemMoy2_Bfinal3}
	\end{aligned}
\ee
for $m\in\{1,2\}$ and
 \be
	\begin{split}
	\gam_0^+ &= 	\frac{\eta_0^-}{\xi_1\xi_3} , \quad
	\gam_2^+=  	\eta_2^-\left(
					\frac{1}{\xi_1\xi_3} +
					\frac{1}{(\xi_5)^2}\right) ,\quad
	\theta_0^+ =	\frac{\mu_{0,1}^-}{\xi_3} + \gam_0^+ ,\\
	\theta_1^+ &=	\frac{\norm{\sQb_{(j,k)}}_{\fK^-_0}}{\xi_1\xi_3}+
					\frac{\mu_{2,1}^-}{\xi_3} ,\\
	\theta_2^+ &=	
					\frac{\norm{\sQb_{(j,k)}}_{\fK^-_0}}{(\xi_5)^2} +
					\frac{\mu_{2,0}^-}{(\xi_5)^2} +				
				\gam_2^+ .
	\end{split}
	\label{eq:LemMoy2_Bfinal4}
\ee
\end{lemma}

\begin{proof}
We define
$\Psib^+:\fK_1^-\longrightarrow \fK_0^-$ as the time-one map  of the  Hamiltonian flow generated by some auxiliary function ${\chi^+}$, i.e.
$\Psib^+=\Phi_1^{\chi^+}$
with
\bes
	{\chi^+}(\bJ,\phi_1, \bw, \bwt) = \frac{2\pi}{\sH_1'(J_1)}\int_0^1 s {\sH_{*}^{1,-}}(\bJ,\phi_1+ 2\pi s, \bw, \bwt)\rd s
\ees
such that the following properties are satisfied:
\be
\begin{gathered}
\Poi{\chi^+}{\sH_1} + \sH_*^{1,-} = 0,\\
\overline{\chi}^+(\bJ, \bw, \bwt) = \frac{1}{2\pi}\int_0^{2\pi} \chi^+(\bJ, \phi_1,\bw, \bwt) \rd\phi_1 = 0,\\
\mbox{and} \quad \chi^+ = \chi_0^+  + \sum_{1\leq j,k\leq 2}\chi^+_{(j,k)}w_j\wt_k \qtext{(given by \eqref{eq:Dal_f}).}
\end{gathered}
\label{eq:LemMoy2_prop}
\ee
Thus, for the same reason as in Lemma \ref{Lem:Moy1}, the Hamiltonian  can be written
\be
	\sH^+  = \sH^-\circ\Psib^+ = \sH_1 + \sH_2 + \sQb + \sH_{*}^{0,-} + \underbrace{{\sH_{*}^{1,-}} + \Poi{{\chi^+}}{\sH_1}}_{(*)}   +\sH^+_*\nnb
\ee
with
\be
	\sH_*^+  = 		\int_0^1
		\Poi{{\chi^+}}{\sQb+\sH_{*}^{0,-}+ s\sH_{*}^{1,-}}
	\circ \Phi_s^{\chi^+} \rd s\nnb
\ee
and $(*)$  is equal to zero  by \eqref{eq:LemMoy2_prop}.
 \medskip

Then, in order to estimate the size of the remainder $\sH_*^+$,   the thresholds \eqref{eq:LemMoy2_Binit1} provide
\be
	\norm{\chi_0^+}_{\fK_0^-} 		\leqp  \frac{\eta_0^-}{\nu_0}, \quad 	
	\norm{\chi_{(j,k)}^+}_{\fK_0^-} 	\leqp  \frac{\eta_2^-}{\nu_0},\nnb
\ee
while the Cauchy inequalities imply the following:
\be
\begin{gathered}
	\norm{\partial_{\phi_1}  \chi_0^+}_{\fK_{1/4}^-} 			
		\leqp \frac{\eta_0^-}{\nu_0\xi_3},\quad
	\norm{\partial_{\phi_1}  \chi_{(j,k)}^+}_{\fK_{1/4}^-}
		\leqp \frac{\eta_2^-}{\nu_0\xi_3},\quad
	\norm{\partial_{(\bw,\bwt)}\chi_{(j,k)}^+}_{\fK_{1/4}^-}
		\leqp \frac{\eta_2^-}{\nu_0\xi_5},\\
\qtext{and}	\norm{\partial_{J_l}  \chi_0^+}_{\fK_{1/4}^-} 			
		\leqp \frac{\eta_0^-}{\nu_0\xi_l}, \quad
	\norm{\partial_{J_l}  \chi_{(j,k)}^+}_{\fK_{1/4}^-}
		\leqp \frac{\eta_2^-}{\nu_0\xi_l} \qtext{for $l\in\{1,2\}$,} \nnb
\end{gathered}
\ee
as well as the following estimates on the Poisson brackets:
\be
\begin{gathered}
	\norm{\Poi{\chi_0^+}{{\sH_{*,0}^{0,-}}} }_{\fK_{1/4}^-}
		 \leqp \frac{\eta_0^-}{ \nu_0\xi_3}\mu_{0,1}^-, \quad
	\norm{\Poi{\chi_0^+}{{\sH_{*,0}^{1,-}}} }_{\fK_{1/4}^-}
		 \leqp  \frac{(\eta_0^-)^2}{ \nu_0\xi_1\xi_3} ,\nnb
\end{gathered}
\ee
and
\bes
\begin{gathered}
	\norm{\Poi{\chi^+}{\sQb_{(j,k)}}}_{\fK_{1/4}^-}
		 \leqp  \frac{\eta_2^-}{\nu_0}\norm{\sQb_{(j,k)}}_{\fK^-_0}\left[
				\frac{\eta_0^- }{\eta_2^-}\frac{1}{\xi_1\xi_3 } +
				\rho_2^-\sig_2^-\left(\frac{1}{\xi_1\xi_3 } + \frac{1}{(\xi_5)^2}\right)
				\right] ,\\
	\norm{\partial^2_{(\bw,\bwt)}\Poi{\chi^+}{\sH_{*}^{0,-}}}_{\fK_{1/4}^-}
		    \leqp 	\frac{\eta_2^-}{\nu_0}\left[
				\frac{\mu_{0,1}^-}{\xi_3}+
				\frac{\eta_0^-}{\eta_2^-}\frac{\mu_{2,1}^-}{\xi_3} +
				\rho_2^- \sig_2^-\left(
					\frac{\mu_{2,1}^-}{\xi_3} +
					\frac{\mu_{2,0}^-}{(\xi_5)^2}
				\right) \right] ,\\
	\norm{\partial^2_{(\bw,\bwt)}\Poi{\chi^+}{\sH_{*}^{1,-}}}_{\fK_{1/4}^-}
		    \leqp 	\frac{\eta_{2}^-}{\nu_0}\left[
				\frac{\eta_0^-}{\xi_1\xi_3} +
				\eta_{2}^-\rho_2^- \sig_2^-\left(
				\frac{1}{\xi_1\xi_3 } + \frac{1}{(\xi_5)^2}\right)\right]	
\end{gathered}
\ees
as $ 1 \leq \frac{\sqrt{\rho_2^- \sig_2^-}}{\xi_5} \leq \frac{\rho_2^- \sig_2^-}{(\xi_5)^2}$.
Consequently, the remainder of the transformation $\Psib^+$ is bounded such that
\be
	\norm{\sH^+_{*,0}}_{\fKh_{1/4}^-}
		\leqp \eta_0^-\frac{\theta_0^+}{\nu_0}, \quad 			
	\norm{\sH^+_{*,{(j,k)}}}_{\fKh_{1/4}^-}
		\leqp \eta_2^-\left(
				\frac{\theta_0^+}{\nu_0} +
				\frac{\eta_0^-}{\eta_2^-}\frac{\theta_1^+}{\nu_0} +
				\rho_2^-\sig_2^- \frac{\theta_1^+ + \theta_2^+}{\nu_0}\right)\nnb
\ee
where $\gam_0^+$, $\gam_2^+$, $\theta_0^+$,$\theta_1^+$ and $\theta_2^+$ are defined in \eqref{eq:LemMoy2_Bfinal3}.
Moreover by taking into account that $\overline{\chi}^+=0$ (given by \eqref{eq:LemMoy2_prop}), we deduce the following:
\be
	\norm{\sHb^+_{*,0}}_{\fKh_{1/4}^-}
		\leqp \eta_0^-\frac{\gam_0^+}{\nu_0},\quad
	\norm{\sHb^+_{*,(j,k)}}_{\fKh_{1/4}^-}
		\leqp \eta_2^-\left(
			\frac{\gam_0^+}{\nu_0} +
			 \rho_2^-\sig_2^-\frac{\gam_2^+}{\nu_0}
				\right) .\nnb
\ee
Hence, if we denote
$	{\sH_{*}^{0,+}}  = {\sH_{*}^{0,-}} + \sHb_{*}^+$ and
$	{\sH_{*}^{1,+}}  = \sH_*^+      - \sHb_*^+$
then the triangle inequality gives the estimates (\ref{eq:LemMoy2_Bfinal1}) and (\ref{eq:LemMoy2_Bfinal2}) (together with the Cauchy inequalities for the last).

Finally, in the same way as for Lemma \ref{Lem:Moy1}, the conditions \eqref{eq:LemMoy2_CCond} provide the estimates on the size of the transformation $\Psib^+$ which yields \eqref{eq:LemMoy2_Bfinal3} and \eqref{eq:LemMoy2_Bfinal4}.
 \end{proof}
\medskip

Now, in order to prove Theorem \ref{Th:Moy2}, one applies iteratively Lemma \ref{Lem:Moy2} to the Hamiltonian $\sH$ that can be written:
\be
	\begin{split}
	\sH(\bJ, \phi_1, \bw, \bwt) =& \sH_1(J_1) + \sH_2(J_2) + \sQb(J_1, \bw, \bwt) \\
												+&\sH^{0,0}_*(\bJ, \bw, \bwt)
												+\sH^{1,0}_*(\bJ, \phi_1,\bw, \bwt)\nnb
	\end{split}
\ee
where
\be
	\begin{gathered}
	\begin{aligned}
	\sH^{0,0}_* &= \sH^{0,0}_{*,0} + \sum_{j,k\in\{1,2\}} \sH^{0,0}_{*,(j,k)}w_j\wt_k =  \sRb,\\
	\sH^{1,0}_* &= \sH^{1,0}_{*,0} + \sum_{j,k\in\{1,2\}} \sH^{1,0}_{*,(j,k)}w_j\wt_k  = \sR-\sRb + \sQ - \sQb,\\
	\end{aligned}\\
	 \mbox{and} \quad \sRb(\bJ, \bw, \bwt) = \frac{1}{2\pi}\int_0^{2\pi} \sR(\bJ, \phi_1, \bw, \bwt) d\phi_1 ,
	\end{gathered}\nnb
\ee
with the following thresholds:
\bes
	\begin{aligned}
	&\norm{\sH^{0,0}_{*,0}}_{1} \leq \mu_{0,0}^{0}\leqp \eps^{3\beta},& \quad
	&\norm{\sH^{0,0}_{*,(j,k)}}_{1} \leq \mu_{2,0}^{0} \leqp \eps^{2-2\beta},&\\
	&\norm{\sH^{1,0}_{*,0}}_{1} \leq \eta_{0}^{0} \leqp \eps^{3\beta},& \quad
	&\norm{\sH^{1,0}_{*,(j,k)}}_{1} \leq \eta_{2}^{0} \leqp \eps .&
	\end{aligned}
\ees

Moreover, by reducing the domain of analyticity to $\cK_{5/6}$, one can apply the Cauchy inequalities  and obtain the followings:
\bes
	\begin{aligned}
	&\norm{\partial_{J_1}\sH^{0,0}_{*,0}}_{5/6} \leq \mu_{0,1}^{0} \leqp \eps^{3\beta-\frac{1}{2}-5\rpq},& \quad
	&\norm{\partial_{J_1}\sH^{0,0}_{*,(j,k)}}_{5/6} \leq \mu_{2,1}^{0}\leqp \eps^{\frac{3}{2} - 2\beta-5\rpq},&\\
	&\norm{\partial_{J_2}\sH^{0,0}_{*,0}}_{5/6} \leq \mu_{0,2}^{0}\leqp \eps^{2\beta},& \quad
	&\norm{\partial_{J_2}\sH^{0,0}_{*,(j,k)}}_{5/6} \leq \mu_{2,2}^{0}\leqp \eps^{2-3\beta}.&
	\end{aligned}
\ees

\medskip

In the same way as in the proof of Theorem \ref{Th:Moy2}, let $s$ a non-zero integer such that
$s = \Ent(\eps^{-\rpq}) + 1$
where
\be
	\rpq = \frac{3\beta-1}{15} \qtext{for}  4/9 < \beta <1/2.\nnb
\ee
We define
\bes
	(\xi_1, \xi_3) 		\peq \frac{1}{4s}(\sqrt{\eps}\eps^{5\rpq}, \eps^{\rpq}),\quad
	(\xi_2, \xi_4, \xi_5) 	= \frac{\sig_0}{4s}(\eps^\beta, 1, \eps^{\beta/2}),
\ees
as well as the sequences $\left(\brho^{j}\right)_{j\in\{0,1,\ldots,s\}}$,$\left(\bsig^{j}\right)_{j\in\{0,1,\ldots,s\}}$ with
\bes
	(\brho^{j}, \bsig^{j}) = \left(\frac{5}{6} - \frac{j}{4s}\right)(\brho,\bsig) \qtext{for} j\in\{0,\ldots,s\}
\ees
such that
$\fK^{j}_r = \cK_{\frac{5}{6}-\frac{j+r}{4s}}$ for $0\leq r\leq 1$.

Replacing  the notation $ ^-$ and $ ^+$ by $^{j-1}$ and $^{j}$ and assuming that for all $0<j\leq s$ the following conditions (associated with \eqref{eq:LemMoy2_CCond}) are fulfilled:
\bes
	\begin{split}	
	\modu{\ln \eps}\left(\eta_0^{j-1} + \eta_2^{j-1}\eps^\beta \right)&\pleq \eps^{1+8\rpq},\\
	\modu{\ln \eps}\left(\eta_0^{j} + \eta_2^{j}\eps^\beta \right)&\pleq \eps^{\frac{1}{2}+\beta+2\rpq}	,\label{eq:Moy2_condLem}	
	\end{split}
\ees
then an iterative application of Lemma \ref{Lem:Moy2} to the Hamiltonian $\sH$ provides a sequence of canonical transformations $\left(\Psib^{j}\right)_{j\in\{1,\ldots,s\}}$ such that $\sH\circ\Psib = \Psib^{1}\circ\Psib^{2}\circ\ldots\circ\Psib^{s}$ is equal to the Hamiltonian $\sHb + \sHd_*$ with
$\sF = \sH_{*}^{0,s}$ and $\sHd_* = \sH_{*}^{1,s}$.

For the same reasons as in the proof of Theorem \ref{Th:Moy1}, for all $n\in\{1,\ldots, s\}$, the sequences $\left(\eta_l^{j}\right)_{j\in\{1,\ldots,n\}}$ and $\left(\mu_{l,m}^{j}\right)_{j\in\{1,\ldots, n\}}$ must satisfy the following induction hypothesis:
\be
	\begin{aligned}
		&\eta_l^{j}\leq  \eta_l^{j-1}\exp (-1) ,&\quad
		&\mu_{l,m}^{j} - \mu_{l,m}^{j-1} \leq \frac{\mu_{l,m}^{0}}{s},&
	\end{aligned}
	\label{eq:Moy2_hyprec}
\ee
for $l\in\{0,2\}$ and $m\in\{0,1,2\}$.

For $n=1$, \eqref{eq:Moy2_condLem} is fulfilled as $4/9 <\beta < 1/2$ and $\eps\pleq 1$.
Moreover,
\be
\begin{gathered}
\gam_0^{1}     \eqp \eps^{3\beta-\frac{1}{2}-8\rpq}, \quad
\gam_2^{1}     \leqp \eps^{\frac{1}{2} - 8\rpq},\\
\theta_0^{1}	\leqp \eps^{\frac{1}{2} + \rpq}, \quad
\theta_1^{1}	\leqp \eps^{\frac{1}{2} - 8\rpq},\qtext{and}
\theta_2^{1}	\leqp \eps^{\frac{1}{2}-8\rpq}\nnb
\end{gathered}
\ee
imply that
\be
		\eta_l^{1} 	\leqp \eta_0^{0}\modu{\ln \eps}\eps^\rpq
					\leq   \eta_l^{0}\exp(-1)  \quad (l\in\{0,2\})\nnb
\ee
for $\modu{\ln \eps}\eps^\rpq \pleq \exp (-1)$ and
\be 	
	\begin{aligned}
	\mu_{0,m}^{1} - \mu_{0,m}^{0} 	&								\leqp \frac{\mu_{0,m}^{0}}{s}\eps^{3\beta -1- 10\rpq}\modu{\ln \eps}
								\leq \frac{\mu_{0,m}^{0}}{s},\\
	\mu_{2,m}^{1} - \mu_{2,m}^{0} 	&								\leqp \frac{\mu_{2,m}^{0}}{s}\eps^{5\beta-2 -10\rpq}\modu{\ln \eps}
								\leq \frac{\mu_{2,m}^{0}}{s}	
	\end{aligned}\nnb
\ee
for  $\eps \pleq 1$ ($m\in\{0,1,2\}$).

For a fixed integer $n$, the induction is immediate since the sequences $\left(\eta_l^{j}\right)_{j\in\{0,\ldots,n\}}$ are decreasing such that
$\displaystyle\frac{\eta_0^{n}}{\eta_2^{n}} \leq \frac{\eta_0^{0}}{\eta_2^{0}}$
while
$\displaystyle \mu_{l,m}^{n} \leq 2 \mu_{l,m}^{0}$.

Hence, this proves the hypothesis (\ref{eq:Moy2_hyprec}) up to $s$ and consequently that
\bes
	\begin{split}
		\eta_l^{s} \leqp \eta_l^{0}\exp(-s) \leqp\eps \exp (-\frac{1}{\eps^\rpq})
		\qtext{and} \mu_{l,m}^{s}	\leq  2\mu_{l,m}^{0}
	\end{split}
\ees
 which provide \eqref{eq:Moy2_Remav} and a part of the thresholds \eqref{eq:Moy2_bounds1} and \eqref{eq:Moy2_bounds2}.
The missing thresholds of \eqref{eq:Moy2_bounds2} are  deduced by using the Cauchy inequalities in a  restricted domain $\cK_p$ with $0<p<7/12$.

Finally, the equation \eqref{eq:Taylor0} as well as the Cauchy inequalities provide the size of the transformation $\Psib$ on $\cK_p$:
\be
	\begin{split}
	\norm{(\bws, \bwts)- (\bw, \bwt)}_{p} 	&\leq \sum_{l=1}^{s}\norm{\chi_{(j,k)}^{l}}_{{p}}\norm{(\bw, \bwt)}_{p}	 	
	 							\leqp \modu{\ln \eps}   \sqrt{\eps}\eps^{-\rpq}\norm{(\bw, \bwt)}_{p},\\
	\norm{\phis_1 - \phi_1}_{p} 		&\leq \sum_{l=1}^{s}\norm{\partial_{J_1}  {\chi^{l}}}_{{p}} \leqp \modu{\ln \eps}   \eps^{3\beta - 1 -6\rpq}\nnb,
	\end{split}
\ee
and  in the same way
\be
	\norm{\phis_2 - \phi_2}_{p} \leqp \modu{\ln \eps} \eps^{2\beta-\frac{1}{2}-\rpq},\quad	\norm{\Js_1 - J_1}_{p}  	\leqp \modu{\ln \eps}   \eps^{3\beta- \frac{1}{2}-2\rpq}.	\nnb
\ee
Remark that as  $\chi^{j}$ does not depend on $\phi_2$ for all $j\in\{1,\ldots,s\}$ then $\Js_2 = J_2$.
This yields $\cK_{5/12} \subseteq \Psib(\cK_{7/12})  \subseteq \cK_{9/12}$ for $\modu{\ln \eps}\eps^{3\beta-1-7\rpq} \pleq 1$.

\subsection{Theorem \ref{Th:normal_freq}: Secular Frequencies}
We denote by $f(J,\phi)$ a regular function on  $\RR\times\TT$ and by $\tilde f(\varphi) $  the real function satisfying the relation
$\tilde f \circ \fF_2 =  f$. Using these notations, the average of $f$ at $J_*\in\cS_*$ reads
\be
\begin{split}
\bar f(J_*)  & =  \frac{1}{2\pi}\int_0^{2\pi} f(J_*,\phi) \rd\phi  = \frac{1}{T_\deltas}\int_0^{T_\deltas}f(J_*, \nu_\deltas t) \rd t \\
 &= \frac{\nu_\deltas}{2\pi\upsilon_0\sqrt{\eps A B}}
  \left[
  \int_{\phicrit}^\pi \frac{\tilde f(\varphi)}{\sqrt{U_\deltas(\varphi)}} \rd\varphi
  +   \int_{\pi}^{2\pi -\phicrit}  \frac{\tilde f(\varphi)}{\sqrt{U_\deltas(\varphi)}} \rd\varphi
\right]  \\
&= \frac{\nu_\deltas}{2\pi\upsilon_0\sqrt{\eps A B}}
 \int_{\phicrit}^\pi
 \left[
 \frac{\tilde f(\varphi) + \tilde f(2\pi - \varphi)}{\sqrt{U_\deltas(\varphi)}}
 \right] \rd\varphi .
\end{split}
\label{eq:f_moy}
\ee
As $\cAt(2\pi - \varphi) = \cAt(\varphi)$ and $\cBt(2\pi - \varphi) = \conj(\cBt(\varphi))$, the expressions of $\cAb(J_*)$ and $\cBb(J_*)$ given by (\ref{eq:moy_AB}) follow.
%%%

The asymptotic expansions of $\cAb(J_*)$ and $\cBb(J_*)$ have now to be derived.  As $\cAt(\pi) = 7/8$, it follows from Lemma 2 and (\ref{eq:f_moy}) that
\be
 \cAb(J_*) = \frac78 + \sqrt{\frac76} {\vert\ln\deltas\vert}^{-1}(1 + \hat h_0(\deltas))
 \int_{\phicrit}^\pi
 \left[
 \frac{ \cAt(\varphi) - \cAt(\pi)}{\sqrt{U_\deltas(\varphi)}}
 \right] \rd\varphi .\nnb
\ee
The main part of the  integral involved in the previous expressions can be computed as follows:
\be
\begin{split}
\int_{\phicrit}^\pi
 \left[
 \frac{ \cAt(\varphi) - \cAt(\pi)}{\sqrt{U_\deltas(\varphi)}}
 \right] \rd\varphi
 =
 &\phantom{+}
 \int_{\phicritz}^\pi
 \left[
 \frac{ \cAt(\varphi) - \cAt(\pi)}{\sqrt{U_0(\varphi)}}
 \right] \rd\varphi
+
 \int_{\phicrit}^{\phicritz}
 \left[
 \frac{ \cAt(\varphi) - \cAt(\pi)}{\sqrt{U_\deltas(\varphi)}}
 \right] \rd\varphi
 \\
&+
 \int_{\phicritz}^\pi
 \left(
 \frac{ 1}{\sqrt{U_\deltas(\varphi)}}  -  \frac{ 1}{\sqrt{U_0(\varphi)}}
  \right)
\left( \cAt(\varphi) - \cAt(\pi)\right)
\rd\varphi .
\end{split}
\nonumber
\ee
As $\vert\phicrit -\phicritz\vert \leqp \delta^*$ and because $\vert  \cAt(\varphi) - \cAt(\pi)\vert \leqp \vert \varphi - \pi\vert^2$,  the two last integrals are respectively $\leqp \delta^*$ and $\leqp \sqrt{\delta^*}$. It turns out that
\bes
\begin{split}
 \cAb(J_*) &= \frac78 + \sqrt{\frac76}\frac{C_\cA} {\vert\ln\deltas\vert}(1 + \hat  h_0(\deltas))  \qtext{with} \\
 C_\cA &=
 \int_{\phicritz}^\pi
\left(\frac{ \cAt(\varphi) - \cAt(\pi)}{\sqrt{U_0(\varphi)}}
 \right) \rd\varphi
 \end{split}
\ees
and $\vert \hat  h_0(\deltas) \vert\leqp \vert\ln\delta^*\vert^{-1}$.

For the same reasons, we also have
\bes
\begin{split}
 \cBb(J_*) &= \frac78 + \sqrt{\frac76}\frac{C_\cB} {\vert\ln\deltas\vert}(1 + \hat h_0(\deltas))  \qtext{with} \\
 C_\cB &=
 \int_{\phicritz}^\pi
\left(\frac{ \re\left(\cBt(\varphi)\right) - \cBt(\pi)}{\sqrt{U_0(\varphi)}}
 \right) \rd\varphi
 \end{split}
\ees
where the real coefficients $ C_\cA$ and $ C_\cB$ are bounded by
\bes
-28 < C_\cA <-27 \qtext{and} 16 < C_\cB < 17.
\ees

This provides all that is needed for deriving the asymptotic expansion of the secular frequencies $\gt_{1,\deltas}$ and $\gt_{2,\deltas}$.
Indeed, these frequencies are given by $\gt_{j,\deltas}= \eps\upsilon_0\frac{m_1m_2}{m_0}\lambda_j $ where $\lam_j$ are the two roots of the polynomial
\bes
\lam^2 - \frac{m_1+m_2}{m_1m_2}\cAb(J_*) \lam - \frac{\cBb(J_*)^2-\cAb(J_*)^2}{m_1m_2} .
\ees
At this point, Theorem \ref{Th:normal_freq} is deduced from an asymptotic expansion of the $\lam_j$, from which it follows that the coefficients $c_2$ involved in (\ref{eq:main_secular}) satisfy the relations
\be
-90  < c_2 = 2(C_\cA - C_\cB) < - 86.\nnb
\ee

%%%%%%%%%%%%%%%%%%%%%%%%%%%%%%
% COORBKAMPLUS
% SEC 5 KAM
%%%%%%%%%%%%%%%%%%%%%%%%%%%%%%	

\subsection{Theorem \ref{Th:diag}: Diagonalization}

By the discussion that precedes Theorem \ref{Th:diag}, as the spectrum of \eqref{eq:quad_part} is simple, there exists a symplectic transformation $\Psic$ which is linear with respect to $\bws$, $\bwts$ and diagonalizes the quadratic form \eqref{eq:quad_part}.

In the general case the diagonalizing transformation is generated by a function which can be written
\be
\chi(\bJs,\bws, \bwts) = \sum_{j,k\in\{1,2\}} \chi_{(j,k)}(\bJs)\ws_j \wts_k\nnb
\ee
where $\chi_{(j,k)}(\bJs)$ are of order 1 over the considered domain.

 Using Cauchy inequalities to bound the derivatives of $\chi$ in order to control the variation of the angles associated with $\bJs$ under the considered transformation, we obtain the upper bounds
 \be
 \norm{\psi_1 - \phis_1}_{p,r} \leqp r^2\eps^{-\frac{1}{2}-5\rpq}\leqp \sig_1 \qtext{and}  \norm{\psi_2 - \phis_2}_{p,r} \leqp r^2 \eps^{-\beta}\leqp \sig_2\nnb
 \ee
since  $0<r\pleq \eps^{\frac{1}{4} + 3\rpq}$.

Finally, by Lemma \ref{lem:DAl} the Taylor expansion reads
\be
\begin{gathered}
\sF_{(j,k)}(\bJs, \bws, \bwts) - \sF_{(j,k)}(\bJs, \bzero, \bzero) = \int_0^1 (1-t)g_{(j,k)}''(t)\rd t \\\mbox{with}\quad g_{(j,k)}(t) = \sF_{(j,k)}(\bJs, t\bws, t\bwts)\nnb.
\end{gathered}
\ee
Together with the estimates \eqref{eq:Moy2_bounds2} of Theorem \ref{Th:Moy2}, this  provides the threshold \eqref{eq:frorm_norm0} on $\sRc =  \left(\sF_2 - \sum_{j,k\in\{1,2\}} \sF_{(j,k)}(\, \cdot  \,, \bzero, \bzero) \ws_j \wts_k\right)\circ\Xi$.

\subsection{Theorem \ref{th:lastone}: Application of a \citeauthor{1996Po} version of KAM Theory}
\label{sec:appendixKAM}
As it was specified in Section \ref{sec:KAM}, from now on, we constrain $\eps$ to be inside an interval $[\eps_0/2, \eps_0]$ for an arbitrary $\eps_0>0$.

Let us consider the frequency map linked to the Hamiltonian $\sHc$ (see Theorem \ref{Th:diag}) that is denoted $(\bom(\bGam), \bOm(\bGam))$ with
$\om_j = \sH'_j + \partial_{\Gam_j} \sF_0$ and
$\Om_j = g_j$,
and the following thresholds:
\bes
\begin{aligned}
&\norm{\om_1}_{p,r} \leqp \eps_0^{\frac{1}{2} - \rpq},&
&\norm{\om_2}_{p,r} \leqp 1,&\\
&\norm{\Om_1}_{p,r} \leqp \eps_0,&
&\norm{\Om_2}_{p,r} \leqp \frac{\eps_0}{\modu{\ln \eps_0}},&
\end{aligned}
\ees
that are deduced from $\norm{\sH'_2}_p\leqp 1$ and the bounds \eqref{eq:Moy2_bounds2} and \eqref{boundg1g2}.
Moreover, we have the following thresholds on the derivatives:
\bes
\begin{aligned}
&\norm{\partial_{\Gam_1} \omega_1 - \frac{\sE_1}{\eps_0^{\rpq}\modu{\ln \eps_0}^3}}_{p,r} \leqp \frac{1}{\eps_0^{\rpq}\modu{\ln \eps_0}^4},&\quad
&\norm{\partial_{\Gam_2} \omega_1 }_{p,r} \leqp \eps_0^{\beta-\frac{1}{6}},&\\
&\norm{\partial_{\Gam_1} \omega_2}_{p,r} \leqp \eps_0^{\beta-\frac{1}{6}},&\quad
&\norm{\partial_{\Gam_2} \omega_2 - \sE_2}_{p,r} \leqp \eps_0^\beta&
\end{aligned}
\ees
with $\sE_1= \upsilon_0K^2B^{-1}$ and $\sE_2 = -2E\upsilon_0$ (that are not equal to zero)  from the bounds \eqref{eq:Moy2_bounds2} and the mean value theorem.
Consequently the eigenvalues of $\rd\bom$ are small perturbations of $\displaystyle\frac{\sE_1}{\eps_0^{\rpq}\modu{\ln{\eps_0}}^4}$ and $\sE_2$.
We also ensure that $\rd\bom$ is inversible with the eigenvalues $\lam_1(\bGam)$, $\lam_2(\bGam)$ such that
\bes
\norm{\lam_1 - \frac{\eps_0^{\rpq}\modu{\ln\eps_0}^4}{\sE_1}}_{p,r} \leqp \eps_0^{\rpq} \modu{\ln \eps_0}^2, \quad \norm{\lam_2 - \frac{1}{\sE_2}}_{p,r}\leqp \eps_0^\beta.
\ees
Hence, $\bom$ is a local diffeomorphism.

In order to apply \citeauthor{1996Po} version of KAM theory for the persistence of lower dimensional normally elliptic invariant tori \citep{1996Po}, we must consider a domain where the internal frequency map $\bom$ is a diffeomorphism.
Hence, we set out the following
\begin{lemma}
\label{prop1}
 For $\eps\in [\eps_0/2,\eps_0]$, the internal frequency map $\bom$ is a diffeomorphism from $\Pi = \cB^2_{\rho}$ onto its image provided by
 \bes
 \rho \peq \frac{\eps^{1 + 9\rpq}}{\modu{\ln \eps}^3}  \qtext{with} 4/9 <\beta< 1/2 .
\ees
Moreover, we have the upper bounds
\be
\begin{split}
&\norm{\bom}_{\Pi} \leqp 1, \quad \norm{\rd \bom}_{\Pi} \leqp \frac{1}{\eps_0^{\rpq}\modu{ \ln\eps_0}^{3}}, \quad
\norm{\bom^{-1}}_{\Pi} \leqp \eps_0^\beta, \\
& \norm{\rd \bom^{-1}}_{\bom (\Pi )} \leqp 1,   \qtext{and}  \norm{\rd\bOm}_\Pi \leqp \eps_0^{1/3}.
\end{split}
\label{borne_pi}
\ee
\end{lemma}

\begin{proof}
We consider $\bom_0=\bom - \bom(\bzero)$ where
$\bom_0$ is holomorphic on the closed ball $\cB_{\rho_1}^2$ with $\rho_1 \peq \eps_0^{\frac{1}{2} + 5\rpq}$.
Then, we define
\bes
\begin{split}
 \bomt_0=& (\rd\bom(\bzero))^{-1}\bom_0 - \Id \qtext{such that}
\rd \bomt_0 = (\rd\bom(\bzero))^{-1}(\rd\bom - \rd\bom(\bzero)).
\end{split}
\ees
 Hence,
$ \displaystyle\norm{\rd\bomt_0}_{\cB^2_{\rho_1}} \leqp \sE_2^{-1} \norm{\rd \bom - \rd \bom(\bzero)}_{\cB^2_{\rho_1}}
$
since the highest eigenvalue of $(\rd \bom(\bzero))^{-1}$ satisfies $\modu{\lam_2(\bzero)}\leqp \sE_2^{-1}$ for $\eps$ small enough with $\sE_2 \neq 0$.
Furthermore, by the mean value theorem as well as the Cauchy inequalities, we can ensure that on the closed ball $\cB^2_{\rhot}$ such that $2\rhot \peq \rho_1^2= \eps^{1 + 10\rpq}$ then
\bes
 	\norm{\rd \bomt_0}_{\cB^2_{\rhot}} \leqp \norm{\rd^2\bom}_{\cB^2_{\rho_1/4}}\rhot \leqp \frac{\rhot}{\rho_1^2} \leq \frac{1}{2}.
\ees
Consequently, the application:
\bes
\bomh_0 = \Id + \bomt_0
\ees
is a diffeomorphism from $\cB^2_{\rhot}$ to $\bomt_0(\cB^2_{\rhot})$ by the fixed point theorem.
Moreover, $\bomh(\bzero) = \bzero$ yields
\bes
 \cB^2_{\rhot/2} \subset \bomh_0(\cB_{\rhot}^2) \subset \cB^2_{3\rhot/2}
\ees
and $\bomh_0^{-1}$ is a Lipschitz mapping with a constant $2$.

Now, as
$\bom = \bom(\bzero) + \rd\bom(\bzero) \bomh_0$, we consider
$$\bom^{-1}(\by) = \bomh_0^{-1}\left((\rd\bom(\bzero))^{-1}(\by - \bom(\bzero))\right) .$$
If $(\rd\bom(\bzero))^{-1}(\by - \bom(\bzero))\in\cB^2_{\rhot/2}$, then there exists
\bes
\rhoh \peq \frac{\rhot}{2} \qtext{such that}
\norm{\by - \bom(\bzero)}_{\cB^2_{\rhoh}} \leq \rhoh
\ees
(as $\norm{\rd \bom(\bzero)^{-1}}_{\cB^2_{\rhot/2}} \leqp 1$).
Hence, we have determined $\bom^{-1}$ over $\cB_{\rhoh}\{\bom(\bzero)\}$.

Finally for $(\by, \by') \in (\cB_{\rhoh}\{\bom(\bzero)\})^2$, we have
\bes
\begin{split}
\norm{\bom^{-1}(\by) - \bom^{-1}(\by')}_{\cB^2_{\rhoh}} &\leqp \norm{\by - \by'}_{\cB^2_{\rhoh}}
\end{split}
\ees
as $\rd\bomh_0^{-1}$ is 2-Lispshitz and $\norm{\rd \bom(\bzero)^{-1}}_{\cB^2_{\rhoh}} \leqp 1$.
Hence, as
$\displaystyle\norm{\rd \bom}_{\rhoh} \leqp \frac{1}{\eps_0^{\rpq}\modu{\ln\eps_0}^3}$ then
\bes
\bom(\cB^2_\rho) \subset \cB_{\rhoh}\{\bom(\bzero)\} \qtext{for} \rho \eqp \frac{\eps_0^{1+9\rpq}}{\modu{\ln\eps_0}^3}.
\ees

Consequently, $\bom$ is a diffeomorphism from $\cB^2_\rho$ onto its image and the estimates \eqref{borne_pi}
are ensured (by the estimates \eqref{eq:Moy2_bounds2}).
\end{proof}

\medskip

By the notations of Section \ref{sec:KAM}, with  $\vert f \vert_\Pi^{\rm Lip}\leq\vert\vert \rd f\vert\vert_\Pi$ for a differentiable function and the upper bounds (\ref{borne_pi}) ensure
\bes
\vert \bom \vert_\Pi^{\rm Lip}  + \vert \bOm \vert_\Pi^{\rm Lip} \leq  M \eqp \frac{1}{\eps_0^{\rpq}\modu{ \ln\eps_0}^3}   ,
\quad \vert \bom^{-1} \vert_\Pi^{\rm Lip} \leq L \eqp 1.
\ees

A property needed to apply the \citeauthor{1996Po} results on the persistence of normally elliptic tori is to ensure Melnikov's condition for multi-integers of length bounded by $K_0 = 16L M$. This is the content of the following

\begin{proposition}
\label{prop2}
Let
\bes
K_0  \eqp \frac{1}{ \eps_0^{\rpq}\modu{ \ln\eps_0}^3} \qtext{and} \gamma_0 \peq \frac{\eps_0}{\vert\ln\eps_0\vert},
\ees
we have, for $\eps\in [\eps_0/2,\eps_0]$ with $\eps_0 \pleq 1$,
\bes
\begin{gathered}
\min\limits_{\bxi\in\Pi} \big\{ \modu{\Om_1(\bxi)},  \modu{\Om_2(\bxi)},     \modu{\Om_1(\bxi) - \Om_2(\bxi)}  \big\} \geq \gamma_0 \qtext{and} \\
\min\limits_{\bxi\in\Pi}  \modu{ \bom(\bxi)\bigcdot \bk + \bOm(\bxi)\bigcdot \bl} \geq \gamma_0  \quad
\forall   0 <\vert \bk\vert  \leq K_0, \; \vert \bl\vert \leq 2.
\end{gathered}
\ees
\label{prop2_appendix}
\end{proposition}

\begin{proof}
First of all, for $\bxi \in \Pi$ we have the followings:
\be
\begin{gathered}
\frac{\sqrt{\eps_0}}{\modu{\ln \eps_0}} \leqp \modu{\om_1(\bxi)} \leqp \frac{\sqrt{\eps_0}}{\modu{\ln \eps_0}}, \quad
1 \leqp \modu{\om_2(\bxi)} \leqp 1,\\
\eps \leqp \modu{\Om_1(\bxi)} \leqp \eps_0, \quad
\frac{\eps}{\modu{\ln\eps_0}} \leqp \modu{\Om_2(\bxi)} \leqp \frac{\eps_0}{\modu{\ln \eps_0}}
\end{gathered}\label{eq:borne_freq}
\ee
that are deduced from (\ref{eq:bound_nu}) and (\ref{boundg1g2}).
As a consequence, with $\eps \in [\eps_0/2, \eps_0]$, for $\bxi\in \Pi$:
\bes
\vert\Om_1(\bxi)\vert\geq \gam_0,\quad
\vert\Om_2(\bxi)\vert\geq \gam_0,\quad
\vert\Om_1(\bxi) - \Om_2(\bxi)\vert \geq \gam_0.
\ees

For $(\bk,\bl)\in\ZZ^2\times\ZZ^2$ with $0<\modu{\bk} \leq K_0$ and $\modu{\bl} \leq 2$ we have
\bes
\norm{k_1 \om_1 + \bl\bigcdot \bOm}_\Pi \leqp K_0 \norm{\om_1}_\Pi \leqp \eps_0^{2/5}
\ees
since $4/9 < \beta< 1/2$.
Especially, for a large enough constant $C>0$, we have
\bes
\modu{\bk \bigcdot \bom(\bxi) + \bl\bigcdot  \bOm(\bxi)} \geq \modu{k_2} \modu{\om_2(\bxi)} -C\eps_0^{2/5} \geq \modu{\om_2(\bxi)} - C\eps^{2/5}\geq \gam_0
\ees
deduced from \eqref{eq:borne_freq} with $\eps_0 \pleq 1$ and $k_2 \neq 0$.
Likewise, if $k_2=0$ then for a large enough constant $C>0$, we have
\bes
\begin{split}
 \modu{k_1\omega_1 + \bl \bigcdot \bOm} \geq
		\modu{k_1}\modu{\om_1} - C \eps_0 \geq \modu{\om_1} - C\eps_0 \geq \gam_0
\end{split}
\ees
with $\eps_0 \pleq 1$.
\end{proof}

The Hamiltonian $\rH$ defined in \eqref{eq:planHam} is analytic over the domain $D(\rb,\sb)$ defined in \eqref{eq:domain_Dbar} with $0<\rb < r$ and $\sb \peq \eps_0^{\rpq}$.

With the estimates given in Proposition \ref{prop2_appendix}, it remains to check the thresholds of the Proposition 2.2 in \cite{BiCheVa2003} which become here the threshold \eqref{Dernier_Seuil} of Theorem \ref{th:lastone} and has to be satisfied for a small enough bound $\eps_0$ on the mass ratio.

 We decompose the perturbation (\ref{eq:perturbation}) in $\rP=\rP_1 + \rP_2 + \rP_3 + \rP_4$ with
 \bes
 \begin{split}
 \rP_1(\by,\bz, \bzt; \bxi) &=\sum_{j\in\{1,2\}} \Big(\sH_j(\xi_j +y_j)-\sH_j(\xi_j)-  \omega_j (\bxi )y_j \Big)+ \sF_0(\bxi +\by )-\sF_0(\bxi ),\\
 \rP_2(\by,\bz, \bzt; \bxi)&= \sum_{j\in\{1,2\}} i \big(g_j(\bxi +\by )- g_j(\bxi )\big) z_j\zt_j,\\
 \rP_3(\by,\bz, \bzt; \bxi)&=\sRc(\bxi +\by,\bz, \bzt), \qtext{and} \rP_4(\by,\bpsi,\bz, \bzt; \bxi)=\sHc_{*}(\bxi +\by,\bpsi,\bz, \bzt).
 \end{split}
\ees

 With the estimates of Theorem \ref{Th:Moy2} together with Taylor formula, since $\rP_1$, $\rP_2$ (resp. $\rP_3$) are of order 2 in $y_i$, $z_j\zt_j$ (resp. of order 4 in $z_j$, $\zt_j$). Likewise, with the corollary \ref{cor1}, $\bpsi$ appears only in $\rP_4$ which is exponentially small.
 As a consequence, we obtain for $\eps\in [\eps_0/2,\eps_0 ]$ that
\bes
\begin{split}
\norm{X_\rP}_{\rb, D(\rb,\sb)} 		&\leqp \frac{\rb^2}{\eps_0^{p}}+\frac{\eps_0^{p'}}{\rb^2}\exp(-\frac{1}{\eps_0^{\alpha}}) \\\mbox{and}\quad
\norm{X_\rP}^{\rm Lip}_{\rb, D(\rb,\sb)}	&\leqp \frac{\rb^2}{\eps_0^{p}}+\frac{\eps_0^{p'}}{\rb^2}\exp(-\frac{1}{\eps_0^{\alpha}}),
\end{split}
\ees
hence
$$\epsilon =\norm{X_\rP}_{\rb, D(\rb,\sb)}   +  \frac{\gamma}{M\gamma_0}\norm{X_\rP}_{\rb, D(\rb,\sb)}^{\rm Lip}\leqp \frac{\rb^2}{\eps_0^{p}}+\frac{\eps_0^{p'}}{\rb^2}\exp(-\frac{1}{\eps_0^{\alpha}})$$
for some positive exponents $p$ and $p'$ (remark that $p=p'$ can be chosen).
 We need
 $$
 \frac{\rb^2}{\eps_0^{p}}+\frac{\eps_0^{p}}{\rb^2}\exp(-\frac{1}{\eps_0^{\alpha}})\pleq \frac{c\gamma}{L^aM^a} \sigma_1^b
 $$
and we choose $\rb = r_0\eps_0^{d}$ for a small enough constant $r_0>0$ and a large enough exponent $d$ which ensure
$$
 r_0^2\eps_0^{2d-p}\pleq \frac{c\gamma}{2L^aM^a} \sigma_1^b.
 $$
Then,
$$\frac{\eps_0^{p'-2d}}{r_0^2}\exp(-\frac{1}{\eps_0^{\alpha}})\pleq \frac{c\gamma}{2L^aM^a} \sigma_1^b
 $$
is ensured for small enough $\eps_0 <\eps_*$ and the main threshold (\ref{Dernier_Seuil}) is satisfied.
Hence, we can find quasi-periodic horseshoe orbits for mass ratio $0<\eps <\eps_*$.

\bigskip
\bigskip

\begin{center}
\textsc{Acknowledgment}
\end{center}
{\it
The authors are indebted to Jacques F\'ejoz for key discussions concerning KAM theory.

A.P. acknowledges the support of the H2020-ERC project 677793 StableChaoticPlanetM and this research is part
of this project. L.N. acknowledges the support of the ANR project BEKAM (ANR-15-CE40-0001) and the NSF-Grant No. DMS-1440140 as well as the MSRI-Berkeley where he was in residence.
}

%%%%%%%%%%%%%%%%%%%%
% MAJ 2018-03-19

\bibliographystyle{apalike}
%\bibliography{biblio}

\begin{thebibliography}{}

\bibitem[{Aksnes}, 1985]{1985Ak}
{Aksnes}, K. (1985).
\newblock {The tiny satellites of Jupiter and Saturn and their interactions
  with the rings}.
\newblock In {Szebehely}, V.~G., editor, {\em NATO (ASI) Series C}, volume 154
  of {\em NATO (ASI) Series C}, pages 3--16.

\bibitem[{Arnol'd}, 1963]{1963Ar}
{Arnol'd}, V.~I. (1963).
\newblock {Small Denominators and Problems of Stability of Motion in Classical
  and Celestial Mechanics}.
\newblock {\em Russian Mathematical Surveys}, 18:85--191.

\bibitem[{Bengochea} et~al., 2013]{2013BeFaPe}
{Bengochea}, A., {Falconi}, M., and {P\'erez-Chavela}, E. (2013).
\newblock {Horseshoe periodic orbits with one symmetry in the general planar
  three-body problem.}
\newblock {\em {Discrete Contin. Dyn. Syst.}}, 33(3):987--1008.

\bibitem[Biasco and Chierchia, 2015]{BiChe2015}
Biasco, L. and Chierchia, L. (2015).
\newblock On the measure of {L}agrangian invariant tori in nearly-integrable
  mechanical systems.
\newblock {\em Atti Accad. Naz. Lincei Rend. Lincei Mat. Appl.},
  26(4):423--432.

\bibitem[{Biasco} and {Chierchia}, 2017]{BiChe2017}
{Biasco}, L. and {Chierchia}, L. (2017).
\newblock {KAM Theory for secondary tori}.
\newblock {\em arXiv e-prints}, arXiv:1908.10395.

\bibitem[Biasco et~al., 2003]{BiCheVa2003}
Biasco, L., Chierchia, L., and Valdinoci, E. (2003).
\newblock Elliptic two-dimensional invariant tori for the planetary three-body
  problem.
\newblock {\em Arch. Rational Mech. Anal.}, 170:91--135.

\bibitem[{Brown}, 1911]{1911Br}
{Brown}, E.~W. (1911).
\newblock {Orbits, Periodic, On a new family of periodic orbits in the problem
  of three bodies}.
\newblock {\em MNRAS}, 71:438--454.

\bibitem[Chenciner and Llibre, 1988]{MR967629}
Chenciner, A. and Llibre, J. (1988).
\newblock A note on the existence of invariant punctured tori in the planar
  circular restricted three-body problem.
\newblock {\em Ergodic Theory Dynam. Systems}, 8$^*$(Charles Conley Memorial
  Issue):63--72.

\bibitem[{Chierchia} and {Pinzari}, 2011]{ChPi2011}
{Chierchia}, L. and {Pinzari}, G. (2011).
\newblock The planetary n-body problem: symplectic foliation, reductions and
  invariant tori.
\newblock {\em Invent math}, 186: 1--77.

\bibitem[{Cors} and {Hall}, 2003]{2003CoHa}
{Cors}, J.~M. and {Hall}, G.~R. (2003).
\newblock {Coorbital Periodic Orbits in the Three Body Problem}.
\newblock {\em SIAM J. Appl. Dyn. Syst.}, 2:219--237.

\bibitem[Cors et~al., 2019]{2019CoPaYa}
Cors, J.~M., Palaci\'{a}n, J.~F., and Yanguas, P. (2019).
\newblock On co-orbital quasi-periodic motion in the three-body problem.
\newblock {\em SIAM J. Appl. Dyn. Syst.}, 18(1):334--353.

\bibitem[Delshams and Guti\'{e}rrez, 1996]{MR1419016}
Delshams, A. and Guti\'{e}rrez, P. (1996).
\newblock Estimates on invariant tori near an elliptic equilibrium point of a
  {H}amiltonian system.
\newblock {\em J. Differential Equations}, 131(2):277--303.

\bibitem[{Dermott} and {Murray}, 1981a]{1981DeMu}
{Dermott}, S.~F. and {Murray}, C.~D. (1981a).
\newblock {The dynamics of tadpole and horseshoe orbits. I - Theory.}
\newblock {\em Icarus}, 48:1--11.

\bibitem[{Dermott} and {Murray}, 1981b]{1981DeMua}
{Dermott}, S.~F. and {Murray}, C.~D. (1981b).
\newblock {The dynamics of tadpole and horseshoe orbits II. The coorbital
  satellites of Saturn}.
\newblock {\em Icarus}, 48:12--22.

\bibitem[{Eliasson}, 1988]{1988El}
{Eliasson}, L. (1988).
\newblock Perturbations of stable invariant tori for Hamiltonian systems.
\newblock {\em Annali della Scuola Normale Superiore di Pisa, Classe di
  Scienze}, 15:115--147.

\bibitem[{F\'ejoz}, 2004]{2004Fe}
{F\'ejoz}, J. (2004).
\newblock {D{\'e}monstration du th{\'e}or{\`e}me d'Arnold' sur la stabilit{\'e}
  du syst{\`e}me plan{\'e}taire (d'apr{\`e}s Herman)}.
\newblock {\em Ergodic Theory Dynam. Systems}, 24(5):1521-- 1582.

\bibitem[{Garfinkel}, 1977]{1977Ga}
{Garfinkel}, B. (1977).
\newblock {Theory of the Trojan asteroids. I}.
\newblock {\em Astronomical Journal}, 82:368--379.

\bibitem[{Garling}, 2013]{2013Ga}
{Garling}, D. (2013).
\newblock {\em {A Course in Mathematical analysis, Volume 2, Metric and
  Topological Spaces, Function of a Vector Variable}}.
\newblock {Cambridge University Press}.

\bibitem[Gascheau, 1843]{1843Ga}
Gascheau, G. (1843).
\newblock Examen d'une classe d'{\'e}quations diff{\'e}rentielles et
  application {\`a} un cas particulier du probl{\`e}me des trois corps.
\newblock {\em C. R. Acad. Sci. Paris}, 16(7):393--394.

\bibitem[Giorgilli et~al., 1989]{MR980547}
Giorgilli, A., Delshams, A., Fontich, E., Galgani, L., and Sim\'{o}, C. (1989).
\newblock Effective stability for a {H}amiltonian system near an elliptic
  equilibrium point, with an application to the restricted three-body problem.
\newblock {\em J. Differential Equations}, 77(1):167--198.

\bibitem[{Kuksin}, 1988]{1988Ku}
{Kuksin}, S.~B. (1988).
\newblock Perturbation theory of conditionally periodic solutions of
  infinite-dimensional hamiltonian systems and its applications to the
  korteweg-de vries equation.
\newblock {\em Math. Sb.}, 136(178):396--412.

\bibitem[{Laskar} and {Robutel}, 1995]{1995LaRo}
{Laskar}, J. and {Robutel}, P. (1995).
\newblock {Stability of the Planetary Three-Body Problem. I. Expansion of the
  Planetary Hamiltonian}.
\newblock {\em Celest. Mech. Dyn. Astron.}, 62:193--217.

\bibitem[{Lei}, 2015]{Lei2015}
{Lei}, Z. (2015).
\newblock Quasi-periodic almost-collision orbits in the spatial three-body
  problem.
\newblock {\em Com. in pure and applied math.}, 68(12):2144--2176.

\bibitem[{Leontovich}, 1962]{1962Le}
{Leontovich}, A.~M. (1962).
\newblock On the stability of the lagrange periodic solution for the reduced
  problem of three bodies.
\newblock {\em Dokl. Akad. Nauk SSSR}, 143(3):525--528.

\bibitem[{Llibre} and {Oll{\'e}}, 2001]{2001LlOl}
{Llibre}, J. and {Oll{\'e}}, M. (2001).
\newblock {The motion of Saturn coorbital satellites in the restricted
  three-body problem}.
\newblock {\em Astronomy and Astrophysics}, 378:1087--1099.

\bibitem[Medvedev et~al., 2015]{2015MeNeTr}
Medvedev, A.~G., Neishtadt, A.~I., and Treschev, D.~V. (2015).
\newblock Lagrangian tori near resonances of near-integrable hamiltonian
  systems.
\newblock {\em Nonlinearity}, 28(7):2105--2130.

\bibitem[{Melnikov}, 1965]{1965Me}
{Melnikov}, V.~K. (1965).
\newblock On certain cases of conservation of almost periodic motions with a
  small change of the hamiltonian function.
\newblock {\em Dokl. Akad. Nauk SSSR}, 165:1245--1248.

\bibitem[{P{\"o}schel}, 1996]{1996Po}
{P{\"o}schel}, J. (1996).
\newblock {A KAM-theorem for some nonlinear partial differential equations}.
\newblock {\em Annali della Scuola Normale Superiore di Pisa, Classe di
  Scienze}, 23:119--148.

\bibitem[{Rabe}, 1961]{1961Ra}
{Rabe}, E. (1961).
\newblock {Determination and survey of periodic Trojan orbits in the restricted
  problem of three bodies}.
\newblock {\em Astronomical Journal}, 66:500.

\bibitem[Robutel, 1995]{Ro1995}
Robutel, P. (1995).
\newblock Stability of the planetary three-body problem {II}: KAM theory and
  existence of quasiperiodic motions.
\newblock {\em Celest. Mech. Dyn. Astron.}, 62:219--261.

\bibitem[{Robutel} and {Gabern}, 2006]{2006RoGa}
{Robutel}, P. and {Gabern}, F. (2006).
\newblock {The resonant structure of Jupiter's Trojan asteroids -I.}
\newblock {\em MNRAS}, 372:1463--1482.

\bibitem[{Robutel} et~al., 2016]{2016RoNiPo}
{Robutel}, P., {Niederman}, L., and {Pousse}, A. (2016).
\newblock Rigorous treatment of the averaging process for co-orbital motions in
  the planetary problem.
\newblock {\em Computational and Applied Mathematics}, 35(3):675--699.

\bibitem[{Robutel} and {Pousse}, 2013]{2013RoPo}
{Robutel}, P. and {Pousse}, A. (2013).
\newblock {On the co-orbital motion of two planets in quasi-circular orbits}.
\newblock {\em Celest. Mech. Dyn. Astron.}, 117:17--40.

\bibitem[{Robutel} et~al., 2011]{RoRaCa2011}
{Robutel}, P., {Rambaux}, N., and {Castillo-Rogez}, J. (2011).
\newblock {Analytical description of physical librations of saturnian coorbital
  satellites Janus and Epimetheus}.
\newblock {\em Icarus}, 211:758--769.

\bibitem[{Spirig} and {Waldvogel}, 1985]{1985SpWa}
{Spirig}, F. and {Waldvogel}, J. (1985).
\newblock {The three-body problem with two small masses - A
  singular-perturbation approach to the problem of Saturn's coorbiting
  satellites}.
\newblock In {Szebehely}, V.~G., editor, {\em NATO (ASI) Series C}, volume 154
  of {\em NATO (ASI) Series C}, pages 53--63.

\bibitem[{Yoder} et~al., 1983]{1983YoCoSy}
{Yoder}, C.~F., {Colombo}, G., {Synnott}, S.~P., and {Yoder}, K.~A. (1983).
\newblock {Theory of motion of Saturn's coorbiting satellites}.
\newblock {\em Icarus}, 53:431--443.

\end{thebibliography}

\end{document}